\documentclass[a4paper,reqno]{amsart}

\makeatletter
\@namedef{subjclassname@2020}{%
  \textup{2020} Mathematics Subject Classification}
\makeatother

\usepackage[toc]{appendix}
\usepackage{amsmath}
\usepackage{amssymb}
\usepackage{amsthm}
\usepackage{mathrsfs}
\usepackage{latexsym}
\usepackage{amscd}
\usepackage{xypic}
\xyoption{curve}
\usepackage{ifthen} 
\usepackage{hyperref} 
\usepackage{graphicx}
\usepackage{enumerate}
\usepackage{enumitem}
\usepackage{rotating}
\usepackage{xcolor}
\usepackage{mathtools}
\usepackage{todonotes}
\usepackage{color,soul}

\usepackage{float}
\restylefloat{table}

\numberwithin{equation}{section}

\def\cocoa{{\hbox{\rm C\kern-.13em o\kern-.07em C\kern-.13em o\kern-.15em A}}}

\newtheorem{theorem}{Theorem}[section]

\newtheorem{lemma}[theorem]{Lemma}
\newtheorem{proposition}[theorem]{Proposition}
\newtheorem{corollary}[theorem]{Corollary}

\theoremstyle{definition}
\newtheorem{remark}[theorem]{Remark}
\newtheorem{definition}[theorem]{Definition}
\newtheorem{example}[theorem]{Example}

\newtheorem{construction}[theorem]{Construction}

\newcommand {\spec}{\mathrm{spec}}

\newcommand {\reg}{\mathrm{reg}}
\newcommand {\sHom}{\mathcal{H}\kern -0.25ex{\mathit om}}
\newcommand {\sExt}{\mathcal{E}\kern -0.25ex{\mathit xt}}
\newcommand {\sTor}{\mathcal{T}\kern -0.25ex{\mathit or}}

\newcommand {\im}{\mathrm{im}}
\newcommand {\rk}{\mathrm{rk}}

\newcommand {\Ext}{\mathrm{Ext}}

\newcommand {\Hilb}{\mathcal{H}\kern -0.25ex{\mathit ilb\/}}

\newcommand {\quantum}{k}
\newcommand {\defect}{\delta}
\newcommand {\field}{\mathbf k}

\newcommand {\cA}{\mathcal{A}}
\newcommand {\cB}{\mathcal{B}}

\newcommand {\cU}{\mathcal{U}}
\newcommand {\cV}{\mathcal{V}}

\newcommand{\cC}{{\mathcal C}}
\newcommand{\cS}{{\mathcal S}}
\newcommand{\cE}{{\mathcal E}}
\newcommand{\cF}{{\mathcal F}}
\newcommand{\cM}{{\mathcal M}}
\newcommand{\cN}{{\mathcal N}}
\newcommand{\cO}{{\mathcal O}}
\newcommand{\cG}{{\mathcal G}}
\newcommand{\cT}{{\mathcal T}}
\newcommand{\cI}{{\mathcal I}}

\newcommand {\bZ}{\mathbb{Z}}

\newcommand {\bC}{\mathbb{C}}
\newcommand {\bP}{\mathbb{P}}

\newcommand{\Pic}{\operatorname{Pic}}

\def\p#1{{\bP^{#1}}}

\def\ga#1{{{\accent"12 #1}}}

\def\mapright#1{\mathbin{\smash{\mathop{\longrightarrow}
\limits^{#1}}}}

\title[Instanton sheaves on projective schemes]{Instanton sheaves on projective schemes}

\thanks{The  authors are members of GNSAGA group of INdAM and are supported by the framework of the MIUR grant Dipartimenti di Eccellenza 2018-2022 (E11G18000350001). The first author is a postdoctoral research fellow (``titolare di Assegno di Ricerca") of Istituto Nazionale di Alta Matematica (INdAM, Italy). }



\author[V. Antonelli, G. Casnati]{Vincenzo Antonelli, Gianfranco Casnati}

\begin{document}

\maketitle

\begin{abstract}
A $h$--instanton sheaf on a closed subscheme $X$ of some projective space endowed with an ample and globally generated line bundle $\cO_X(h)$ is a coherent sheaf whose cohomology table has a certain prescribed shape. In this paper we deal with $h$--instanton sheaves relating them to Ulrich sheaves. Moreover, we study $h$--instanton sheaves on smooth curves and surfaces, cyclic $n$--folds, Fano $3$--folds and scrolls over arbitrary smooth curves. We also deal with a family of monads associated to $h$--instanton bundles on varieties satisfying some mild extra technical conditions.
\end{abstract}

\section{Introduction}
The study of coherent sheaves is a fruitful research field in algebraic geometry. In particular, it is interesting to deal with the existence of  coherent sheaves $\cE$ satisfying some extra conditions on a fixed {\sl projective scheme $X$}, i.e. a closed subscheme of some projective space over an algebraically closed field $\field$. 

For instance, if $X$ is irreducible of dimension $n$ and it is endowed with an ample and globally generated line bundle $\cO_X(h)$, then it is natural to deal with coherent sheaves defined by some particular vanishings of their cohomology groups. For instance, it may be interesting to deal with the existence and classification of  {\sl Ulrich sheaves (with respect to $\cO_X(h)$)}, i.e. non--zero coherent sheaves $\cE$ on $X$ such that:
\begin{itemize}
\item $h^0\big(\cE(-(t+1)h)\big)=h^n\big(\cE((t-n)h)\big)=0$ if $t\ge0$;
\item $\cE$ is {\sl without intermediate cohomology}, i.e. $h^i\big(X,\cE(th)\big)=0$ for $0< i<n$, $t\in\bZ$. 
\end{itemize}

Ulrich sheaves have been object of deep study and several attempts of generalization both from the algebraic and the geometric point of view: without any claim of completeness we refer the interested reader to \cite{E--S--W, Ap--Kim, C--H2, Bea6, K--M--S} for further details on such sheaves. In particular, as pointed out in \cite{E--S--W}, it is natural to ask whether Ulrich sheaves actually exist on each irreducible projective scheme. Ulrich line bundles always exist on {\sl curves} (i.e. irreducible projective schemes of dimension $1$) when smooth: see \cite[p. 542]{E--S--W} for further details. When the dimension is at least $2$ such question is still wide open, despite the great number of scattered results proved in the last decade. 

Thus it could be reasonable to look for some kind of approximations of Ulrich sheaves, allowing  more non--zero cohomology groups. E.g. we can consider torsion--free coherent sheaves $\cE$ on a projective scheme $X$ with $\dim(X)=n$ such that:
\begin{itemize}
\item $h^0\big(\cE(-(t+1)h)\big)=h^n\big(\cE((t-n)h)\big)=0$ if $t\ge0$;
\item $h^i\big(\cE(-(i+t+1)h)\big)=h^{n-i}\big(\cE((t-n+i)h)\big)=0$ if $1\le i\le n-2$, $t\ge0$.
\end{itemize}
If $X\cong\p n$ and $\cO_X(h)\cong\cO_{\p n}(1)$, such sheaves are characterized as the torsion--free ones which are cohomology of some {\sl linear monad}, i.e. a monad of the form
\begin{equation*}
0\longrightarrow\cO_{\p n}(-1)^{\oplus a}\longrightarrow\cO_{\p n}^{\oplus b}\longrightarrow\cO_{\p n}(1)^{\oplus c}\longrightarrow0
\end{equation*}
(see \cite{Ja}). It is immediate to check that the above monad has a symmetric shape, i.e. $a=c$, if and only if $h^1\big(\cE(-h)\big)=h^{n-1}\big(\cE(-nh)\big)$ or, equivalently, if and only if $c_1(\cE)=0$: a sheaf $\cE$ with such properties is called {\sl instanton sheaf} in \cite{Ja}.

Classically, a {\sl mathematical instanton bundle} is a rank two vector bundle $\cE$ on $\p3$ satisfying $h^0\big(\cE\big)=h^1\big(\cE(-2)\big)=0$ and  $c_1(\cE)=0$. Instanton bundles have been widely studied from different viewpoints since the discovery, through the Atiyah--Penrose--Ward transformation, of their connection with the solutions of the Yang--Mills equations (see \cite{A--W}) and therefore with the physics of particles.
 
Because of the intrinsic interest due to its physical importance, the notion of instanton bundle has been extended in several papers to different classes of  projective schemes: e.g. see \cite{Ok--Sp, A--O1,Fa2, Kuz,Ja, J--MR, C--MR3}.  In particular, M. Jardim and R.M. Mir\'o--Roig introduced the following definition on smooth irreducible projective schemes of dimension $n$ ({\sl $n$--folds} for short).

\begin{definition}[see \cite{J--MR}]
\label{dJardim}
Let $X$ be an $n$--fold endowed with a very ample line bundle $\cO_X(h)$.

A torsion--free coherent sheaf $\cE$ is called instanton sheaf with quantum number $\quantum\ge0$ if it is the cohomology of a monad 
\begin{equation}
\label{MonadJardim}
0\longrightarrow\cO_{X}(-h)^{\oplus \quantum}\longrightarrow\cO_{X}^{\oplus b}\longrightarrow\cO_{X}(h)^{\oplus \quantum}\longrightarrow0.
\end{equation}
\end{definition}

Another possible way of generalizing mathematical instanton bundles is by considering rank two bundles supported on $3$--folds, which are similar to $\p3$. 

Recall that the $n$--fold $X$ is {\sl Fano} if the dual of its canonical line bundle $\omega_X$ is ample. The greatest integer $i_X$ such that $\omega_X\cong\cO_X(-i_Xh)$ for some ample line bundle $\cO_X(h)$ is well--defined and called {\sl index of $X$}. Similarly, the line bundle $\cO_X(h)$ is well--defined as well and called {\sl fundamental line bundle of $X$}. It is well--known that $1\le i_X\le n+1$. Moreover, $i_X=n+1$ if and only if $X\cong\p n$ and $i_X=n$ if and only if $X$ is a smooth quadric hypersurface in $\p{n+1}$: there also exists a complete classification when $X$ is also {\sl del Pezzo}, i.e.  $i_X=n-1$. When $i_X\le n-2$ there are many deformation classes.

From the Fano $3$--folds viewpoint $\cO_{\p3}(-2)$ is the square root of $\omega_{\p3}$, hence an instanton bundle $\cE$ on $\p3$ can be also viewed as the normalization of a stable bundle $\cF$ of rank $2$ such that $\cF\cong\cF^\vee\otimes\omega_{\p3}$ and $h^1\big(\cF\big)=0$. For this reason D. Faenzi introduced the following definition in \cite{Fa2} (see also \cite{Kuz} for $i_X=2$). 

\begin{definition}[see \cite{Fa2}: see also \cite{Kuz}]
\label{dFaenzi}
Let $X$ be a Fano $3$--fold with cyclic Picard group, fundamental line bundle $\cO_X(h)$, index $i_X$ and let
$
q_X:=\left[\frac{i_X}2\right]
$.

A vector bundle $\cE$ of rank two on $X$ is called instanton bundle if the following properties hold:
\begin{itemize}
\item $c_1(\cE)=(2q_X-i_X)h$;
\item $h^0\big(\cE\big)=h^1\big(\cE(-q_Xh)\big)=0$.
\end{itemize}
\end{definition}
The above definition has been generalized also to Fano $3$--folds  $X$ with non--cyclic Picard group in \cite{C--C--G--M} (see also \cite{M--M--PL}) and to any $3$--fold in \cite{A--M2}. 

In particular there are two different possible definitions for an instanton bundle $\cE$ on the smooth quadric hypersurface $X\subseteq\p4$ in the literature, i.e. the one according to Definition \ref{dJardim}, which satisfies $c_1(\cE)=0$, and the one according to Definition \ref{dFaenzi}, which satisfies $c_1(\cE)=-h$. A similar duplicity of definitions has been recently explored in \cite{A--C--G} for each Fano $3$--fold (see also \cite{El--Gr} when $X\cong\p3$).

In this paper we try to merge all the aforementioned different viewpoints from \cite{J--MR, Kuz, Fa2, El--Gr, A--M2} in the following purely cohomological definition.

\begin{definition}
\label{dMalaspinion}
Let $X$ be an irreducible projective scheme of dimension $n\ge1$ endowed with an ample and globally generated line bundle $\cO_X(h)$.

A non--zero coherent sheaf  $\cE$ on $X$ is called instanton sheaf (with respect to $\cO_X(h)$) with defect $\defect\in\{\ 0,1\ \}$ and quantum number $\quantum\in\bZ$ if the following properties hold:
\begin{itemize}
\item $h^0\big(\cE(-(t+1)h)\big)=h^n\big(\cE((\defect-n+t)h)\big)=0$ if $t\ge0$;
\item $h^i\big(\cE(-(i+t+1)h)\big)=h^{n-i}\big(\cE((\defect+t-n+i)h)\big)=0$ if $1\le i\le n-2$, $t\ge0$;
\item $\defect h^i\big(\cE(-ih)\big)=0$ for $2\le i\le n-2$;
\item $h^1\big(\cE(-h)\big)=h^{n-1}\big(\cE((\defect-n)h)\big)=\quantum$;
\item $\defect(\chi(\cE)-(-1)^n\chi(\cE(-nh)))=0$.
\end{itemize}
The instanton sheaf $\cE$ is called ordinary or non--ordinary according to $\defect$ is $0$ or $1$.

If $\cE$ is an instanton sheaf with respect to $\cO_X(h)$ we will also briefly say that $\cE$ is an $\cO_X(h)$--instanton sheaf or a $h$--instanton sheaf or simply an instanton sheaf if the polarization is evident from the context.
\end{definition}

The first three properties in Definition \ref{dMalaspinion} are related to the existence of few non--zero entries in the cohomology table of $\cE$: in particular, $h^i\big(\cE(th)\big)=0$ for $t\in\bZ$ and $2\le i\le n-2$. The two latter conditions imply that the cohomology $h^i\big(\cE(th)\big)$ is approximately symmetric in the range $-n\le t\le\defect-1$.

Moreover, it is clear from the above definition that the $h$--instanton sheaf $\cE$ is Ulrich if $\defect=\quantum=0$. Also the converse is true (see Corollary \ref{cUlrich}), hence instanton sheaves may be also viewed as some kind of approximations of Ulrich sheaves.
\medbreak

We briefly describe below the content of the paper. In Section \ref{sPrel} we fix the notation used in the paper and collect several general results. 

In Section \ref{sSpace} we recall some well known facts from \cite{Ja} about ordinary instanton sheaves on $\p n$ with respect to $\cO_{\p n}(1)$, giving some kind of generalization to the non--ordinary case. More precisely, when $n\ge2$, we recall how to use  the Beilinson theorem for associating to each instanton sheaf on $\p n$ a monad $\cM^\bullet$ with a very precise shape (see equalities \eqref{MonadValue}). Moreover, such an $\cM^\bullet$ completely characterizes instanton sheaves on $\p n$ with respect to $\cO_{\p n}(1)$ with arbitrary defect.

In Section \ref{sGeneral} we prove several properties of instanton sheaves. Notice that Definition \ref{dMalaspinion} features the notion of ordinary instanton sheaf as a natural generalization of the notion of Ulrich sheaf and they actually coincide when $n=1$.

If $X$ is a projective scheme of dimension $n\ge1$ endowed with an ample and globally generated line bundle $\cO_X(h)$, then the morphism $X\to\p N$ induced by any choice of a basis of $H^0\big(\cO_X(h)\big)$ is finite. Thus its image has dimension $n$ and we can project it from $N-n$ general points obtaining a finite morphism $p\colon X\to\p n$ such that $\cO_X(h)\cong p^*\cO_{\p n}(1)$. Conversely each such morphism arises in this way.

Thus, the following characterization of instanton sheaves extending \cite[Theorem 2.1]{E--S--W} seems to be quite natural. Moreover, it is clearly  helpful  for both reducing the infinite set of vanishings required in Definition \ref{dMalaspinion} to a finite one and relating $h$--instanton sheaves with the aforementioned monad $\cM^\bullet$ on the projective space.

\begin{theorem}
\label{tCharacterization}
Let $X$ be an irreducible projective scheme of dimension $n\ge1$ endowed with an ample and globally generated line bundle $\cO_X(h)$.

If $\cE$ is a coherent sheaf on $X$, $\quantum$ a non--negative integer and $\defect\in\{\ 0, 1\ \}$, then the following assertions are equivalent.
\begin{enumerate}
\item $\cE$ is a $h$--instanton sheaf with defect $\defect$ and quantum number $\quantum$.
\item $\cE$ is non--zero and the following finite set of properties holds:
\begin{itemize}
\item $h^0\big(\cE(-h)\big)=h^n\big(\cE((\defect-n)h)\big)=0$;
\item $h^i\big(\cE(-(i+1)h)\big)=h^{n-i}\big(\cE((\defect-n+i)h)\big)=0$ if $1\le i\le n-2$;
\item $\defect h^i\big(\cE(-ih)\big)=0$ for $2\le i\le n-2$;
\item $h^1\big(\cE(-h)\big)=h^{n-1}\big(\cE((\defect-n)h)\big)=\quantum$;
\item $\defect(\chi(\cE)-(-1)^n\chi(\cE(-nh)))=0$.
\end{itemize}
\item For each  finite morphism $p\colon X\to\p n$ with $\cO_X(h)\cong p^*\cO_{\p n}(1)$, then $p_*\cE$ is an instanton sheaf on $\p n$ with respect to $\cO_{\p n}(1)$  with defect $\defect$ and quantum number $\quantum$.
\item There is a finite morphism $p\colon X\to\p n$ with $\cO_X(h)\cong p^*\cO_{\p n}(1)$ such that $p_*\cE$ is an instanton sheaf on $\p n$ with respect to $\cO_{\p n}(1)$ with defect $\defect$ and quantum number $\quantum$.\end{enumerate}
\end{theorem}

Taking  account  of the monadic representation of instanton sheaves on $\p n$ described in Section \ref{sSpace}, it follows that instanton sheaves can be still characterized, in some sense, in terms of  monads on each projective scheme $X$. Though such a description is not so easy to handle, we obtain some further results concerning the Hilbert function of an instanton sheaf and their relation with Ulrich sheaves. Moreover, we also inspect the behaviour of their restrictions to general divisors in $\vert h\vert$.

When we allow some further properties for the projective scheme $X$ or for the sheaf $\cE$ we can say something more. E.g. if $\cO_X(h)$ is very ample then we can deal with the minimal free resolution of $\cE$. More precisely, for each $h$--instanton sheaf $\cE$ on $X$ with defect $\defect$, we define
\begin{gather*}
v(\cE):=\min\{\ t\in\bZ\ \vert\ h^0\big(\cE(th)\big)\ne0\ \},\qquad 
w(\cE):=h^1\big(\cE((\defect-1)h)\big)+\defect.
\end{gather*}
The Castelnuovo--Mumford regularity of $\cE$ is bounded from above by $w(\cE)$ (see Proposition \ref{pRegularity}) and the following result holds.

\begin{theorem}
\label{tResolution}
Let $X$ be an irreducible projective scheme of dimension $n\ge2$ endowed with a very ample line bundle $\cO_X(h)$. Let $V\subseteq H^0\big(\cO_X(h)\big)$ be a subspace associated to an embedding $X\subseteq\p N$ and $S$ the symmetric $\field$--algebra of $V$.

If $\cE$ is a $h$--instanton sheaf, then the minimal free resolution of $H^0_*\big(\cE\big)$ as $S$--module has the form
\begin{equation}
\label{Resolution}
0\longrightarrow F_{N-1}\longrightarrow F_{N-2}\longrightarrow\dots\longrightarrow F_1\longrightarrow F_0\longrightarrow H^0_*\big(\cE\big)\longrightarrow0,
\end{equation}
where $F_p\cong \bigoplus_{i=v(\cE)}^{w(\cE)}S(-i-p)^{\oplus \beta_{p,i}}$.
\end{theorem}

The property of being $h$--instanton imposes a strong condition on the first Chern class of the sheaf (at least when $X$ is smooth). Indeed, in Section \ref{sBundle} we prove the following theorem.

\begin{theorem}
\label{tSlope}
Let $X$ be an $n$--fold with $n\ge1$ endowed with an ample and globally generated  line bundle $\cO_X(h)$. Assume that either the characteristic of $\field$ is $0$ or $\cO_X(h)$ is very ample.

Let $\cE$ be a non--zero vector bundle on $X$ and $\defect\in\{\ 0,1\ \}$ such that the following conditions hold:
\begin{itemize}
\item $h^0\big(\cE(-h)\big)=h^n\big(\cE((\defect-n)h)\big)=0$;
\item $h^i\big(\cE(-(i+1)h)\big)=h^{j}\big(\cE((\defect-j)h)\big)=0$ if $1\le i\le n-2$, $2\le j\le n-1$;
\item $\defect h^i\big(\cE(-ih)\big)=0$ for $2\le i\le n-2$;
\item $\defect h^1\big(\cE(-h)\big)=\defect h^{n-1}\big(\cE((\defect-n)h)\big)$.
\end{itemize}
Then $\cE$ is a $h$--instanton bundle on $X$ with defect $\defect$ if and only if
\begin{equation}
\label{Slope}
c_1(\cE)h^{n-1}=\frac{\rk(\cE)}2((n+1-\defect)h^n+K_Xh^{n-1}).
\end{equation}
\end{theorem}

In Proposition \ref{pSpecial}, we explain how Theorem \ref{tSlope} is simplified by the assumption that $\cE$ is a rank two vector bundle which is {\sl orientable}, i.e. $c_1(\cE)=(n+1-\defect)h+K_X$ as elements in $A^1(X)\cong\Pic(X)$. 

Ordinary instanton sheaves on smooth curves are exactly the Ulrich bundles, hence they always exist: the existence of non--ordinary instanton sheaves on smooth curves is similarly easy to check (see Proposition \ref{pCurve}). 

When $X$ is a {\sl surface} (i.e. an irreducible projective scheme with $n=2$) which is also smooth, rank two orientable instanton bundles are easy to produce, even when it is not known if $X$ supports Ulrich bundles (see Examples \ref{eMukai} and \ref{eGenus0}). 

As we pointed out, instanton sheaves are often defined in the literature as the cohomology of linear monads. We prove some results in this direction in Section \ref{sMonad}. 

More precisely, let $X$ be an $n$--fold with $n\ge3$ endowed with a very ample line bundle $\cO_X(h)$ and let $X\subseteq\p N$ be the induced embedding. A coherent sheaf $\cE$ on $X$ without intermediate cohomology is necessarily a vector bundle: in this case $\cE$ is called {\sl aCM bundle (with respect to  $\cO_X(h)$)}. If $\cO_X$ is an aCM bundle and $h^1\big(\cI_{X\vert \p N}(t)\big)=0$ for each $t\in \bZ$ we say that $X$ is {\sl aCM (with respect to  $\cO_X(h)$)}.

\begin{theorem}
\label{tMonad}
Let $X$ be an $n$--fold with $n\ge3$ endowed with a very ample line bundle $\cO_X(h)$. Assume that $X$ is aCM with respect to  $\cO_X(h)$.

If $\cE$ is a non--zero vector bundle on $X$, $\quantum$ a non--negative integer and $\defect\in\{\ 0, 1\ \}$, then the following assertions are equivalent.
\begin{enumerate}
\item $\cE$ is a $h$--instanton bundle with defect $\defect$ and quantum number $\quantum$.
\item $\cE$ satisfies $h^0\big(\cE(-h)\big)=h^n\big(\cE((\defect-n)h)\big)=0$ and it is the cohomology of a monad of the form
\begin{equation}
\label{Monad}
0\longrightarrow\cA\longrightarrow\cB\longrightarrow\cC\longrightarrow0,
\end{equation}
where 
$$
\cA:=\omega_X((n-\defect)h)^{\oplus \quantum}\oplus\omega_X((n+1-\defect)h)^{\oplus a},\qquad 
\cC:=\cO_X^{\oplus c}\oplus \cO_X(h)^{\oplus \quantum}
$$
are such that
\begin{equation}
\label{DimensionMonad}
a=\defect h^{n-1}\big(\cE(-n h)\big),\qquad c=\defect h^1\big(\cE\big)
\end{equation}
and $\cB$ is an aCM bundle such that
\begin{gather}
\label{h^iB}
{\begin{gathered}
h^0\big(\cB(-h)\big)=\quantum h^0\big(\omega_X((n-1-\defect)h)\big)+ah^0\big(\omega_X((n-\defect)h)\big),\\
h^n\big(\cB((\defect-n)h)\big)=\quantum h^0\big(\omega_X((n-1-\defect)h)\big)+ch^0\big(\omega_X((n-\defect)h)\big),
\end{gathered}}\\
\label{ChiB}
{\begin{gathered}
\defect(\chi(\cB)-(-1)^n\chi(\cB(-nh)))=\defect(c-a)(\chi(\cO_X)-(-1)^n\chi(\cO_X(-nh))).
\end{gathered}}
\end{gather}
\end{enumerate}
\end{theorem}
 
Moreover, we are also able to exploit another interesting link among ordinary instanton and Ulrich bundles for aCM $n$--folds $X$ whose adjoint linear system is not globally generated (see Corollary \ref{cMonad}). In particular, we recover the already known results for monads associated to instanton sheaves on smooth quadrics (and projective spaces) and we give some new examples of monads associated to instanton sheaves on some rational normal scrolls.

In Section \ref{sCyclic} we deal with the case of instanton bundles on $n$--folds $X$ which are {\sl cyclic}, i.e. such that $\Pic(X)$ is free of rank $\varrho_X=1$. On the one hand, we deal with instanton sheaves of rank up to two on such an $X$, studying their $\mu$--(semi)stability properties. On the other hand, we confront in Proposition \ref{pFano} rank two instanton bundles with respect to the fundamental line bundle on a Fano $3$--fold in the sense of Definition \ref{dMalaspinion} with the notion of instanton bundle introduced in \cite{A--C--G}, which extends to each Fano $3$--fold Definition \ref{dFaenzi}.

In particular, we show that if $\cE$ is not an ordinary rank two $\cO_X(h)$--instanton bundle on a {\sl prime Fano $3$--fold} $X$, i.e. a cyclic Fano $3$--fold with $i_X=1$, then the two notions coincide with very few exceptions, hence the existence of such bundles follows from the results proved in \cite{Fa2}.

When $i_X=1$ the existence of ordinary rank two $\cO_X(h)$--instanton bundle is slightly less immediate. Recall that in this case $\cO_X(h)\cong\omega_X^{-1}$ and $h^3$ is even, say $h^3=2g_X-2$: the number $g_X$ is called the {\sl genus} of $X$. In the case $\cO_X(h)$ is very ample, then it induces an embedding $X\subseteq\p{g_X+1}$. It is well--known that the $3$--fold $X$ contains lines and we say that $X$ is {\sl ordinary} if it contains a line with normal bundle $\cO_{\p1}\oplus\cO_{\p1}(-1)$.

The following result answers the question of the existence of $\cO_X(h)$--instanton bundle in the case $i_X=1$. 

\begin{theorem}
\label{tPrime}
Let $X$ be an ordinary prime Fano $3$--fold with very ample fundamental line bundle $\cO_X(h)$. Assume that $\field=\bC$.

For each $k\ge0$ there exists a rank two orientable, ordinary, $\mu$--stable, $h$--instanton bundle $\cE$ with quantum number $\quantum$, such that
$$
h^1\big(\cE\otimes\cE^\vee\big)=4+g_X+2k,\qquad h^i\big(\cE\otimes\cE^\vee\big)=0$$
for $i\ge2$.
\end{theorem}

In Section \ref{sSextic} we deal with low rank instanton bundles on del Pezzo $3$--folds $X$ of degree $6$, namely the general hyperplane section of  the image of the Segre embedding $\p2\times\p2\subseteq\p8$, which we usually call {\sl flag $3$--fold}, and  the image of the Segre embedding $\p1\times\p1\times\p1\subseteq\p7$.

In these cases $X$ is not cyclic, hence  several pathologies may appear as we  show in Examples \ref{eSegreDeform} and \ref{eSegreStable} when $X\cong\p1\times\p1\times\p1$. Nevertheless, when $X$ is the flag $3$--fold, we are able to prove the following result.

\begin{theorem}
\label{tFlag2}
Let $X$ be the flag $3$--fold and let $\cO_X(h)$ be its fundamental line bundle.

Every rank two orientable, ordinary, $h$--instanton bundle $\cE$ is $\mu$--semistable. If it is indecomposable, then it is also simple.
\end{theorem}

In Section \ref{sScrollCurve} we describe Construction \ref{conScrollCurve}, partially extending some results from \cite{A--M2} to scrolls over smooth curves of arbitrary genus, and we prove the following existence theorem.

\begin{theorem}
\label{tScrollCurve}
Let $\cG$ be a vector bundle of rank $n\ge3$ on a smooth curve $B$ and set $X:=\bP(\cG)$. Assume that $\cO_X(h):=\cO_{\bP(\cG)}(1)$  is an ample and globally generated line bundle.

For each integer $\quantum\ge0$ the bundle $\cE$ defined in Construction \ref{conScrollCurve} is an ordinary rank two orientable, $\mu$--semistable, $h$--instanton bundle with quantum number $\quantum$. 
\end{theorem}

The bundles obtained via Construction \ref{conScrollCurve} are actually simple if and only if $\quantum\ge1$ (see Proposition \ref{pExt}). The existence of rank two orientable, Ulrich, simple bundles on $X$, i.e. ordinary rank two orientable, $\mu$--semistable, simple, $h$--instanton bundles with quantum number $0$, is also proved in Proposition \ref{pExt0}.

\subsection{Acknowledgements}
The authors express their thanks to F. Malaspina and M. Jardim for some helpful suggestions, and to A.F. Lopez for an illuminating comment on the characterization of Ulrich sheaves via finite projections.

\section{Notation and first results}
\label{sPrel}
Throughout the whole paper we will work over an algebraically closed field $\field$. The projective space of dimension $N$ over $\field$ will be denoted by $\p N$: $\cO_{\p N}(1)$ will denote the hyperplane line bundle. A projective scheme $X$ is a closed subscheme of some projective space over $\field$: $X$ is a {{\sl variety}} if it is also integral. A {\sl manifold} $X$ is a smooth variety: we often use the term $n$--fold for underlying that $X$ is a manifold of dimension $n$.  The structure sheaf of a projective scheme $X$ is denoted by $\cO_X$ and its Picard group by $\Pic(X)$. 

Let $X$ be a projective scheme. For each closed subscheme $Z\subseteq X$ the ideal sheaf $\cI_{Z\vert X}$ of $Z$ in $X$ fits into the exact sequence
\begin{equation}
\label{seqStandard}
0\longrightarrow \cI_{Z\vert X}\longrightarrow \cO_X\longrightarrow \cO_Z\longrightarrow 0.
\end{equation}

If $\cA$ is a coherent sheaf on a projective scheme $X$ we set $h^i\big(\cA\big):=\dim H^i\big(\cA\big)$. If $X$ is  endowed with a globally generated line bundle $\cO_X(h)$, then we have an induced morphism $\phi_h\colon X\to\p N$ where $N+1=h^0\big(\cO_X(h)\big)$ and $\phi_h$ is finite if and only is $\cO_X(h)$ is also ample.

For each $i\ge0$  we set $H^i_*\big(\cA\big):=\bigoplus_{t\in\bZ} H^i\big(\cA(th)\big)$.
If $S$ is the symmetric $\field$--algebra over $H^0\big(\cO_X(h)\big)$, then  $S\cong \field[x_0,\dots,x_N]$  and $H^i_*\big(\cA\big)$ is naturally an $S$--module. The morphism $\phi_h$ induces a morphism $S\to H^0_*\big(\cO_X\big)$ of $\field$--algebras whose image is denoted by $S[X]$.

Let $\cA$ be a coherent sheaf on $X$. We say that a $\cA$ has {\sl natural cohomology in shift $t$} if $h^i\big(\cA(th)\big)=0$ for all $i\in\bZ$ but one. The sheaf $\cA$ is called {\sl $m$--regular (in the sense of Castelnuovo--Mumford)} if $h^i\big(\cA((m-i)h)\big)=0$ for $i\ge1$ and the regularity $\reg(\cA)$ of $\cA$ is defined as the minimum integer $m$ such that $\cA$ is $m$--regular. We refer the reader to \cite[Chapter 4]{Ei} for further details about this notion.

The following result will be used several times.

\begin{proposition}
\label{pNatural}
Let $X$ be a projective scheme of dimension $n\ge1$ endowed with a globally generated line bundle $\cO_X(h)$.

If $\cA$ is a coherent sheaf on $X$ and there is $m\le n-1$ (resp $m\ge 1$) such that $h^i\big(\cA(-ih)\big)=0$ for each $i\le m$ (resp $i\ge m$), then the following assertions hold.
\begin{enumerate}
\item $h^i\big(\cA(-th)\big)=0$ for each $t\ge i$ and $i\le m$ (resp. $t\le i$ and $i\ge m$).
\item $h^i\big(\cO_Y\otimes\cA(-ih)\big)=0$ for each $i\le m-1$ (resp. $i\ge m$) and general $Y\in\vert h\vert$.
\end{enumerate}
\end{proposition}
\begin{proof}
For each general $Y\in\vert h\vert$, let $\cO_Y(h):=\cO_Y\otimes\cO_X(h)$ and $\cA_Y:=\cO_Y\otimes\cA$: notice that $\cA_Y$ is coherent and $\dim(Y)=n-1$.

Assume that $\cA$ is not zero, otherwise the statement is trivially true. The set of associated points of $\cA$ is finite, hence the general $Y\in\vert h\vert$  does not contain any such point. Thus the exact sequence
\begin{equation}
\label{seqSection} 
0\longrightarrow \cO_X(-h)\longrightarrow\cO_X\longrightarrow\cO_Y\longrightarrow0.
\end{equation}
tensored by $\cA(th)$ remains  exact. 

Let $n\ge1$, $m\le n-1$ (resp. $m\ge1$) and let  $\cA$ be a coherent sheaf on $X$ satisfying $h^i\big(\cA(-ih)\big)=0$ for each $i\le m$ (resp. $i\ge m$). If $m=0$ (resp. $m=n$), then the cohomology of sequence \eqref{seqSection} tensored by $\cA(th)$ implies
$$
h^0\big(\cA(-(t+1)h)\big)\le h^0\big(\cA(-th)\big),\qquad (\text{resp. $h^n\big(\cA(-(t+1)h)\big)\ge h^n\big(\cA(-th)\big)$})
$$
when $t\ge 0$ (resp. $t\le 1$). Thus the statement is true in this case.

We complete the proof of the statement by induction on $n\ge1$, the base case $n=1$ being a particular case of the discussion above. Thus we assume that the statement holds true for each projective scheme of dimension $n - 1 \ge  1$ and that $m\ge1$ (resp. $m\le n-1$) in what follows.

The cohomology of sequence \eqref{seqSection} tensored by $\cA(-ih)$ implies $h^i\big(\cA_Y(-ih)\big)=0$ for each $i\le m-1$ (resp. $i\ge m$), hence assertion (2) is proved. 

By the inductive hypothesis $h^i\big(\cA_Y(-th)\big)=0$ for each $t\ge i$ and $i\le m-1$ (resp. $t\le i$ and $i\ge m$). Again sequence \eqref{seqSection} tensored by $\cA(th)$ implies 
$$
h^{i}\big(\cA(-(t+1)h)\big)\le h^{i}\big(\cA(-th)\big),\qquad (\text{resp. $h^i\big(\cA(-(t+1)h)\big)\ge h^i\big(\cA(-th)\big)$})
$$
for each $t\ge i-1$ and $i\le m$ (resp. $t\le i$ and $i\ge m$). We have $h^{i}\big(\cA(-ih)\big)=0$ for $i\le m$ (resp. $i\ge m$) by hypothesis, hence $h^{i}\big(\cA(-th)\big)=0$ for each $i\le m$ and $t\ge i$ (resp. $t\le i$ and $i\ge m$). 
\end{proof}

If $\cA$ and $\cB$ are coherent sheaves on an $n$--fold $X$, then the Serre duality holds
\begin{equation}
\label{Serre}
\Ext_X^i\big(\cA,\cB\otimes\omega_X\big)\cong \Ext_X^{n-i}\big(\cB,\cA\big)^\vee
\end{equation}
(see \cite[Proposition 7.4]{Ha3}).

If $X$ is an $n$--fold, we denote by  $A^r(X)$ the group of cycles on $X$ of codimension $r$ modulo rational equivalence: in particular $A^1(X)\cong\Pic(X)$ (see \cite[Proposition 1.30]{Ei--Ha2}) and we set $A(X):=\bigoplus_{r\ge0}A^r(X)$. The Chern classes of a coherent sheaf $\cA$ on $X$ are elements in $A(X)$: in particular, when $\cA$ is locally free, then $c_1(\cA)$ is identified with $\det(\cA)$ via the isomorphism $A^1(X)\cong\Pic(X)$.

If $X$ is a smooth curve, a smooth surface, a $3$--fold, Hirzebruch--Riemann--Roch formulas for a coherent sheaf $\cA$ are 
\begin{gather}
  \label{RRcurve}
\chi(\cA)=\rk(\cA)\chi(\cO_X)+\deg(c_1(\cA)),\\
\label{RRsurface}
\chi(\cA)=\rk(\cA)\chi(\cO_X)+{\frac12}c_1(\cA)^2-{\frac12}\omega_Xc_1(\cA)-c_2(\cA),\\
  \label{RRgeneral}
\begin{aligned}
    \chi(\cA)&=\rk(\cA)\chi(\cO_X)+{\frac16}(c_1(\cA)^3-3c_1(\cA)c_2(\cA)+3c_3(\cA))\\
    &-{\frac14}(\omega_Xc_1(\cA)^2-2\omega_Xc_2(\cA))+{\frac1{12}}(\omega_X^2c_1(\cA)+c_2(\Omega^1_{X}) c_1(\cA)),
  \end{aligned}
 \end{gather}
 respectively (see \cite[Theorem 14.4]{Ei--Ha2}). 

Let $\cA$ be a rank two vector bundle on an $n$--fold $X$ and let $s\in H^0\big(\cA\big)$. In general its zero--locus
$(s)_0\subseteq X$ is either empty or its codimension is at most
$2$. We can always write $(s)_0=Y\cup Z$
where $Z$ has codimension $2$ (or it is empty) and $Y$ has pure codimension
$1$ (or it is empty). In particular $\cA(-Y)$ has a section vanishing
on $Z$, thus we can consider its Koszul complex 
\begin{equation}
  \label{seqSerre}
  0\longrightarrow \cO_X(Y)\longrightarrow \cA\longrightarrow \cI_{Z\vert X}(-Y)\otimes\det(\cA)\longrightarrow 0.
\end{equation}
Sequence \ref{seqSerre} tensored by $\cO_Z$ yields $\cI_{Z\vert X}/\cI^2_{Z\vert X}\cong\cA^\vee(Y)\otimes\cO_Z$, whence the normal bundle $\cN_{Z\vert X}:=(\cI_{Z\vert X}/\cI^2_{Z\vert X})^\vee$ of $Z$ inside $X$ satisfies
\begin{equation}
\label{Normal}
\cN_{Z\vert X}\cong\cA(-Y)\otimes\cO_Z.
\end{equation}
If $Y=\emptyset$, then $Z$ is locally complete intersection inside $X$, because $\rk(\cA)=2$. In particular, it has no embedded components.

The Serre correspondence reverts the above construction as follows.

\begin{theorem}
  \label{tSerre}
Let $X$ be an $n$--fold with $n\ge2$ and $Z\subseteq X$ a local complete intersection subscheme of codimension $2$.

If $\det(\cN_{Z\vert X})\cong\cO_Z\otimes\mathcal L$ for some $\mathcal L\in\Pic(X)$ such that $h^2\big(\mathcal L^\vee\big)=0$, then there exists a rank two vector bundle $\cA$ on $X$ satisfying the following properties.
  \begin{enumerate}
  \item $\det(\cA)\cong\mathcal L$.
  \item $\cA$ has a section $s$ such that $Z$ coincides with the zero locus $(s)_0$ of $s$.
  \end{enumerate}
  Moreover, if $H^1\big({\mathcal L}^\vee\big)= 0$, the above two conditions  determine $\cA$ up to isomorphism.
\end{theorem}
\begin{proof}
See \cite{Ar}.
\end{proof}

Let $X$ be an $n$--fold with $n\ge1$ endowed with an ample line bundle $\cO_X(h)$. If $\cA$ is any torsion--free sheaf we define the {\sl slope of $\cA$ (with respect to $\cO_X(h)$)} as
$$
\mu_h(\cA):=\frac{c_1(\cA)h^{n-1}}{\rk(\cA)}.
$$
The torsion--free sheaf $\cA$ is {\sl $\mu$--semistable} (resp. {\sl $\mu$--stable}) if for all proper subsheaves $\cB$ with $0<\rk(\cB)<\rk(\cA)$ we have $\mu_h(\cB) \le \mu_h(\cA)$ (resp. $\mu_h(\cB)<\mu_h(\cA)$). Each $\mu$--stable bundle $\cA$ is {\sl simple}, i.e. $h^0\big(\cA\otimes\cA^\vee\big)=1$ (see \cite[Theorem II.1.2.9]{O--S--S}).

\section{Instanton sheaves on projective spaces}
\label{sSpace}
In this section we deal with instanton sheaves on $\p n$ with respect to $\cO_{\p n}(1)$. We start by quickly dealing with the case $n=1$.

\begin{remark}
\label{rLine}
If $\cE$ is an instanton sheaf with respect to $\cO_{\p1}(1)$, then it splits in the direct sum of a vector bundle plus its torsion subsheaf $\cT$. Thus $h^0\big(\cT\big)\le h^0\big(\cE(-h)\big)=0$, hence $\cT=0$ and $\cE$ is actually a vector bundle. Thanks to Definition \ref{dMalaspinion} one immediately deduces
$$
\cE\cong\left\lbrace\begin{array}{ll} 
\cO_{\p1}^{\oplus \chi(\cE)}\quad&\text{if $\defect=0$,}\\
(\cO_{\p1}\oplus\cO_{\p1}(-1))^{\oplus \chi(\cE)}\quad&\text{if $\defect=1$:}
\end{array}\right.\\
$$
hence the quantum number of $\cE$ is $\defect \chi(\cE)=\defect\rk(\cE)/2$.
\end{remark}

In what follows we assume $n\ge2$. We prove below a generalization of the well--known characterization of instanton sheaves (e.g. see \cite[Sections 1 and 2]{Ja} and \cite[Section 5.2]{El--Gr}) in terms of the cohomology of a very simple monad $\cM^\bullet$, i.e. a complex
$$
0\longrightarrow \cM^{-1}\longrightarrow \cM^{0}\longrightarrow \cM^{1}\longrightarrow0
$$
which is exact everywhere but in degree $0$. To this purpose we set
\begin{equation}
\label{MonadValue}
\begin{gathered}
\cM^{-1}:=\left\lbrace\begin{array}{ll} 
\cO_{\p n}(-1)^{\oplus k}\quad&\text{if $\defect=0$,}\\
\cO_{\p n}(-1)^{\oplus b_1-\chi(\cE)}\quad&\text{if $\defect=1$,}
\end{array}\right.\\
\cM^{0}:=\left\lbrace\begin{array}{ll} 
\cO_{\p n}^{\oplus \chi(\cE)+(n+1)\quantum}\quad&\text{if $\defect=0$,}\\
\begin{aligned}
\cO_{\p n}^{\oplus b_0}&\oplus \Omega_{\p n}^1(1)^{\oplus \quantum} \\
&\oplus\Omega_{\p n}^{n-1}(n-1)^{\oplus \quantum}\oplus\cO_{\p n}(-1)^{\oplus b_1}
\end{aligned}
\quad&\text{if $\defect=1$, $n\ge3$,}\\
\begin{aligned}
\cO_{\p 2}^{\oplus b_0}\oplus \Omega_{\p 2}^1(1)^{\oplus \quantum}\oplus\cO_{\p 2}(-1)^{\oplus b_1}
\end{aligned}
\quad&\text{if $\defect=1$, $n=2$,}
\end{array}\right.\\
\cM^{1}:=\left\lbrace\begin{array}{ll} 
\cO_{\p n}(1)^{\oplus k}\quad&\text{if $\defect=0$,}\\
\cO_{\p n}^{\oplus b_0-\chi(\cE)}\quad&\text{if $\defect=1$.}
\end{array}\right.
\end{gathered}
\end{equation}

\begin{proposition}
\label{pSpace}
Let $n\ge2$.

For a non--zero coherent sheaf $\cE$ on $\p n$ the following assertions hold.
\begin{enumerate}
\item If $\cE$ is an instanton sheaf with respect to $\cO_{\p n}(1)$ with defect $\defect$ and quantum number $\quantum$, then $\cE$ is the cohomology of a monad $\cM^\bullet$ where the $\cM^i$'s are as in equalities \eqref{MonadValue} with, if $\defect=1$, $b_0\le h^0(\cE)$ and $b_1\le h^{n}(\cE(-n))$.
\item If $\cE$ is the cohomology of a monad $\cM^\bullet$ where the $\cM^i$'s are as in equalities \eqref{MonadValue} with, if $\defect=1$, $b_0,b_1\ge\chi(\cE)$, then $\cE$ is an instanton sheaf with respect to $\cO_{\p n}(1)$ with defect $\defect$ and quantum number $\quantum$.
\end{enumerate}
\end{proposition}
\begin{proof}
Assume $\cE$ is an instanton bundle on $\p n$.

If $\defect=0$, then \cite[Theorem 3]{Ja} implies the existence of a monad of the form 
$$
0\longrightarrow\cO_{\p n}(-1)^{\oplus a}\longrightarrow\cO_{\p n}^{\oplus b}\longrightarrow\cO_{\p n}(1)^{\oplus c}\longrightarrow0,
$$
whose cohomology is $\cE$ under the additional hypothesis that it is torsion--free. Such latter hypothesis is only used for proving that sequence \eqref{seqSection} tensored by $\cE(t)$ remains exact for a general hyperplane $Y\subseteq\p n$. Since the general hyperplane does not contain any associated point of $\cE$, it follows that the torsion--freeness hypothesis on $\cE$ can be removed.

 Splitting the monad above in the two short exact sequences 
\begin{gather*}
0\longrightarrow\cU\longrightarrow\cO_{\p n}^{\oplus b}\longrightarrow\cO_{\p n}(1)^{\oplus c}\longrightarrow0,\\
0\longrightarrow\cO_{\p n}(-1)^{\oplus a}\longrightarrow\cU\longrightarrow\cE\longrightarrow0.
\end{gather*}
and taking their cohomology, possibly twisted by $\cO_X(-h)$ and $\cO_X(-nh)$, we obtain
$$
c=h^1\big(\cE(-1)\big),\qquad a=h^{n-1}\big(\cE(-n)\big),\qquad \chi(\cE)=b-(n+1)c
$$
hence $a=c=\quantum$ and $b=\chi(\cE)+(n+1)\quantum$. 

If $\defect=1$, then $\cE$ is the cohomology of a complex $\cM^\bullet$ which is everywhere exact but in degree $0$ and with $i^{\textrm {th}}$--term of the form
$$
\widehat{\cM}^i:=\bigoplus_{p+q=i}H^q\big(\cE(p)\big)\otimes \Omega_{\p n}^{-p}(-p):
$$
see \cite[Beilinson Theorem (strong form)]{A--O1}. In particular $\cM^i=0$ if $\vert i\vert\ge2$ and
\begin{gather*}
\widehat{\cM}^{-1}\cong \cO_{\p n}(-1)^{\oplus h^{n-1}(\cE(-n))},\qquad \widehat{\cM}^{1}\cong \cO_{\p n}^{\oplus h^{1}(\cE)},\\
\widehat{\cM}^0\cong\left\lbrace\begin{array}{ll} 
 \cO_{\p n}^{\oplus h^0(\cE)}\oplus \left(\Omega_{\p n}^1(1)\oplus \Omega_{\p n}^{n-1}(n-1)\right)^{\oplus \quantum}\oplus\cO_{\p n}(-1)^{\oplus h^{n}(\cE(-n))},\quad &\text{if $n\ge3$,}\\
 \cO_{\p 2}^{\oplus h^0(\cE)}\oplus \Omega_{\p 2}^1(1)^{\oplus \quantum}\oplus\cO_{\p 2}(-1)^{\oplus h^{2}(\cE(-2))},\quad &\text{if $n=2$.}
\end{array}\right.
\end{gather*}
The equality
$$
h^{0}\big(\cE\big)- h^{1}\big(\cE\big)=\chi(\cE)=(-1)^n\chi(\cE(-n))= h^{n}\big(\cE(-n)\big)- h^{n-1}\big(\cE(-n)\big).
$$
implies $h^{1}\big(\cE\big)=h^0\big(\cE\big)-\chi(\cE)$, $h^{n-1}\big(\cE(-n)\big)=h^{n}\big(\cE(-n)\big)-\chi(\cE)$. Erasing isomorphic summands in the $\widehat{\cM}^i$'s if any, we finally obtain the monad $\cM^\bullet$ where $b_0\le h^0(\cE)$ and $b_1\le h^{n}(\cE(-n))$. 

Conversely, regardless of the value of $\defect$, if $\cE$ is the cohomology of the monad $\cM^\bullet$ whose $\cM^i$'s are as in equalities \eqref{MonadValue}, then we have the induced exact sequences
\begin{equation}
\label{DisplayM}
\begin{gathered}
0\longrightarrow\cU\longrightarrow\cM^0\longrightarrow\cM^1\longrightarrow0,\\
0\longrightarrow\cM^{-1}\longrightarrow\cU\longrightarrow\cE\longrightarrow0.
\end{gathered}
\end{equation}
Computing the cohomologies of the above sequences \eqref{DisplayM} after suitable twists, one deduces that $\cE$ is an instanton bundle with defect $\defect$ and quantum number $\quantum$.
\end{proof}
Let $n\ge2$. In the following remarks $\cE$ is an instanton sheaf on $\p n$  with defect $\defect$ and quantum number $\quantum$. We define $\rk(\cE)$ as the rank of the stalk of $\cE$ at the generic point of $\p n$. We will see later on in Corollary \ref{cCharacteristic} that the support of $\cE$ is $\p n$, hence $\rk(\cE)\ge1$ (see \cite[Corollary of Lemma II.1.14]{O--S--S}).

\begin{remark}
\label{rRankChern}
If $\cE$ is ordinary we deduce that $\rk(\cE)=\chi(\cE)+(n-1)\quantum$ thanks to a direct computation via the monad $\cM^\bullet$. If $\cE$ is non--ordinary we similarly obtain
$$
\rk(\cE)=\left\lbrace\begin{array}{ll} 
2\chi(\cE)+2n\quantum\quad&\text{if $n\ge3$,}\\
2\chi(\cE)+2\quantum\quad&\text{if $n=2$.}
\end{array}\right.
$$
Thus non--ordinary instanton sheaves on $\p n$ have necessarily even rank.
\end{remark}

\begin{remark}
\label{rChernPoly}
Monad $\cM^\bullet$ and Remark \ref{rRankChern} imply that
$$
c_t(\cE)=\left\lbrace\begin{array}{ll} 
\dfrac1{(1-t^2)^\quantum}\quad&\text{if $\defect=0$,}\\
\dfrac{(1-t)^{\frac{\rk(\cE)}2+\quantum}}{(1+t)^\quantum(1-2t)^\quantum}\quad&\text{if $\defect=1$, $n\ge3$,}\\
\dfrac{(1-t)^{\frac{\rk(\cE)}2}}{(1-t^2)^\quantum}\quad&\text{if $\defect=1$, $n=2$.}
\end{array}\right.
$$
Thus, via the canonical identification $A^i(\p n)\cong\bZ$, we obtain
\begin{equation}
\label{ChernP^n}
c_1(\cE)=-\defect\dfrac{\rk(\cE)}2,\qquad c_2(\cE)=\epsilon\quantum+\defect\dfrac{\rk(\cE)(\rk(\cE)-2)}8
\end{equation}
where
$$
\epsilon=\left\{
\begin{array}{ll}
1\quad&\text{if $n=2$,}\\
1+\defect\quad&\text{if $n\ge3$.}
\end{array}\right.
$$
\end{remark}

\section{Instanton sheaves on irreducible projective schemes}
\label{sGeneral}
In this section we characterize instanton sheaves on projective schemes. We start the section by proving some  general properties of instanton sheaves.

\begin{proposition}
\label{pQuantum}
Let $X$ be an irreducible  projective scheme of dimension $n\ge1$ endowed with an ample and globally generated line bundle $\cO_X(h)$.

If $\cE$ is a $h$--instanton sheaf on $X$  with defect $\defect$ and quantum number $\quantum$, then it has natural cohomology with respect to $\cO_X(h)$ in shifts $\defect-n\le t\le -1$. In particular 
\begin{equation*}
\quantum=-\chi(\cE(-h))=(-1)^{n-1}\chi(\cE((\defect-n)h)),
\end{equation*}
\end{proposition}
\begin{proof}
If $\cE$ is a $h$--instanton sheaf, then $h^i\big(\cE(-th)\big)=0$ if $0\le i\le n$ and $2\le t\le n-1-\defect$ or $i\ne1$ and $t=1$ or $i=n-1$ and $t=\defect-n$ by definition.
\end{proof}

The proof of the proposition below follows immediately from Definition \ref{dMalaspinion}.

\begin{proposition}
\label{pDirectSum}
Let $X$ be an irreducible projective scheme of dimension $n\ge1$ endowed with an ample and globally generated line bundle $\cO_X(h)$.

Every extension of two instanton sheaves on $X$ with the same defect is an instanton sheaf with the same defect and whose quantum number is the sum of the quantum numbers.
\end{proposition} 

The main result of this section is the proof of Theorem \ref{tCharacterization}. 
\medbreak
\noindent{\it Proof of Theorem \ref{tCharacterization}.}
Assertion (2) follows trivially from assertion (1). Conversely, assertion (1) follows from assertion (2) thanks to Proposition \ref{pNatural}.

We now show that assertions (1), (3) and (4) are equivalent. Assume assertion (1) holds. Thanks to \cite[Corollary III.11.2 and Exercises III.8.1, III.8.3]{Ha2} we have
\begin{equation}
\label{PushDown}
h^i\big((p_*\cE)(t)\big)=h^i\big(\cE(th)\big)
\end{equation}
for each $i,t\in\bZ$. It follows that $p_*\cE$ is an instanton sheaf on $\p n$ with respect to $\cO_{\p n}(1)$. Thus assertion (3) is true.

Trivially, assertion (3) implies assertion (4). Assume finally that assertion (4) holds. Thanks to Proposition \ref{pSpace} we know that $p_*\cE$ is the cohomology of monad $\cM^\bullet$ where $\cM^i$ is as in equalities \eqref{MonadValue}. The cohomology of the twists of the exact sequences \eqref{DisplayM} allows us to check that the values of $h^i\big(\cE(th)\big)=h^i\big((p_*\cE)(t)\big)$ are as in Definition \ref{dMalaspinion}.
\qed
\medbreak

We list below some immediate corollaries of Theorem \ref{tCharacterization}. From now on we set
$$
{m\choose n}:=\frac1{n!}{\prod_{i=0}^{n-1}(m-i)}
$$
for each integer $n\ge1$.

\begin{corollary}
\label{cCharacteristic}
Let $X$ be an irreducible  projective scheme of dimension $n\ge1$ endowed with an ample and globally generated line bundle $\cO_X(h)$.

If $\cE$ is a $h$--instanton sheaf on $X$ the following assertions hold.
\begin{enumerate}
\item The support of $\cE$ is the underlying space of $X$.
\item If the defect and quantum number of $\cE$ are  $\defect$ and $\quantum$ respectively, then
\begin{equation*}
\chi(\cE(th))=\left\lbrace\begin{array}{ll} 
{\begin{aligned}
(\chi(\cE)&+(n+1)\quantum)\left({{t+n}\choose n}+\defect{{t+n-\defect}\choose n}\right)\\
&-\quantum\left({{t+n+1}\choose n}+{{t+n-1-\defect}\choose n}\right)
\end{aligned}}\quad&{\begin{aligned}
&\text{if $n\ge2$ and\ \ }\\
&\text{ \ $(n,\defect)\ne(2,1)$,}
\end{aligned}}
\vspace{5pt}\\
\,(\chi(\cE)+\quantum)(t+1)^2-\quantum\quad&\text{if $(n,\defect)=(2,1)$,}
\vspace{5pt}\\
\,\chi(\cE)(t+1+\defect t)\quad&\text{if $n=1$.}
\end{array}\right.
\end{equation*}
\end{enumerate}
\end{corollary}
\begin{proof}
We first prove assertion (2). To this purpose, let $p\colon X\to \p n$ be finite with $\cO_X(h)\cong p^*\cO_{\p n}(1)$. Taking into account that $\chi(\cE(th))=\chi((p_*\cE)(t))$ by equalities \eqref{PushDown}, the statement follows by combining Proposition \ref{pSpace} and Remark \ref{rLine}.

We prove below assertion (1) for the case $\defect=1$ and $n\ge3$: in the other cases the argument is similar. If $\cE$ is supported on a proper subscheme, then the coefficient of $t^n$ in $\chi(\cE(th))$ must vanish. It follows that $\chi(\cE)=-n\quantum$, hence
\begin{align*}
\chi(\cE(th))=-\quantum\frac{n(n-1)(2t+n)}{(t-1)t(t+1)}{{t+n-2}\choose n}.
\end{align*}
Since $\cE$ is not the zero sheaf, it follows that $0<h^0\big(\cE(th)\big)=\chi(\cE(th))\le 0$ for $t\gg0$, a contradiction.
\end{proof}

\begin{corollary}
\label{cUlrich}
Let $X$ be an irreducible projective scheme of dimension $n\ge1$ endowed with an ample and globally generated line bundle $\cO_X(h)$.

If $\cE$ is a $h$--instanton sheaf on $X$ with quantum number $\quantum$ and defect $\defect$, then the following assertions hold.
\begin{enumerate}
\item $\quantum=\defect h^1\big(\cE\big)=\defect h^{n-1}\big(\cE(-nh)\big)=0$  if and only if $\cE$ is a sheaf without intermediate cohomology.
\item $\quantum=\defect=0$ if and only if $\cE$ is an Ulrich sheaf.
\end{enumerate}
\end{corollary}
\begin{proof}
If $\cE$ is without intermediate cohomology, then we certainly have the vanishings in assertion (1). Conversely, if such vanishings hold, then $\cE$ is without intermediate cohomology thanks to Proposition \ref{pNatural}. 

If $\quantum=\defect=0$, then $\cE$ is trivially Ulrich. Conversely, let $\cE$ be Ulrich. Thus $\quantum=-h^1\big(\cE(-h)\big)=0$. If $p\colon X\to\p n$  is finite and $\cO_X(h)\cong p^*\cO_{\p n}(1)$, then $p_*\cE\cong\cO_{\p n}^{\oplus \chi(\cE)}$ by \cite[Proposition 2.1]{E--S--W}, hence $c_1(p_*\cE)=0$. Thus Remark \ref{rChernPoly} implies that the defect of $p_*\cE$ vanishes, hence $\defect=0$.
\end{proof}

If $X$ is an irreducible  projective scheme of dimension $n\ge2$ and $\cO_X(h)$ is ample and globally generated, then \cite[Corollaire I.6.11 3)]{Jou} implies that each general $Y\in \vert h\vert$ is irreducible. In particular it makes sense to ask if the restriction of a $h$--instanton sheaf to $Y$ is still an instanton sheaf.

\begin{corollary}
\label{cRestriction}
Let $X$ be a  projective scheme of dimension $n\ge3$ endowed with an ample and globally generated line bundle $\cO_X(h)$.

If $Y\in \vert h\vert$ is general and $\cE$ is an instanton sheaf  with respect to $\cO_X(h)$, then $\cO_Y\otimes\cE$ is an instanton sheaf with respect to $\cO_Y\otimes\cO_X(h)$. Moreover, the following assertions hold.
\begin{enumerate} 
\item The defect of $\cO_Y\otimes\cE$ coincides with the defect  $\defect$ of $\cE$.
\item The quantum number of $\cO_Y\otimes\cE$ coincides with the quantum number  $\quantum$ of $\cE$ if $(n,\defect)\ne(3,1)$ and with $2\quantum$ if $(n,\defect)=(3,1)$.
\end{enumerate}
\end{corollary}
\begin{proof}
Let $\cE_Y:=\cO_Y\otimes\cE$, $\cO_Y(h_Y):=\cO_Y\otimes\cO_X(h)$.

Let $p\colon X\to\p n$ be a finite morphism such that $\cO_X(h)\cong p^*\cO_{\p n}(1)$. If $Y=p^{-1}(H)$ for some hyperplane $H\subseteq \p n$, then the restriction $p_Y\colon Y\to H$ of $p$ is still finite. Moreover, both $p$ and $p_Y$ are affine, hence for each open affine subset $\cU\cong\spec(A)\subseteq\p n$, let $\cV\cong\spec(B)=p^{-1}(\cU)$ and $H\cap \cU\cong\spec(A/I)$ for some ideal $I\subseteq A$, so that $Y\cap \cV\cong\spec(B\otimes_AA/I)$. If the restriction of $\cE$ to $\cV$ is $\widetilde{M}$, then the canonical isomorphism of $A$--algebras $M\otimes_B(B\otimes_AA/I)\cong M\otimes_AA/I$ holds true. Glueing together all such canonical isomorphisms we obtain $(p_Y)_*\cE_Y\cong \cO_{H}\otimes(p_*\cE)$.

If $p_*\cE$ is the cohomology of monad $\cM^\bullet$ as in Proposition \ref{pSpace} and $Y\in\vert h\vert$ does not contain any associated point of $\cE$, then $(p_Y)_*\cE_Y$ is the cohomology of monad $\cO_{H}\otimes\cM^\bullet$. Notice that if $\defect=1$, then $\Omega_{\p n}^{n-1}(n-1)\cong(\Omega_{\p n}^1(2))^\vee$, hence
\begin{gather*}
\cO_{H}\otimes\Omega_{\p n}^1(1)\cong\cO_{\p{n-1}}\oplus\Omega_{\p{n-1}}^1(1),\\
\cO_{H}\otimes\Omega_{\p n}^{n-1}(n-1)\cong\cO_{\p{n-1}}(-1)\oplus\Omega_{\p{n-1}}^{n-2}(n-2).
\end{gather*}
In particular $\cO_{H}\otimes\cM^\bullet$ has the same shape of $\cM^\bullet$. 

Thus $(p_Y)_*\cE_Y$ is still an instanton sheaf on $H\cong\p{n-1}$ with defect $\defect$ by Proposition \ref{pSpace}, hence $\cE_Y$ is an instanton sheaf on $Y$. Moreover $\Omega_{\p{n-1}}^1(1)\not\cong\Omega_{\p{n-1}}^{n-2}(n-2)$ if and only if $n\ge4$, hence the assertion on the quantum number of $\cE_Y$ follows.
\end{proof}

The restriction $n\ge3$ is sharp. Indeed if $\cE$ is an instanton sheaf on $\p2$ with defect $\defect$, rank $(1+\defect)a$ and quantum number $\quantum\ne \defect a$ its restriction to a line cannot be an instanton sheaf due to Remark \ref{rLine}.

\section{Resolutions of instanton sheaves\\ on embedded irreducible projective schemes}
\label{sResolution}
In this section we will assume that $X$ is an irreducible projective scheme endowed with a very ample line bundle $\cO_X(h)$. Let $V\subseteq H^0\big(\cO_X(h)\big)$ be a subspace of dimension $N+1$ associated to an embedding $X\subseteq\p N$. If $\cE$ is a $h$--instanton sheaf, in this section we will mainly deal with the free resolution of the module of global sections $H^0_*\big(\cE\big)$ as a module over the symmetric $\field$--algebra $S$ over $V$, when $\cE$ is an  instanton sheaf.

As a first step, we bound $\reg(\cE)$ from above in terms of cohomology of $\cE$: see also Remark \ref{rSpace1} for a confront of the bound below for non--ordinary instanton sheaves with the one in \cite[Theorem 3.2]{C--MR1}.

\begin{proposition}
\label{pRegularity}
Let $X$ be an irreducible projective scheme of dimension $n\ge2$ endowed with a very ample line bundle $\cO_X(h)$.

If $\cE$ is a $h$--instanton sheaf with defect $\defect$, then
$\reg(\cE)\le h^1\big(\cE((\defect-1)h)\big)+\defect$.
\end{proposition}
\begin{proof}
In what follows we explain the argument only for $\defect=1$, because in the case $\defect=0$ it is similar (and well--known: see  \cite[Corollary 3.3]{C--MR1}).

Notice that $h^i\big(\cE((1-i)h)\big)=0$ for $i\ge2$ by definition, hence it suffices to check that also $h^1\big(\cE((w-1)h)\big)=0$ for $w=h^1\big(\cE\big)+1$. If $p\colon X\to\p n$ is any finite morphism such that $\cO_X(h)\cong p^*\cO_{\p n}(1)$, then $h^1\big(\cE((w-1)h)\big)=h^1\big((p_*\cE)(w-1)\big)$, hence it suffices to deal with $X=\p n$, $\cO_X(h)= \cO_{\p n}(1)$ and $V=H^0\big(\cO_{\p n}(1)\big)$. 

Let $\quantum$ be the quantum number of $\cE$. Thanks to Proposition \ref{pSpace}, we know that $\cE$ is the cohomology of a monad $\cM^\bullet$ whose sheaves are as in equalities \eqref{MonadValue}. If the restriction of the map $\cM^{0}\to\cM^{1}$ to the direct summand $\cO_{\p n}^{\oplus b_0}$ is non--zero, then such a morphism necessarily splits, hence we can assume that such map vanishes. 

By splitting ${\cM}^\bullet$, we obtain sequences as in \eqref{DisplayM}. Their cohomologies return $h^1\big(\cE(-t)\big)=h^1\big(\cU(-t)\big)=0$ for $t\ge-1$ and
\begin{equation}
\label{seqLong}
\begin{aligned}
0\longrightarrow H^0_*\big({\cU}\big)&\longrightarrow S^{\oplus b_0}\oplus H^0_*\big(\Omega_{\p n}^1(1)\oplus \Omega_{\p n}^{n-1}(n-1)\big)^{\oplus \quantum}\oplus S(-1)^{\oplus b_1}\\
&\mapright B S^{\oplus g}\longrightarrow H^1_*\big({\cU}\big)\longrightarrow H^1_*\big(\Omega_{\p n}^1(1)\big)^{\oplus \quantum}\longrightarrow0
\end{aligned}
\end{equation}
where $g:=b_0-\chi(\cE)\le h^1\big(\cE\big)=w-1$ and the restriction of $B$ to $S^{\oplus b_0}$ is zero. The exact sequence
\begin{equation*}\label{Eulerpn}
0 \longrightarrow \Omega_{\p n}^p(p-1) \longrightarrow \cO_{\p n}^{{n+1}\choose{p}}(-1)\longrightarrow \Omega_{\p n}^{p-1}(p-1)\longrightarrow 0
\end{equation*}
with $p=n$ and $p=2$ yields surjective maps
$$
S(-1)^{\oplus n+1}\longrightarrow H^0_*\big( \Omega_{\p n}^{n-1}(n-1)\big), \qquad
S(-1)^{\oplus {{n+1\choose2}}}\longrightarrow H^0_*\big( \Omega_{\p n}^{1}(1)\big).
$$
Notice that  $H^1_*\big(\Omega_{\p n}^1(1)\big)$ coincides with $H^1\big(\Omega_{\p n}^1\big)$ concentrated in degree $-1$, hence the last module on the right in sequence \eqref{seqLong} is isomorphic to $\field^{\oplus \quantum}$ concentrated in degree $-1$. The lowest degree homogeneous part of $H^1_*\big({\cU}\big)$ is $H^1\big({\cU}(-1)\big)\cong\field^{\oplus \quantum}$ concentrated in degree $-1$. Thus the last non--zero morphism on the right in sequence \eqref{seqLong} splits and we finally obtain an exact sequence of the form
\begin{equation}
\label{seqEN}
S(-1)^{\oplus f} \mapright{A}S^{\oplus g}\longrightarrow \bigoplus_{t\ge0} H^1\big({\cU}(t)\big)\longrightarrow0,
\end{equation}
where $f:={{n+2\choose2}}+b_1$. 

Let $\cF:=\cO_{\p n}^{\oplus f}$, $\cG:=\cO_{\p n}(1)^{\oplus g}$. Thus we have an induced surjective morphism $\alpha\colon\cF\to\cG$, because the $S$--module on the right in sequence \eqref{seqEN} has finite length. The Eagon--Northcott complex
\begin{align*}
0\longrightarrow\wedge^f\cF\otimes S^{f-g}\cG^\vee&\mapright{\alpha_{f-g}}\wedge^{f-1}\cF\otimes S^{f-g-1}\cG^\vee\mapright{\alpha_{f-g-1}}\dots\\
&\mapright{\alpha_2}\wedge^{g+1}\cF\otimes \cG^\vee\mapright{\alpha_1}\wedge^g\cF\mapright{\wedge^g\alpha}\wedge^g\cG\longrightarrow0
\end{align*}
is exact thanks to \cite[Proposition 3]{Ea--No}. 
 
We have $\wedge^g\cG\cong\cO_{\p n}(g)$ and $\wedge^{g+i}\cF\otimes S^{i}\cG^\vee\cong\cO_{\p n}(-i)^{\oplus e_i}$ for some integers $e_i$. Thus $H^i(\ker \alpha_i)=0$ for all $i\ge1$, hence $H^0\big(\wedge^g\cF\big)\to H^0\big(\wedge^g\cG\big)$ is surjective. In particular, we deduce that the component $S_g\subseteq S$ of degree $g$ is generated by the maximal minors of the matrix of $A$ and it is not difficult to check that $\bigoplus_{t\ge g}S_t^{\oplus g}\subseteq S^{\oplus g}$ is contained in $\im(A)$. It follows that $h^1\big(\cE(g)\big)=h^1\big(\cU(g)\big)=0$, hence $\cE$ is certainly $(g+1)$--regular. Thus $\cE$ is also $w$--regular thanks to \cite[Corollary 4.18]{Ei} and the obvious inequality $w\ge g+1$.
\end{proof}

Recall that in the introduction we defined the numbers
\begin{gather*}
v(\cE):=\min\{\ t\in\bZ\ \vert\ h^0\big(\cE(th)\big)\ne0\ \},\qquad 
w(\cE):=h^1\big(\cE((\defect-1)h)\big)+\defect.
\end{gather*}
The following corollary is immediate by \cite[Corollary 4.18]{Ei}.

\begin{corollary}
\label{cRegularity}
Let $X$ be an irreducible projective scheme of dimension $n\ge2$ endowed with a very ample line bundle $\cO_X(h)$.

If $\cE$ is a $h$--instanton sheaf, then $\cE(w(\cE))$ is globally generated. In particular, if $\cE$ is ordinary and its quantum number is $\quantum$, then $\cE(\quantum)$ is globally generated.
\end{corollary}

We prove Theorem \ref{tResolution} below.

\medbreak
\noindent{\it Proof of Theorem \ref{tResolution}.}
Let $e_{0,1},\dots,e_{0,m_0}\in E_{0}:=H^0_*\big(\cE\big)$ be a minimal set of generators as $S$--module such that the sequence $\deg(e_{0,j})=\nu_{0,j}$ is non--decreasing. 
There is a surjective morphism $F_0:=\bigoplus_{j=1}^{m_0}S(-\nu_{0,j})\to E_0$: by definition we have $\nu_{0,1}\ge v(\cE)\ge0$ and Corollary \ref{cRegularity} implies $\nu_{0,m_0}\le  w(\cE)$.  

By sheafification, we obtain a short exact sequence 
$$
0\longrightarrow \cE_1\longrightarrow \cF_{0}\mapright{\varphi_{0}} \cE\longrightarrow0.
$$
The definition of $\varphi_0$ yields $H^1_*\big(\cE_1\big)=0$. Moreover, the cohomology of the above sequence twisted by $\cO_{\p N}(v(\cE))$ implies $h^0\big(\cE(v(\cE))\big)=0$, hence the minimal generators of $E_{1}:=H^0_*\big(\cE_1\big)$ have degree $\nu_{0,1}+1\ge v(\cE)+1\ge1$ at least. The cohomology of the above sequence with suitable twists also yields $\reg(\cE_1)\le w(\cE)+1$. 

By induction, the same argument leads for $0\le p\le N-1$ to exact sequences 
\begin{equation}
\label{seqE}
0\longrightarrow \cE_p\longrightarrow \cF_{p-1}\mapright{\varphi_{p-1}} \cE_{p-1}\longrightarrow0
\end{equation}
such that $\cF_p\cong \bigoplus_{j= v(\cE)}^{w(\cE)}\cO_{\p N}(-j-p)^{\oplus \beta_{p,j}}$ and $\varphi_{p-1}$ is surjective on global sections.

It follows that $H^1_*\big(\cE_p\big)=0$ for $1\le p\le N-1$. Moreover, the cohomologies of sequences \eqref{seqE} yield $H^i_*\big(\cE_p\big)=H^{i-1}_*\big(\cE_{p-1}\big)$ for $2\le i\le N-1$, hence $H^i_*\big(\cE_{N-1}\big)=H^1_*\big(\cE_{N-i}\big)=0$ for $1\le i\le N-1$. The Horrocks theorem implies that $\cE_{N-1}$ splits as a sum of line bundles. 

Since $H^1_*\big(\cE_p\big)=0$ for $1\le p\le N-1$, it follows that we can glue together the twisted cohomologies of sequences \eqref{seqE} obtaining sequence \eqref{Resolution}.
\qed
\medbreak

\begin{corollary}
\label{cResolutionSheaf}
Let $X$ be an irreducible projective scheme of dimension $n\ge1$ endowed with a very ample line bundle $\cO_X(h)$. Let $V\subseteq H^0\big(\cO_X(h)\big)$ be a subspace associated to an embedding $X\subseteq\p N$.

If $\cE$ is a $h$--instanton sheaf, then there exists an exact sequence of the form
\begin{equation}
\label{ResolutionSheaf}
0\longrightarrow \cF_{N}\longrightarrow \cF_{N-1}\longrightarrow\dots\longrightarrow \cF_1\longrightarrow \cF_0\longrightarrow \cE\longrightarrow0,
\end{equation}
where $\cF_p\cong \cO_{\p N}(-w(\cE)-p)^{\oplus \beta_{p}}$.
\end{corollary}
\begin{proof}
Each hyperplane not passing through any of the associated points of $X$ corresponds to a section in $V$ which is not a zero--divisor in ${H^0_*\big(\cE\big)}$, hence in ${\bigoplus_{t\ge w(\cE)}H^0\big(\cE(th)\big)}$. In particular, the depth of the latter is at least $1$, hence the statement follows from \cite[Theorem 1.2 (1)]{Ei--Go}.
\end{proof}

\begin{remark}
\label{rResolution}
If $\cE$ is a $h$--instanton sheaf with resolution \eqref{Resolution}, then Theorem \ref{tResolution} and Corollary \ref{cCharacteristic} yield linear equations in the $\beta_{p,i}$'s. Similar constraints can be obtained for the $\beta_p$'s in resolution \ref{ResolutionSheaf}.\end{remark}

Let $\cE$ be a $h$--instanton sheaf with natural cohomology in positive degrees. The following corollary generalizes the results in \cite{Rah1,Rah2}.

\begin{corollary}
\label{cResolution}
Let $X$ be an irreducible projective scheme of dimension $n\ge1$ endowed with a very ample line bundle $\cO_X(h)$. Let $V\subseteq H^0\big(\cO_X(h)\big)$ be a subspace associated to an embedding $X\subseteq\p N$ and $S$ the symmetric $\field$--algebra of $V$.

If $\cE$ is a $h$--instanton sheaf with natural cohomology in positive degrees, then the minimal free resolution of $H^0_*\big(\cE\big)$ as $S$--module has the form
\begin{equation*}
0\longrightarrow F_{N-1}\longrightarrow F_{N-2}\longrightarrow\dots\longrightarrow F_1\longrightarrow F_0\longrightarrow H^0_*\big(\cE\big)\longrightarrow0,
\end{equation*}
where $F_p\cong S(-v(\cE)-p)^{\oplus \beta_{p,0}}\oplus S(-v(\cE)-p-1)^{\oplus \beta_{p,1}}$.
\end{corollary}
\begin{proof}
The statement follows from the inequality $\reg(\cE)\le v(\cE)+1$, because we know from Theorem \ref{tResolution} that the projective dimension of $H^0_*\big(\cE\big)$ over $S$ is $N-1$ at most. 
\end{proof}

\section{Instanton bundles on $n$--folds}
\label{sBundle}
In this section we focus our attention on instanton bundles supported on an $n$--fold $X$, i.e. a smooth variety of dimension $n$. In particular $\omega_X\cong\cO_X(K_X)$ is a line bundle. 

Notice that each finite morphism $p\colon X\to \p 1$ is flat, thanks to \cite[Exercise III.9.3]{Ha2}. If  $\cO_X(h)\cong p^*\cO_{\p 1}(1)$, then the degree of $p$ is $h^n$. Thus $\rk(p_*\cE)=\rk(\cE)h^n$ for each vector bundle $\cE$ on $X$.

We start by studying the case of smooth curves integrating Remark \ref{rLine}.

\begin{proposition}
\label{pCurve}
Let $X$ be a smooth curve endowed with an ample and globally generated line bundle $\cO_X(h)$. Assume that the characteristic of $\field$ is not $2$.

The following assertions hold.
\begin{enumerate}
\item Every instanton sheaf is a vector bundle.
\item $X$ supports ordinary instanton sheaves of each rank.
\item $X$ supports  non--ordinary instanton sheaves of each even rank.
\item The quantum number of $\cE$ is $\frac\defect2\rk(\cE)\deg(X)$.
\end{enumerate}
\end{proposition}
\begin{proof}
Assertion (1) can be proved as the analogous assertion in Remark \ref{rLine}.

Let $g$ be the genus of $X$. If $\theta$ is a non--effective theta--characteristic on $X$, then it is easy to check that $\cO_X(\theta+h)^{\oplus a}$ and $(\cO_X(\theta)\oplus\cO_X(\theta+h))^{\oplus a}$ are respectively an ordinary and a non--ordinary instanton sheaf. Thus assertions (2) and (3) are proved.

If $p\colon X\to \p 1$ is any finite morphism, then $p_*\cE$ is an $\cO_{\p 1}(1)$--instanton bundle by Theorem \ref{tCharacterization} with $\rk(p_*\cE)=\rk(\cE)\deg(X)$. Thanks to Remark \ref{rLine}, its quantum number is $\defect\rk(p_*\cE)/2$. Thus the quantum number of $\cE$ is as in the assertion (4).
\end{proof}

Though only smooth curves of even degree can support non--ordinary $h$--instanton line bundles (and, more generally, of odd rank), the following example shows that each smooth curve can be endowed with many very ample line bundles in such a way it supports non--ordinary $h$--instanton line bundles.

\begin{example}
\label{eCurve}
Assume $X$ is a smooth curve and let $g$ be its genus. Let $\theta\in\Pic^{g-1}(X)$ be any non--effective theta--characteristic and $\cO_X(A)$ be any effective ample line bundle such that $\cO_X(h):=\cO_X(2A)$ is very ample. Thus $\cO_X(\theta+A)$ is a non--ordinary $h$--instanton line bundle.
\end{example}

Our first general result in this section is that instanton bundles with positive quantum number on an $n$--fold must have sufficiently large rank.

\begin{proposition}
\label{pHorrocks} 
Let $X$ be an $n$--fold with $n\ge4$ endowed with an ample and globally generated line bundle $\cO_X(h)$.

If $\cE$ is a  $h$--instanton bundle with quantum number $\quantum$ on $X$, then the following assertions hold.
\begin{enumerate}
\item If $\rk(\cE)h^n<2\left[\frac n2\right]$, then $\cE$ is aCM. In particular, $\quantum=0$.
\item If $\cE$ is ordinary, and $\quantum\ge1$, then $\rk(\cE)h^n\ge n-1$. 
\item If $\cE$ is ordinary, $\rk(\cE)h^n= n-1$ and $n$ is even, then $\cE$ is Ulrich.
\end{enumerate}
\end{proposition}
\begin{proof}
Let $p\colon X\to \p n$ be any finite morphism with $\cO_X(h)\cong p^*\cO_{\p n}(1)$, then $p_*\cE$ is a $\cO_{\p n}(1)$--instanton  bundle on $\p n$ with $\rk(p_*\cE)=\rk(\cE)h^n$. 

Theorem \ref{tCharacterization} yields $H^i_*\big(p_*\cE\big)=0$ for $2\le i\le n-2$, hence \cite[Theorem 1]{MK--P--R} implies that $H^i_*\big(\cE\big)=H^i_*\big(p_*\cE\big)=0$ for $1\le i\le n-1$. The vanishings follow from Corollary \ref{cUlrich}. Thus assertion (1) holds.

If $\cE$ is ordinary, then the same holds for $p_*\cE$. Thus $p_*\cE$ is the cohomology of a monad $\cM^\bullet$ as in Proposition \eqref{pSpace} and assertion (2) follows from \cite[Main Theorem]{Flo}. Assertion (3) follows by combining assertion (1) with Corollary \ref{cUlrich}.
\end{proof}

For instanton bundles the following partial converse of Corollary \ref{cRestriction} holds. 

\begin{proposition}
\label{pExtension}
Let $X$ be an $n$--fold with $n\ge5$ endowed with an ample and globally generated line bundle $\cO_X(h)$ and let $Y\in\vert h\vert$ be an $(n-1)$--fold. 

If $\cE$ is a vector bundle on $X$ such that $\cO_Y\otimes\cE$ is an instanton bundle on $Y$ with respect to $\cO_Y\otimes\cO_X(h)$, then $\cE$ is an instanton bundle with respect to $\cO_X(h)$. Moreover, the defects and the quantum numbers of $\cE$ and $\cO_Y\otimes\cE$ coincide. 
\end{proposition}
\begin{proof}
Let $\cE_Y:=\cO_Y\otimes\cE$, $\cO_Y(h_Y):=\cO_Y\otimes\cO_X(h)$ and denote by $\defect$ and $\quantum$ the defect and the quantum number of $\cE_Y$.

We have $h^i\big(\cE(-(1+t)h)\big)\le h^i\big(\cE(-(2+t)h)\big)$ for $t\ge i$ and $i\le n-3$ by computing the cohomology of sequence \eqref{seqSection} tensored by $\cE(th)$, hence $h^i\big(\cE(-(1+t)h)\big)=0$ in the same range by \cite[Theorem III.5.2]{Ha2}. If $2\le i\le n-3$, then $h^i\big(\cE_Y(-ih_Y)\big)=0$, hence the cohomology of sequence \eqref{seqSection} tensored by $\cE(-ih)$ combined with the same argument used above implies $h^i\big(\cE(-ih)\big)=0$ in that range. 

The same argument for $\cE_Y^\vee((\defect-n) h_Y+K_Y)$ yields $h^i\big(\cE((\defect-i-t+t)h)\big)=0$ for $i\ge2$ and $t\ge 1$ and $h^{n-2}\big(\cE(-(n-2)h)\big)=0$. 

Since $n\ge 5$, it follows that $h^i\big(\cE(-th)\big)=0$ for $i\in \bZ$ and $2\le t\le \defect+1-n$. Thus the cohomology of sequence \eqref{seqSection} tensored by $\cE(-h)$ and $\cE((\defect+1-n)h)$ and the equality $h^1\big(\cE_Y(-h_Y)\big)=h^{n-2}\big(\cE_Y((\defect+1-n)h_Y)\big)$ finally yield
$$
h^1\big(\cE(-h)\big)=h^{n-1}\big(\cE((\defect-n)h)\big).
$$
The same argument yields the equality $\chi(\cE)=(-1)^n\chi(\cE(-nh))$ if $\defect=1$.
\end{proof}

The main result of this section is the proof of  Theorem \ref{tSlope}.

\medbreak
\noindent{\it Proof of Theorem \ref{tSlope}.}
Thanks to Theorem \ref{tCharacterization}, it suffices to check that $\cE$ also satisfies the additional condition
\begin{equation}
\label{A}
(1-\defect)(\chi(\cE(-h))-(-1)^n\chi(\cE((\defect-n)h)))+\defect(\chi(\cE)-(-1)^n\chi(\cE(-nh)))=0
\end{equation}
if and only if equality \eqref{Slope} holds. 

In what follows we will provide a proof of the statement in the case $\defect=1$. The argument in the case $\defect=0$ is analogous. Thus, from now on, we will assume that the following conditions hold:
\begin{itemize}
\item $h^0\big(\cE(-h)\big)=h^n\big(\cE((1-n)h)\big)=0$;
\item $h^i\big(\cE(-(i+1)h)\big)=h^{n-i}\big(\cE((1-n+i)h)\big)=0$ if $1\le i\le n-2$;
\item $h^i\big(\cE(-ih)\big)=0$ for $2\le i\le n-2$;
\item $ h^1\big(\cE(-h)\big)= h^{n-1}\big(\cE((1-n)h)\big)$.
\end{itemize}

We first prove the statement for $1\le n\le2$ and then by induction on $n\ge2$. If $n=2$, equality \eqref{A} becomes $\chi(\cE)=\chi(\cE(-2h))$. The statement then follows by computing the two sides of this identity by means of equality \eqref{RRsurface}. If $n=1$ we can argue similarly by using equality \eqref{RRcurve} instead of \eqref{RRsurface} (see also  \cite[Lemma 2.4]{C--H2}).

Let $n\ge3$. Each general $Y\in\vert h\vert$ is an $(n-1)$--fold by the Bertini theorem (see \cite[Corollaire I.6.11 2), 3)]{Jou}) and we set $\cO_Y(h_Y):=\cO_Y\otimes\cO_X(h)$, $\cE_Y:=\cO_Y\otimes\cE$. Assume that the statement holds on such a $Y$. Thus
$$
(n+1-\defect)h^n+K_Xh^{n-1}=(n-\defect)h^{n-1}_Y+K_Yh^{n-2}_Y
$$
thanks to the adjunction formula on $X$. It follows that the equality \eqref{Slope} holds for the bundle $\cE$ on $X$ if and only if it holds for the bundle $\cE_Y$ on $Y$. Moreover, the cohomology of sequence \eqref{seqSection} tensored by the shifts of $\cE$ yields that $\cE_Y$ satisfies a list of condition similar to the one for $\cE$

If $\cE$ is a non--ordinary $h$--instanton bundle, then $\cE_Y$ is a non--ordinary $h_Y$--instanton bundle thanks to Corollary \ref{cRestriction}, hence equality \eqref{Slope} holds for $\cE_Y$ by induction. It follows that equality  \eqref{Slope} holds for $\cE$ too. 

Conversely, let $\cE$ satisfy equality \eqref{Slope}. By induction we know that $\cE_Y$ is a $h_Y$--instanton bundle on $Y$, hence $ \chi(\cE_Y)= \chi(\cE_Y((1-n)h_Y)\big)$. Thus, tensoring sequence \eqref{seqSection} by $\cE$ and $\cE((1-n)h)$, we obtain
$$
\chi(\cE)-(-1)^n\chi(\cE(-nh))=\chi(\cE(-h))-(-1)^{n-1}\chi(\cE((1-n)h))=0
$$
thanks to Proposition \ref{pQuantum}. Thus $\cE$ is a $h$--instanton sheaf on $X$.
\qed
\medbreak

\begin{remark}
\label{rSlope}
Let $\cE$ be a $h$--instanton bundle on an $n$--fold $X$ and consider any finite morphism $p\colon X\to \p n$ such that $\cO_X(h)\cong p^*\cO_{\p n}(1)$.

If either $\cO_X(h)$ is very ample or the characteristic of $\field$ is $0$, then the Bertini theorem (see \cite[Corollaire I.6.11 2), 3)]{Jou}) implies that $S:=p^{-1}(H)$ is a smooth surface such that
$$
\chi(\cO_S)=\sum_{i=0}^{n-2}(-1)^i{{n-2}\choose i}\chi(\cO_X(-ih))
$$
when $\p2\cong H\subseteq \p n$ is a general linear subspace. Since $X$ is smooth, it follows that $p$ is flat, hence $p_*\cE$ is locally free of rank $\rk(p_*\cE)=\rk(\cE)h^n$. Thus Remark \ref{rRankChern} yields
$$
\chi(\cO_S\otimes\cE)=\frac{\rk(\cE)}{1+\defect}h^n-\epsilon k
$$
where $\epsilon$ is as in Remark \ref{rChernPoly}. Equality \eqref{RRsurface} for $\cO_S\otimes\cE$ and the adjunction formula on $H$ then imply
\begin{equation*}
\begin{aligned}
\epsilon\quantum&=\left(c_2(\cE)-\frac12c_1(\cE)(c_1(\cE)-K_X-(n-2)h)\right)h^{n-2}\\
&+\rk(\cE)\left(\frac {1}{1+\defect}h^n-\sum_{i=0}^{n-2}(-1)^i{{n-2}\choose i}\chi(\cO_X(-ih))\right).
\end{aligned}
\end{equation*}
When $X\cong\p n$ the above  coincides with the second equality \eqref{ChernP^n}.

Similarly, one can obtain equalities involving the quantum number and the first $m$ Chern classes of each $h$--instanton sheaf on $X$ when $n\ge m$.
\end{remark}

If $\cF$ is a vector bundle on $X$, we define its {\sl Ulrich dual (with respect to $\cO_X(h)$)} as $\cF^{U,h}:=\cF^\vee((n+1)h+K_X)$: trivially $(\cF^{U,h})^{U,h}\cong\cF$. Moreover, instanton bundles are preserved by Ulrich duality.

\begin{proposition}
\label{pDual}
Let $X$ be an $n$--fold with $n\ge1$ endowed with an ample and globally generated  line bundle $\cO_X(h)$. 

If $\cE$ is a vector bundle on $X$, then $\cE$ is a $h$--instanton bundle with defect $\defect$ and quantum number $\quantum$ if and only if the same is true for $\cE^{U,h}(-\defect h)$. 
\end{proposition}
\begin{proof}
The statement follows from equality \eqref{Serre} because $\cE$ is a vector bundle.
\end{proof}

Recall that a rank two $h$--instanton bundle is called orientable if it has  defect $\defect$ and $c_1(\cE)=(n+1-\defect)h+K_X$ in $A^1(X)\cong\Pic(X)$.

\begin{proposition}
\label{pSpecial}
Let $X$ be an $n$--fold with $n\ge1$ endowed with an ample and globally generated  line bundle $\cO_X(h)$. 

The rank two vector bundle $\cE$ on $X$ is an orientable $h$--instanton bundle with defect $\defect$ if and only if the following conditions hold:
\begin{itemize}
\item $c_1(\cE)=(n+1-\defect)h+K_X$;
\item $h^0\big(\cE(-h)\big)=0$;
\item $h^i\big(\cE(-(i+1)h)\big)=0$ if $1\le i\le n-2$;
\item $\defect h^i\big(\cE(-ih)\big)=0$ for $2\le i\le n-2$.
\end{itemize}
\end{proposition}
\begin{proof}
The statement follows by combining the definition of rank two orientable $h$--instanton bundle, equality  \eqref{Serre} and Theorem \ref{tSlope}.
\end{proof}

\begin{remark}
\label{rAM2}
When $n=3$, it follows from the above proposition that $\cF$ is a $h$--instanton bundle as defined in \cite{A--M2} if and only if $\cE:=\cF(h)$ is a rank two orientable, $\mu$--semistable, ordinary $h$--instanton bundle.
\end{remark}

The proof of the following corollary is immediate.

\begin{corollary}
\label{cSpecial}
Let $X$ be an $n$--fold with $1\le n\le 2$ endowed with an ample and globally generated  line bundle $\cO_X(h)$. 

The rank two vector bundle $\cE$ on $X$ is an orientable $h$--instanton bundle with defect $\defect$ if and only if $c_1(\cE)=(n+1-\defect)h+K_X$ and $h^0\big(\cE(-h)\big)=0$.
\end{corollary}

It is immediate to check that if $\theta$ is a non--effective theta--characteristic on a smooth curve $X$, then $\cO_X(\theta+(1-\defect)h)\oplus\cO_X(\theta+h)$ is a rank two orientable  $h$--instanton bundle with defect $\defect$.

In \cite[Introduction]{K--M--S} the notion of $\delta$--Ulrich sheaf on a projective scheme $X$ is defined, proving that if $X$ is a normal surface such sheaves are instanton sheaves and exist if $X$ is aCM (see \cite[Proposition 5.1 and Theorem A]{K--M--S}). Nevertheless, the existence of $\delta$--Ulrich sheaves on each smooth surface and their admissible ranks are actually open problems. 

In the following examples we show that smooth surfaces always support rank two orientable  $h$--instanton bundles with large enough quantum number. If $\cE$ is such a bundle, then its intermediate cohomology does not necessarily vanish, but Propositions \ref{pRegularity}, \ref{pNatural} and equality \eqref{Serre} guarantee $h^1\big(\cE(th)\big)=0$ unless possibly if $\defect-1-w(\cE)\le t\le w(\cE)-2$.

\begin{example}
\label{eMukai}
Let $X$ be a smooth surface endowed with a very ample line  bundle $\cO_X(h)$: the construction below is the same as in \cite[Proposition 6.2]{E--S--W}.

The linear system $\vert (3-\defect)h+K_X\vert$ contains a smooth curve $C$ if $\defect\in\{\ 0,1\ \}$, $\kappa(X)\ge0$ and $\field=\bC$ (see \cite[Theorem 0.1]{So--VV} and \cite[Exercise II.7.5 (d)]{Ha2}). If $\cO_C(D)$ is globally generated and $\sigma_0,\sigma_1\in H^0\big(\cO_C(D)\big)$ are two sections without common zeros, then we have a surjective morphism $\sigma\colon\cO_X^2\to\cO_C(D)$ and we set $\cE:=\ker(\sigma)^\vee$. Thus there is the exact sequence
$$
0\longrightarrow\cE^\vee\longrightarrow\cO_X^2\mapright\sigma\cO_C(D)\longrightarrow0.
$$
If we denote by $i\colon C\to X$ the inclusion, then  the multiplicativity of the Chern polynomials in short exact sequences yields
$$
c_1(\cE)=c_1(i_*\cO_C(D)), \qquad c_2(\cE)=c_1(i_*\cO_C(D))^2-c_2(i_*\cO_C(D)).
$$
Thus  $c_1(\cE)=(3-\defect)h+K_X$. Moreover $\chi(i_*\cO_C(D))=\chi(\cO_C(D))$, hence $c_2(\cE)=\deg(D)$ thanks to equalities \eqref{RRcurve} and \eqref{RRsurface}.

The cohomology of the above sequence tensored by $\cO_X(h+K_X)$, equality \eqref{Serre} and the Kodaira vanishing theorem imply 
$$
h^0\big(\cE(-h)\big)=h^1\big(\cO_X(h+K_X)\otimes\cO_C(D)\big).
$$
Thus, if $h^1\big(\cO_X(h+K_X)\otimes\cO_C(D)\big)=0$, then $\cE$ is a rank two orientable $h$--instanton bundle with defect $\defect$, thanks to Corollary \ref{cSpecial}. In this case Equality \eqref{RRsurface} and Proposition \ref{pQuantum} yield that the quantum number of $\cE$ is 
\begin{equation*}
\begin{aligned}
\quantum&=-\chi(\cE(-h))=\deg(D)-2\chi(\cO_X)-\frac12\left((\defect^2-4\defect+5)h^2+(3-\defect)K_Xh\right)
\end{aligned}
\end{equation*}
\end{example}

\begin{example}
\label{eGenus0}
In the above example we assumed that $\omega_X$ is positive enough: in the following example we deal with any smooth surface $X$ with $p_g(X)=0$ and endowed with a very ample line bundle $\cO_X(h)$ such that $h^0\big(\cO_X(h)\big)=N+1$: the construction here is completely analogous to the one described in \cite{Cs4}.

Choose a $0$--dimensional subscheme  $Z\subseteq X$ of degree $z\ge (1-\defect)(N+1)+1$ with $\defect\in\{\ 0,1\ \}$: if $\defect=0$ we also assume that no subscheme $Z'\subseteq Z$ of degree $N+1$ is contained in any divisor in $\vert h\vert$. Thus, the Cayley--Bacharach theorem (see \cite[Theorem 5.1.1]{Hu--Le}) yields the existence of a rank two vector bundle $\cF$ fitting into the exact sequence
\begin{equation}
\label{Genus0}
0\longrightarrow\cO_X\longrightarrow\cF\longrightarrow\cI_{Z\vert X}((1-\defect) h-K_X)\longrightarrow0.
\end{equation}
Thus, the vector bundle $\cE:=\cF(h+K_X)$ satisfies $c_1(\cE)=(3-\defect)h+K_X$ and $h^0\big(\cE(-h)\big)=p_g(X)=0$, hence $\cE$ is a rank two orientable $h$--instanton bundle on $X$. In particular $h^2\big(\cE(-h)\big)=0$, hence the cohomologies of sequences \eqref{Genus0} tensored by $\omega_X$ and \eqref{seqStandard} by $\cO_X((1-\defect)h)$ imply that its quantum number is 
$$
\quantum=h^1\big(\cE(-h)\big)=z+(1+\defect)(q(X)-1)+(1-\defect)(h^1\big(\cO_X(h)\big)-N-1).
$$
\end{example}

\begin{remark}
The case of smooth surfaces $X\subseteq\p3$ has a particular interest. In this case, if $\field=\bC$, then rank two, orientable, Ulrich bundles exist on the very general  $X$ if and only if $\deg(X)\le 15$: see \cite{Bea, Fa1}. Moreover, they exist on each $X$ when $\deg(X)\le4$: see \cite{Bea} for $\deg(X)\le 3$ and \cite{C--K--M, C--N} for $\deg(X)=4$.
\end{remark}

We close this section by generalizing the tight relation between ordinary $h$--instanton and Ulrich bundles with respect to $\cO_X(dh)$ evidenced for the first time in \cite{C--MR8} when $X\cong\p3$.

\begin{corollary}
\label{cVeronese} 
Let $X$ be an $n$--fold with $1\le n\le 3$ endowed with an ample and globally generated line bundle $\cO_X(h)$. 

For a vector bundle $\cE$ on $X$, the following assertions are equivalent for every positive integer $d$ such that $(n+1)(d-1)$ is even.
\begin{enumerate}
\item $\cE\left(\frac{n+1}2(d-1)h\right)$ is an Ulrich bundle with respect to $\cO_X(dh)$.
\item $\cE$ is an  ordinary $h$--instanton bundle with natural cohomology in each shift and quantum number
\begin{equation}
\label{Quantum}
\quantum=\frac{(n-1)^n\rk(\cE)(d^2-1)}{2^nn!}h^n.
\end{equation}
\end{enumerate}
\end{corollary}
\begin{proof}
Let $p\colon X\to\p n$ be a finite morphism such that $\cO_X(h)\cong p^*\cO_{\p n}(1)$. 

If assertion (1) holds, then $(p_*\cE)\left(\frac{n+1}2(d-1)\right)$ is an Ulrich bundle with respect to $\cO_{\p n}(d)$, thanks to Theorem \ref{tCharacterization} and Corollary \ref{cUlrich}. Moreover, $\rk(p_*\cE)=\rk(\cE)h^n$ hence \cite[Theorem 5.1]{E--S--W} and \cite[Lemma 2.6]{C--H2} imply that $p_*\cE$ has natural cohomology,
\begin{equation}
\label{QuantumChar}
\chi\left((p_*\cE)\left(\frac{n+1}2(d-1)+t\right)\right)=\rk(\cE)d^nh^n{{\frac td+n}\choose n},
\end{equation}
$h^0\big((p_*\cE)(-1)\big)=h^n\big((p_*\cE)(-n)\big)=0$ and, when $n=3$, also $h^1\big((p_*\cE)(-2)\big)=h^2\big((p_*\cE)(-2)\big)=0$. 
We deduce that $p_*\cE$ is an ordinary $\cO_{\p n}(1)$--instanton sheaf and equality \eqref{Quantum} is obtained by taking $t=-1-\frac{n+1}2(d-1)$ in formula \eqref{QuantumChar}. Thus assertion (2) follows from Theorem \ref{tCharacterization}.

Let assertion (2) hold. Computing $\chi\left((p_*\cE)\left(\frac{n+1}2(d-1)+td\right)\right)$ via Corollary \ref{cCharacteristic}, taking into account equality \eqref{Quantum} and that $\rk(\cE)h^n=\chi(\cE)+(n-1)\quantum$ (see Remark \ref{rRankChern}) one obtains
$$
\chi\left((p_*\cE)\left(\frac{n+1}2(d-1)-jd\right)\right)=0
$$
for $1\le j\le n$. It follows from the naturality of the cohomology of $\cE$ that the cohomology of  $(p_*\cE)\left(\frac{n+1}2(d-1)\right)$ vanishes in the range $-n\le t\le -1$, hence it is an Ulrich sheaf with respect to $\cO_{\p n}(d)$. 
\end{proof}

\begin{example}
\label{eVeronesePlane}
Thanks to \cite[Theorem 3]{C--MR7} there exist indecomposable Ulrich bundles with respect to $\cO_{\p2}(d)$ for each $r\ge2$ and $d\ge3$ such that $r(d-1)$ is even. It follows that, in the same range, there exist indecomposable ordinary $\cO_{\p2}(1)$--instanton bundles with quantum number $\quantum =\frac{r(d^2-1)}{8}$ and natural cohomology.
\end{example}

\begin{example}
\label{eVeroneseQuadric}
Let $X\subseteq\p4$ be a smooth quadric hypersurface and $\cO_X(h):=\cO_X\otimes\cO_{\p4}(1)$. 

Let $\cM_X(2;h,3)$ be the moduli space of stable rank two bundles $\cE$ with $c_1(\cE)=h$ and $c_2(\cE)h=3$ on $X$. The scheme $\cM_X(2;h,3)$ is irreducible, unirational, reduced of dimension $12$ and its general point represents a bundle $\cE$ with $h^0\big(\cE(h)\big)=0$ and $\reg(\cE)=1$ (see \cite[Theorem 5.2 and its proof]{Ot--Sz}), hence $h^1\big(\cE(th)\big)=0$ for $t\ge0$ and $h^0\big(\cE(th)\big)=0$ for $t\le 1$. Moreover, \cite[Corollary 2.4]{Ein--So} implies $h^1\big(\cE(-2h)\big)=0$. 

It follows that $h^3\big(\cE(-5h)\big)=h^2\big(\cE(-2h)\big)=0$, thanks to equality \eqref{Serre}, hence $h^2\big(\cE(th)\big)=0$ for $t\le-4$ and $h^3\big(\cE(th)\big)=0$ for $t\ge-5$. Thus $\cE$ is a rank two, ordinary,  $h$--instanton bundle with natural cohomology in each shift. Equality \eqref{RRgeneral} finally implies that its quantum number is $\quantum=2$. 

Corollary \ref{cVeronese} then implies that $\cF:=\cE(2h)$ is a rank two Ulrich bundle on $X$ with respect to $\cO_X(2h)$.
\end{example}

\section{Monadic representation of instanton bundles on aCM $n$--folds}
\label{sMonad}
As pointed out in Theorem \ref{tCharacterization}, instanton sheaves on a projective scheme $X$ endowed with a very ample line bundle $\cO_X(h)$ are exactly the sheaves whose direct images via a suitable finite morphism on $\p n$ are $\cO_{\p n}(1)$--instanton sheaves or, equivalently, the cohomology of the monads $\cM^\bullet$  described in  Proposition \ref{pSpace}. 

Several authors used the property of being cohomology of monads of a certain fixed shape for defining instanton bundles on $n$--folds: e.g. see \cite{J--MR}. The property of being cohomology of a linear monad with a certain fixed shape does not characterize $h$--instanton bundles when $X\not\cong\p n$. 

Nevertheless, $h$--instanton bundles can be associated to some particular monads making use of the results in \cite[Sections 2 and 3]{J--VM} as claimed in Theorem \ref{tMonad}. 

Recall that $\cF^{U,h}:=\cF^\vee((n+1)h+K_X)$ for each vector bundle $\cF$ on $X$.

\medbreak
\noindent{\it Proof of Theorem \ref{tMonad}.}
If $\cE$ is a $h$--instanton bundle, then $h^p\big(\cE(\defect-p)\big)=0$ for each $p\ge2$ because $n\ge3$, hence $H^1_*\big(\cE\big)$ is generated by $\bigoplus_{i=0}^\defect H^1\big(\cE((i-1)h)\big)$ as an $S[X]$--module by \cite[Lemma 3.4]{J--VM}. Similarly, $H^1_*\big(\cE^{U,h}(-\defect h)\big)$ is generated by $\bigoplus_{i=0}^\defect H^1\big(\cE^{U,h}((i-1-\defect)h)\big)$ thanks to Proposition \ref{pDual}. 

Thanks to \cite[Theorem 2.3]{J--VM} we know the existence of monad \eqref{Monad} such that $H^i_*\big(\cB\big)=0$ for $1\le i\le n-1$ where $a$ and $c$ are as in equalities \eqref{DimensionMonad}. Monad \eqref{Monad} induces the two exact sequences
\begin{equation}
\label{Display}
\begin{gathered}
0\longrightarrow\cU\longrightarrow\cB\longrightarrow\cO_X^{\oplus c}\oplus \cO_X(h)^{\oplus \quantum}\longrightarrow0,\\
0\longrightarrow\omega_X((n-\defect)h)^{\oplus \quantum}\oplus\omega_X((n+1-\defect)h)^{\oplus a}\longrightarrow\cU\longrightarrow\cE\longrightarrow0.
\end{gathered}
\end{equation}
The cohomology of sequences \eqref{Display} tensored by $\cO_X(-h)$ and $\cO_X((\defect-n)h)$ and equality \eqref{Serre} yield equalities \eqref{h^iB} and \eqref{ChiB}. 

Conversely, let $\cE$ be a vector bundle satisfying $h^0\big(\cE(-h)\big)=h^n\big(\cE((\defect-n)h)\big)=0$. Assume that $\cE$ is the cohomology of monad \eqref{Monad}, where $\cB$ is aCM and satisfies equalities \eqref{h^iB} and \eqref{ChiB}, and $a$, $c$ are as in equalities \eqref{DimensionMonad}. By hypothesis $X$ is aCM, hence  $H^i_*\big(\omega_X\big)=H^i_*\big(\cO_X\big)=0$ for $1\le i\le n-1$ thanks to equality \eqref{Serre}. Thus the cohomology of sequences \eqref{Display}  twisted by $\cO_X(-2h)$ yields
\begin{gather*}
h^1\big(\cE(-2h)\big)=h^1\big(\cU(-2h)\big)\le h^0\big(\cC(-2h)\big)+h^1\big(\cB(-2h)\big)=0.
\end{gather*}
Similarly $h^{n-1}\big(\cE((\defect+1-n)h)\big)=0$ and $H^i_*\big(\cE\big)=0$ for $2\le i\le n-2$. The same argument yields $h^1\big(\cE(-h)\big)=h^{n-1}\big(\cE((\defect-n)h)\big)=\quantum$ and $\defect(\chi(\cE)-(-1)^n\chi(\cE(-nh)))=0$ thanks to equalities \eqref{h^iB} and \eqref{ChiB}. Thus $\cE$ is an instanton bundle.
\qed
\medbreak

\begin{remark}
\label{rMonad}
Assume $h^0\big(\omega_X((n-1)h)\big)=0$. If $\cE$ is the cohomology of monad \eqref{Monad}, then the cohomology of sequences \eqref{Display} implies
\begin{gather*}
h^0\big(\cE(-h)\big)\le h^0\big(\cU(-h)\big)\le h^0\big(\cB(-h)\big)=0,\\
h^n\big(\cE((\defect-n)h)\big)\le h^0\big(\cU((\defect-n)h)\big)\le h^0\big(\cB((\defect-n)h)\big)=0,
\end{gather*}
thanks to equalities \eqref{h^iB}, because, being $X$ aCM, $H^1_*\big(\omega_X\big)=H^{n-1}_*\big(\cO_X\big)=0$.
\end{remark}

If $\cE$ is rank two orientable, then equality \eqref{Serre} yields $a=c$ in monad \eqref{Monad}. The monad above can be put in an even more explicit form only if there is description of aCM bundles on $X$. Unfortunately such a description is known in very few cases, but if we restrict to ordinary $h$--instanton bundles we have a more useful result.

\begin{corollary}
\label{cMonad}
Let $X$ be an $n$--fold with $n\ge3$ endowed with a very ample line bundle $\cO_X(h)$. Assume that $h^0\big(\omega_X((n-1)h)\big)=0$ and $X$ is aCM with respect to  $\cO_X(h)$.

The non--zero vector bundle $\cE$ is an ordinary $h$--instanton bundle with quantum number $\quantum$ if and only if it is the cohomology of a monad of the form
\begin{equation}
\label{MonadOrdinary}
0\longrightarrow\cC^{U,h}\longrightarrow\cB\longrightarrow\cC\longrightarrow0
\end{equation}
where $\cC\cong\cO_X(h)^{\oplus \quantum}$ and $\cB$ is Ulrich.
\end{corollary}
\begin{proof}
If $\cE$ is ordinary, then it is the cohomology of monad \eqref{Monad} with $\cC\cong\cO_X(h)^{\oplus k}$ and $\cA\cong\omega_X(nh)^{\oplus k}\cong\cC^{U,h}$ thanks to  Theorem \ref{tMonad}. The same theorem also yields that $\cB$ is aCM. Thanks to equalities \eqref{h^iB} we deduce that $\cB$ is actually Ulrich.

Conversely, let $\cE$ be the cohomology of monad \eqref{MonadOrdinary}. Being $\cB$  Ulrich, it is aCM and equalities \eqref{h^iB} hold. The thesis follows from Theorem \ref{tMonad} and Remark \ref{rMonad}.
 \end{proof}

Corollary \ref{cMonad} returns the characterization of ordinary $\cO_{\p n}(1)$--instanton bundles in terms of monad \eqref{MonadJardim} because $h^0\big(\omega_{\p n}(n-1)\big)=h^0\big(\cO_{\p n}(-2)\big)=0$. 
 
 \begin{remark}
 \label{rSo--VV}
Let $\field=\bC$ and $n\ge3$: thanks to \cite[Theorem 0.1]{So--VV} and simple calculations, we know that $h^0\big(\omega_X((n-1)h)\big)\ne0$ unless $X\cong\p n$ and $\cO_X(h)\cong\cO_{\p n}(1)$, or  $X\subseteq\p {n+1}$ is the smooth quadric hypersurface and $\cO_X(h)\cong\cO_X\otimes\cO_{\p {n+1}}(1)$, or $X$ is a scroll on a smooth curve $B$.  
\end{remark}

\begin{remark}
Let $\field=\bC$ and $n\ge3$: assume also that $X\subseteq\p N$ is an $n$--fold of minimal degree $d:=N-n+1$. Thus $X$ is one of the varieties listed in Remark \ref{rSo--VV} and, if it is a scroll, $B\cong\p1$ necessarily (see \cite{Ei--Ha1} for further details). The variety $X$ is aCM with respect to $\cO_X(h)\cong\cO_X\otimes\cO_{\p {N}}(1)$, $d=h^n$ and $h^0\big(\omega_X((n-1)h)\big)=0$ (this follows from Remark \ref{rSo--VV}), hence each $h$--instanton bundle $\cE$ on $X$ fits into monad \eqref{MonadOrdinary}.

If $p\colon X\to \p n$ is finite and $p^*\cO_{\p n}(1)\cong\cO_X(h)$, then its degree is $d$ and the functors $R^ip_*$ vanish for $i\ge1$ by \cite[Corollary III.11.2]{Ha2}. Thus, by applying $p_*$ to the exact sequences \eqref{Display}, we deduce that $p_*\cE$ is the cohomology of the monad
\begin{equation}
\label{PushForward}
0\longrightarrow p_*(\omega_X(nh))^{\oplus \quantum}\longrightarrow p_*\cB\longrightarrow p_*(\cO_X(h))^{\oplus \quantum}\longrightarrow0.
\end{equation}
Trivially $\rk(\cB)=2\quantum+\rk(\cE)$ hence $p_*\cB\cong\cO_{\p n}^{2d\quantum+d\rk(\cE)}$, because $\cB$ is Ulrich.

Notice that $p_*\cO_X\cong\bigoplus_{i=0}^{d-1}\cO_{\p n}(-\alpha_i)$ for suitable $\alpha_i\in \bZ$. Thanks to \cite[Corollary III.11.2 and Exercises III.8.1, III.8.3]{Ha2} we have $h^i\big((p_*\cO_X)(t)\big)=h^i\big(\cO_X(th)\big)$ for each $i,t\in\bZ$. Since $h^0\big(\cO_X\big)=1$, it follows that we can assume $\alpha_0=0$ and $\alpha_i\le -1$ for $i\ge1$. Since $h^0\big(\cO_X(h)\big)=N+1$ and $d=N-n+1$, it follows that $\alpha_i= -1$ for $i\ge1$. Thus 
$$
p_*(\cO_X(h))\cong\cO_{\p n}(1)\oplus\cO_{\p n}^{\oplus d-1}.
$$
Moreover, $\omega_X\cong\omega_{X\vert\p n}\otimes\cO_X(-(n+1)h)$, hence 
$$
p_*(\omega_X(nh))\cong(p_*\cO_X)^\vee(-1)\cong\cO_{\p n}(-1)\oplus\cO_{\p n}^{\oplus d-1}
$$
by duality for finite flat morphisms (see \cite[Exercise III.6.10]{Ha2}). 

Thus we have induced morphisms $a\colon \cO_{\p n}(-1)^{\oplus\quantum}\oplus\cO_{\p n}^{\oplus (d-1)\quantum}\to \cO_{\p n}^{2d\quantum+d\rk(\cE)}$ and $b\colon \cO_{\p n}^{2d\quantum+d\rk(\cE)}\to \cO_{\p n}(1)^{\oplus\quantum}\oplus\cO_{\p n}^{\oplus (d-1)\quantum}$. Since $a$ is injective and $b$ is surjective it is easy to check that the above morphisms split. Thus monad \eqref{PushForward} splits into monad \eqref{MonadJardim} having $p_*\cE$ as cohomology and a trivial exact sequence.
\end{remark} 

In the following examples we apply Corollary \ref{cMonad} to some smooth varieties. 

\begin{example}
\label{eQuadric0}
Let $X\subseteq\p{n+1}$ be a smooth quadric hypersurface and set $\cO_X(h):=\cO_X\otimes\cO_{\p {n+1}}(1)$. Thus $h^0\big(\omega_X((n-1)h)\big)=h^0\big(\cO_X(-h)\big)=0$. 

Consider an ordinary $h$--instanton bundle $\cE$ of rank $r$ and quantum number $\quantum$. If $c_1(\cE)=\varepsilon h$, then  equality \ref{Slope}  implies $c_1(\cE)=rh/2$, hence $r$ is even. If $n\ge3$, then $\cE$ is the cohomology of monad \eqref{MonadOrdinary}. 

Recall that the Kn\"orrer theorem (see \cite[Corollary 6.8]{A--O1}) implies that the only indecomposable aCM bundles on $X$ are, up to shifts, $\cO_X$ and the spinor bundles: for their definition and properties we refer the reader to \cite[Definition 1.3, Theorem 2.8 and Remark 2.9]{Ott2}. In particular, if $n$ is odd, then there is exactly one spinor bundle $\cS$, while there are two non--isomorphic spinor bundles $\cS'$ and $\cS''$ if $n$ is even. Regardless of the parity of $n$ we have
\begin{equation}
\label{Spinor}
\begin{gathered}
\rk(\cS)=\rk(\cS')=\rk(\cS'')=2^{\left[\frac{n-1}2\right]},\\
c_1(\cS(h))=c_1(\cS'(h))=c_1(\cS''(h))=2^{\left[\frac{n-3}2\right]}h.
\end{gathered}
\end{equation}
Notice that $\cO_X$ is not an Ulrich bundle, while the shifted spinor bundles $\cS(h)$, $\cS'(h)$, $\cS''(h)$ are Ulrich bundles. Monad \eqref{MonadOrdinary} becomes
\begin{equation}
\label{MonadOdd}
0\longrightarrow\cO_{X}^{\oplus k}\longrightarrow\cS(h)^{\oplus s}\longrightarrow\cO_{X}(h)^{\oplus k}\longrightarrow0,
\end{equation}
if $n$ is odd and 
\begin{equation}
\label{MonadEven}
0\longrightarrow\cO_{X}^{\oplus k}\longrightarrow\cS'(h)^{\oplus s'}\oplus\cS''(h)^{\oplus s''}\longrightarrow\cO_{X}(h)^{\oplus k}\longrightarrow0,
\end{equation}
if $n$ is even. It is clear that $s=2^{-\left[\frac{n-1}2\right]}(r+2k)$ and $s'+s''=2^{-\left[\frac{n-1}2\right]}(r+2k)$,
when $n$ is respectively either odd or even, thanks to equalities \eqref{Spinor}. In the former case the cohomology of sequences \eqref{Display} tensored by $\cS$ returns
$$
s=h^0\big(\cE\otimes\cS\big)-h^1\big(\cE\otimes\cS\big)+2^{\left[\frac{n-1}2\right]}k,
$$
thanks to \cite[Theorems 2.8 and 2.10]{Ott2}. In the latter case we have similarly
\begin{gather*}
s'=\left\lbrace\begin{array}{ll} 
h^0\big(\cE\otimes\cS'\big)-h^1\big(\cE\otimes\cS'\big)+2^{\left[\frac{n-1}2\right]}k\quad&\text{if $n\equiv 0\pmod 4$,}\\
h^0\big(\cE\otimes\cS''\big)-h^1\big(\cE\otimes\cS''\big)+2^{\left[\frac{n-1}2\right]}k\quad&\text{if $n\equiv 2\pmod 4$,}
\end{array}\right.\\
s''=\left\lbrace\begin{array}{ll} 
h^0\big(\cE\otimes\cS''\big)-h^1\big(\cE\otimes\cS''\big)+2^{\left[\frac{n-1}2\right]}k\quad&\text{if $n\equiv 0\pmod 4$,}\\
h^0\big(\cE\otimes\cS'\big)-h^1\big(\cE\otimes\cS'\big)+2^{\left[\frac{n-1}2\right]}k\quad&\text{if $n\equiv 2\pmod 4$.}
\end{array}\right.
\end{gather*}

Recall that $r$ is even, hence each indecomposable ordinary $h$--instanton bundle on $X$ has rank at least $4$ when $n\ge6$. Indeed if $\quantum=0$, then $\cE$ is a spinor bundle, while if $\quantum\ge1$ this follows from Proposition \ref{pHorrocks} (see also \cite[Corollary 1.2]{Mal1}). If $n\le 5$ rank two ordinary $h$--instanton bundles exist and we know what follows (see \cite{Mal1}).
\begin{itemize}
\item If $n=3$, then $s=k+1$: monad \eqref{MonadOdd} coincides with the one in \cite[Lemma 4]{Fa2} up to tensoring by $\cO_X(-h)$. 
\item If $n=4$, besides the obvious case of the $h$--instanton bundles $\cS'(h)$ and $\cS''(h)$, we must have $s'=s''=1$ and $k=1$: in this case the cohomology of monad \eqref{MonadEven} tensored by $\cO_X(-h)$ is the bundle defined in \cite[Proposition p. 205]{A--S}.
\item If $n=5$, then $s=1$ and $\quantum=1$: the cohomology of monad \eqref{MonadOdd} tensored by $\cO_X(-h)$ is the Cayley bundle described in \cite{Ott3}. 

In particular, such an $X$ supports a rank two orientable $h$--instanton bundle, but no any rank two Ulrich sheaf.
\end{itemize}
The general problem of the existence of $h$--instanton bundles for odd $n$ and rank $n-1$ is discussed in \cite[Section 4]{C--MR--PL}.

Conversely, if the cohomology of monads \eqref{MonadEven} and \eqref{MonadOdd} is a vector bundle $\cE$, then it is an ordinary $h$--instanton bundle of rank $r$ with quantum number $k$ thanks to Corollary \ref{cMonad}.
\end{example}

One can work out analogous computations on each manifold $X$ with sufficiently negative canonical line bundle and whose Ulrich bundles are completely described. 

In the following example we deal with the case of some particular scrolls.  Let us spend a few words on scrolls for fixing the notation here and in Section \ref{sScrollCurve}. Recall that an $n$--fold $X$ endowed with a very ample line bundle $\cO_X(h)$ is a scroll over a smooth curve $B$ if there is a vector bundle $\cG$ of rank $n$ on $B$ such that $X\cong\bP(\cG)$ and $\cO_X(h):=\cO_{\bP(\cG)}(1)$: thus $\cG$ must be ample and globally generated, hence $0<h^n=\deg(\frak g)$.
 
 With the notation above we set $\cO_B(\frak g):=\det(\cG)$. In $A^1(X)$ we denote by $f$ the class of a fibre of the projection $\pi\colon\bP(\cG)\to B$ and for each divisor $\frak a\subseteq B$ we write $\frak a f$ instead of $\pi^*\frak a$. If the genus of $B$ satisfies $g:=p_a(B)\ge1$ such fibres are not linearly equivalent: anyhow they form an algebraic system. The Chern equation in $A(X)$ is $h^n-c_1(\cG)fh^{n-1}=0$: the intersection product in $A(X)$ is given by 
$$
h^n=\deg(\frak g),\qquad fh^{n-1}=1,\qquad f^2=0
$$
in $A(X)$. Moreover $K_X=-nh+(\frak g+K_B)f$. The following lemma will be repeatedly used from now on in the paper.

\begin{lemma}
\label{lDerived}
Let $X\cong\bP(\cG)$  and $\cO_X(h):=\cO_{\bP(\cG)}(1)$.

For each divisor $\frak a\subseteq B$ we have
$$
h^i\big(\cO_X(th+\frak a f)\big)=\left\lbrace\begin{array}{ll} 
h^i\big((S^t\cG)(\frak a)\big)\quad&\text{if $t\ge0$,}\\
0\quad&\text{if $1-n\le t\le -1$.}
\end{array}\right.
$$
\end{lemma}
\begin{proof}
Recall that $R^i\pi_*\cO_X(th)=0$ for $t\ge1-n$ and  $i\ge0$, and $\pi_*\cO_X(th)\cong S^t\cG$ for $t\in\bZ$ (see \cite[Exercise III. 8.3]{Ha2}). Thus, the statement follows from the projection formula (see \cite[Exercise III.8.1]{Ha2}).
\end{proof}

As a first consequence of the above lemma we deduce that $h^0\big(\omega_X((n-1)h)\big)=0$. Nevertheless, it is not always true that the embedding induced by $\cO_X(h)$ is aCM, unless $B\cong\p1$. Thus we can assume $\cG\cong\bigoplus_{j=0}^{n-1}\cO_{\p1}(a_j)$ where $1\le a_0\le \dots\le a_{n-1}$. In this case $\cO_X(h)$ is actually very ample and  $d:=h^n=N-n+1=\sum_{j=0}^{n-1}a_j$, i.e. $X$ is a variety of minimal degree on $\p N$. 

Such varieties are described in \cite{Ei--Ha1}. In particular, $\cO_X$ is aCM. Moreover the description of Ulrich bundles is known: see \cite{A--H--M--PL}. 

In the following examples we deal with the cases $n=3$ and $X\cong\p1\times\p3$.

\begin{example}
\label{e2Segre}
If $n=3$, then $K_X=-3h+(d-2)f$. Let  $\Omega^p_{X\vert\p1}:=\wedge^p\Omega^1_{X\vert\p1}$, where $\Omega^1_{X\vert\p1}$ is the sheaf of relative differentials of the projection $\pi\colon X\to \p1$. Recall that it fits into the relative Euler sequence
\begin{equation}
\label{seqEuler}
0\longrightarrow\Omega^1_{X\vert\p1}\longrightarrow\bigoplus_{j=0}^2\cO_{X}(a_jf-h)\longrightarrow\cO_X\longrightarrow0,.
\end{equation}

Let $\cB$ be a vector bundle on $X$. As pointed out in \cite[Theorem 4.7]{A--H--M--PL}, $\cB$ is Ulrich if and only if there is a filtration 
$$
0=:\cB_0\subseteq\cB_1\subseteq\cB_2\subseteq\cB_3:=\cB
$$
such that $\cB_{p+1}/\cB_p\cong\Omega^p_{X\vert\p1}((p+1)h-f)^{\oplus s_{p+1}}$ for $0\le p\le 2$. It follows that $\cB$ is Ulrich if and only if there are exact sequences 
\begin{gather}
\label{seqUlrich}
0\longrightarrow\cB_2\longrightarrow \cB\longrightarrow\cO_X((d-1)f)^{\oplus s_3}\longrightarrow0,\\
\label{seqKernel}
0\longrightarrow\cO_X(h-f)^{\oplus s_1}\longrightarrow \cB_2\longrightarrow\Omega^1_{X\vert\p1}(2h-f)^{\oplus s_2}\longrightarrow0.
\end{gather}
Notice that  the cohomology of the dual of sequence \eqref{seqEuler} tensored by $\cO_X(-h)$ returns
\begin{equation}
\label{D}
h^i\big((\Omega_{X\vert\p1}^1)^\vee(-h)\big)=h^i\big(\Omega_{X\vert\p1}^1(2h-df)\big)=\left\lbrace\begin{array}{ll} 
d-3\quad&\text{if $i=1$,}\\
0\quad&\text{if $i\ne1$.}\\
\end{array}\right.
\end{equation}
Thus $\cB_2\cong \cO_X(h-f)^{\oplus s_1}\oplus\Omega^1_{X\vert\p1}(2h-f)^{\oplus s_2}$ if $a_2=1$. On the other hand, sequence \eqref{seqKernel} could be unsplit when $a_2\ge2$.

If $\cE$ is an ordinary $h$--instanton bundle with $\rk(\cE)=r$ and $\quantum(\cE)=k$, then $c_1(\cE)h=r(d-1)$. Moreover $\cE$ is the cohomology of monad \eqref{MonadOrdinary}, that is
\begin{equation}
\label{Monad2Segre}
0\longrightarrow\cO_{X}((d-2)f)^{\oplus k}\longrightarrow\cB\longrightarrow\cO_{X}(h)^{\oplus k}\longrightarrow0.
\end{equation}
If $\cB$ is as in sequence \eqref{seqUlrich}, then
$
s_1+2s_2+s_3=r+2k. 
$
The cohomology of the dual of sequence \eqref{seqEuler} and its shift, the one of sequences \eqref{Display}, \eqref{seqKernel} twisted by $\cO_X(f-2h)$, $\cO_X(f-h)$,  $\cO_X(-2h-f)$, $\cO_X(-3h-f)$ and Lemma \ref{lDerived} yield
\begin{gather*}
s_1=h^0\big(\cE(f-h)\big)-h^1\big(\cE(f-h)\big)+2k,\\
s_2=h^1\big(\cE(f-2h)\big)=h^2\big(\cE(-2h-f)\big),\\
s_3=h^3\big(\cE(-3h-f)\big)-h^2\big(\cE(-3h-f)\big)+2k.
\end{gather*}

Conversely, if the cohomology of monad \eqref{Monad2Segre} is a vector bundle, then it is an ordinary $h$--instanton bundle of rank $r$ with quantum number $k$ by Corollary \ref{cMonad}.

In \cite[Theorem 3.5]{A--M2} the existence of completely different and more explicit monads is proved via a derived category approach under restrictive hypotheses on $\cE$.
\end{example}

\begin{example}
\label{e3Segre}
Let $X\subseteq\p7$ be the image of the Segre embedding of $\p1\times\p3$, i.e. the rational normal scroll with $a_0=a_1=a_2=a_3=1$. Notice that, besides $\pi$, we also have the second projection $p\colon X\to \p3$ and $\Omega^1_{X\vert\p1}\cong p^*\Omega^1_{\p3}$. 

If $\cE$ is an ordinary $h$--instanton bundle with $\rk(\cE)=r$ and $\quantum(\cE)=k$, then $c_1(\cE)h^3=3r$ and it is the cohomology of monad \eqref{MonadOrdinary}, which becomes
\begin{equation*}
0\longrightarrow\cO_{X}(2f)^{\oplus k}\longrightarrow\cB\longrightarrow\cO_{X}(h)^{\oplus k}\longrightarrow0
\end{equation*}
where $\cB$ fits into an exact sequence of the form
\begin{align*}
0\longrightarrow \cO_{X}(h-f)^{\oplus s_1}&\oplus p^*\Omega^1_{\p3}(2h-f)^{\oplus s_2}\\
\longrightarrow\cB&\longrightarrow p^*\Omega^2_{\p3}(3h-f)^{\oplus s_3}\oplus \cO_{X}(3f)^{\oplus s_4}\longrightarrow0.
\end{align*}
with $s_1+3s_2+3s_3+s_4=r+2k$ (see \cite[Example 4.17]{A--H--M--PL}). The cohomology of sequences \eqref{Display}, the cohomology of the dual of the Euler sequence on $\p3$ and the K\"unneth formulas yield
\begin{gather*}
s_1=h^0\big(\cE(f-h)\big)-h^1\big(\cE(f-h)\big)+2k,\qquad s_2=h^1\big(\cE(f-2h)\big),\\
s_3=h^2\big(\cE(-3h-f)\big), \qquad s_4=h^4\big(\cE(-4h-f)\big)-h^3\big(\cE(-4h-f)\big)+2k.
\end{gather*}

Conversely, if the cohomology of monad \eqref{Monad2Segre} is a vector bundle, then it is an ordinary $h$--instanton bundle of rank $r$ with quantum number $k$ by Corollary \ref{cMonad}.

The above example can be generalized to every scroll over $\p1$ of dimension $n\ge4$ along the same lines of the previous Example \ref{e2Segre}, though the description of $\cB$ becomes more and more involved as $n$ increases.
\end{example}

Theorem \ref{tMonad} says something even for non--ordinary instanton bundles. 

\begin{example}
\label{eSpace1}
If $\cE$ is a non--ordinary $\cO_{\p n}(1)$--instanton bundle with $r:=\rk(\cE)$, then equality \eqref{ChernFano} implies $c_1(\cE)=-r/2$. If $n\ge3$, then $\cE$ is the cohomology of Monad \eqref{Monad} where $a,c$ are as in equality \eqref{DimensionMonad}.

If $\cB\cong\bigoplus_{i=1}^b\cO_{\p n}(\beta_i)$, then the vanishings $h^0\big(\cB(-1)\big)=h^n\big(\cB(1-n)\big)=0$ force $-1\le \beta_1\le0$. Theorem \ref{tMonad} then yields that $\cE$ is the cohomology of a quasi--linear monad of the form
\begin{equation}
\label{MonadSpace1}
\begin{aligned}
0\longrightarrow\cO_{\p n}(-2)^{\oplus \quantum}&\oplus\cO_{\p n}(-1)^{\oplus a}\\
&\longrightarrow\cO_{\p n}(-1)^{\oplus b_0}\oplus\cO_{\p n}^{\oplus b_1}\longrightarrow\cO_{\p n}^{\oplus c}\oplus\cO_{\p n}(1)^{\oplus \quantum}\longrightarrow0,
\end{aligned}
\end{equation}
where $a,c$ are as in equality \eqref{DimensionMonad}, $\quantum$ is as usual the quantum number of $\cE$ and 
$$
b_0=\frac r2+\quantum+a,\qquad b_1=\frac r2+\quantum+c.
$$
When $n=3$ and $r=2$, monad \eqref{MonadSpace1} has been described and used in \cite[Section 5.2]{El--Gr}. 

The shape of monad \eqref{MonadSpace1} is not surprising. Indeed, if $\cA$ is any ordinary rank two $\cO_{\p n}(1)$--instanton bundle, then $\cE:=\cA(-1)\oplus\cA$ is a non--ordinary $\cO_{\p n}(1)$--instanton bundle. In this case monad \eqref{MonadSpace1} is the direct sum of the monad defining $\cA$ (which is \eqref{MonadJardim}) and its first negative shift. 

Conversely, if the cohomology of the monad above is a vector bundle $\cE$, then it is an instanton bundle on $\p n$ with respect to $\cO_{\p n}(1)$ thanks to a direct computation.
\end{example}

\begin{remark}
\label{rSpace1}
We use the same notation of Example \ref{eSpace1}. Let $n\ge3$ and 
$$
\widehat{w}(\cE):=\min\{\ \chi(\cE)+(n+2)\quantum+h^{n-1}\big(\cE(-n)\big), h^1\big(\cE\big)+2\quantum+n\ \}
$$
Thanks to \cite[Theorem 3.2]{C--MR1} we know that $\reg(\cE)\le \max\{\ \widehat{w}(\cE), 2\ \}$. It is not easy to confront such a bound with the one in Proposition \ref{pRegularity}. Nevertheless, if $n=3$, $r=2$ and $\quantum\ge1$, then equalities \eqref{Serre} and \eqref{RRgeneral} yield $\widehat{w}(\cE)>w(\cE)=h^1\big(\cE\big)+1$. 
\end{remark}

\begin{example}
\label{eQuadric1}
If $\cE$ is a non--ordinary instanton bundle with rank $r$ and quantum number $\quantum$ on a smooth quadric hypersurface $X\subseteq\p {n+1}$, then $c_1(\cE)=0$ by equality \eqref{ChernFano} and the picture is considerably more complicated and far from many very well--known results: e.g. see \cite[Proposition 9.2]{A--C--G} and the references therein for $n=3$ and \cite{C--MR3} for $n\ge5$.

We only suggest what happens when $n$ is odd. We know that the bundle $\cB$ in Theorem \ref{tMonad} is aCM, hence we can write it as follows
$$
\cB\cong\bigoplus_{i=1}^b\cO_{X}(\beta_ih)\oplus\bigoplus_{j=1}^s\cS(\sigma_jh).
$$
The vanishings $h^0\big(\cB(-h)\big)=h^n\big(\cB((1-n)h)\big)=0$ force $\beta_i=0$ and $0\le \sigma_j\le1$. Since $c_1(\cE)=0$, it follows that $\cE$ is the cohomology of a monad of the form
\begin{equation}
\label{MonadQuadric1}
\begin{aligned}
0\longrightarrow\cO_{X}(-h)^{\oplus \quantum}&\oplus\cO_{X}^{\oplus a}\\
&\longrightarrow\cO_{X}^{\oplus b}\oplus\cS^{\oplus s}\oplus\cS(h)^{\oplus s}\longrightarrow\cO_{X}^{\oplus c}\oplus\cO_{X}(h)^{\oplus \quantum}\longrightarrow0,
\end{aligned}
\end{equation}
where $a,c$ are as in equality \eqref{DimensionMonad}. Computing $r$ from monad  \eqref{MonadQuadric1} we deduce
$$
s=2^{-\left[\frac{n+1}2\right]}(r+2\quantum+a+c-b).
$$

Conversely, if the cohomology of monad \eqref{MonadQuadric1} is a vector bundle, then a direct computation shows that it is an $\cO_X(h)$--instanton bundle.

Notice that if $s=a=c=0$, then monad \eqref{MonadQuadric1} coincides with monad \eqref{MonadJardim}.
\end{example}

\section{Instanton bundles of low rank on cyclic $n$--folds}
\label{sCyclic}
We briefly deal with $h$--instanton bundles on an $n$--fold $X$ which is cyclic, i.e. such that $\Pic(X)$ is free of rank $\varrho_X=1$, and let $\cO_X(H)$ be its ample generator. In what follows we focus on the case $n\ge2$, because the unique cyclic curve is $\p1$.

Let $\cO_X(h)\cong\cO_X(uH)$ and $\omega_X\cong\cO_X(vH)$. The ampleness of $h$ implies $u\ge1$. Moreover, if $v<0$, then $X$ is a Fano $n$--fold, hence $v\ge-n-1$ and equality holds if and only if $X\cong\p n$  by \cite[Corollary 2]{Ka--Ko}.

\begin{lemma}
\label{lCyclicChern}
Let $X$ be a cyclic $n$--fold with $n\ge2$ endowed with an ample and globally generated line bundle $\cO_X(h)$. 

If $\cE$ is a $h$--instanton bundle with defect $\defect$, then 
\begin{equation}
\label{Cyclic}
2c_1(\cE)={\rk(\cE)}((n+1-\defect)h+K_X).
\end{equation}
\end{lemma}
\begin{proof}
If $c_1(\cE)=\varepsilon H$, then $
(2\varepsilon-\rk(\cE)((n+1-\defect)u+v))u^{n-1}H^n=0
$
by equality \ref{Slope}, hence the statement follows from the ampleness of $\cO_X(H)$ and the Nakai--Moishezon criterion on $X$ (see \cite[Theorem A.5.1]{Ha2}).
\end{proof}

In particular, if $(n+1-\defect)u+v$ is odd, then $\rk(\cE)$ must be even. In what follows we will deal with $h$--instanton bundles of small rank on a cyclic manifold $X$. 

\begin{proposition}
\label{pCyclicLine}
Let $X$ be a cyclic $n$--fold with $n\ge2$ endowed with an ample and globally generated line bundle $\cO_X(h)$. Assume also that the ample generator $\cO_X(H)$ of $\Pic(X)$ is effective.

If $\mathcal L\in\Pic(X)$ is a $h$--instanton line bundle with defect $\defect$, then one of the following assertions holds.
\begin{enumerate}
\item $X\cong\p n$, $\cO_X(h)\cong\cO_{\p n}(1)$, $\defect=0$, $\mathcal L\cong\cO_{\p{n}}$.
\item $X\cong\p{3}$, $\cO_X(h)\cong\cO_{\p{3}}(2)$, $\defect=1$, $\mathcal L\cong\cO_{\p{3}}(1)$.
\item $X\subseteq\p {n+1}$ is the smooth quadric hypersurface, $\cO_X(h)\cong\cO_X\otimes\cO_{\p {n+1}}(1)$, $\defect=1$, $\mathcal L\cong\cO_X$, $n\ge3$.
\end{enumerate}
\end{proposition}
\begin{proof}
Let $\cO_X(H)$ be the ample generator of $\Pic(X)$. Equality \eqref{Cyclic} yields 
$$
\mathcal L^2\cong\cO_X((n+1-\defect)h+K_X)\cong\cO_X((u(n+1-\defect)+v)H).
$$
It follows that $u(n+1-\defect)+v=2w$ for some integer $w$ and $\mathcal L\cong\cO_X(wH)$. The conditions $h^0\big(\mathcal L(-h)\big)=0$ and $h^0\big(\cO_X(H)\big)\ne0$ imply $w\le u-1$, hence
\begin{equation}
\label{CyclicLine}
u(n-1-\defect)+v\le -2.
\end{equation}
Since $u\ge1$, it then follows $v\le -n-1+\defect$, i.e. either $X\cong\p n$ or $X\subseteq \p{n+1}$ is the smooth quadric hypersurface and $\defect=1$ by \cite[Corollary 2]{Ka--Ko}.

In the latter case $v=-n$, hence inequality \eqref{CyclicLine} becomes $(n-2)(u-1)\le0$. If $n=2$, then $\varrho_X\ge2$, hence $n\ge3$ and $u=1$ necessarily. We deduce that $\mathcal L\cong\cO_X$ and $\cO_X(h)\cong\cO_X\otimes\cO_{\p{n+1}}(1)$, i.e. assertion (3) holds true.

If $X\cong\p n$ and $\defect=0$, then inequality \eqref{CyclicLine} becomes $(n-1)(u-1)\le0$, hence $\cO_X(h)\cong\cO_{\p n}(1)$ and $\mathcal L\cong\cO_X$, i.e. assertion (1) holds. 

If $X\cong\p n$ and $\defect=1$ the same argument leads to $(n-2)(u-1)\le1$. Both the cases $u=1$ (and arbitrary $n$) and $n=2$ (and arbitrary $u$) can be excluded because $u(n+1-\defect)+v$ is not even in these cases. Thus the only possible case is $n=3$ and $u=2$ leading to $\cO_X(h)\cong\cO_{\p3}(2)$ and $\mathcal L\cong\cO_{\p3}(1)$, which is assertion (2).
\end{proof}

If $\cE$ is a vector bundle on $X$, then $t_h^{\cE}:=\left[-{\mu_h(\cE)}/{h^n}\right]$ is the unique  $t\in\bZ$ such that
$-\rk(\cE)h^n<c_1(\cE(th))h\le 0$. The sheaf $\cE_{norm,h}:=\cE(t_h^{\cE}h)$ is called {\sl normalization of $\cE$ (with respect to $\cO_X(h)$)}. 

\begin{lemma}
\label{lHoppe}
Let $X$ be a cyclic manifold.

If $\cO_X(H)$ is the ample generator of $\Pic(X)$ and $\cE$ is a vector bundle on $X$ with $c_1(\cE)=\varepsilon H$, then the following assertions hold.
\begin{enumerate}
\item $\cE$ is $\mu$--semistable (resp. $\mu$--stable) with respect to an ample line bundle if and only the same is true with respect to $\cO_X(H)$.
\item If $\cE$ is $\mu$--semistable (resp. $\mu$--stable), then $h^0\big(\cE_{norm,H}(-H)\big)=0$, (resp. $h^0\big(\cE_{norm,H}\big)=0$).
\item If $\rk(\cE)=2$ and $h^0\big(\cE_{norm,H}\big)=0$, then $\cE$ is $\mu$--stable. 
\item If $\rk(\cE)=2$, $\varepsilon$ is even and $h^0\big(\cE_{norm,H}(-H)\big)=0$, then $\cE$ is $\mu$--semistable.
\item If $\rk(\cE)=2$ and $\varepsilon$ is odd, then $\cE$ is $\mu$--stable if and only if it is $\mu$--semistable.
\end{enumerate}
\end{lemma}
\begin{proof}
If $\cO_X(h)$ is ample, then there is a positive integer $u$ such that $\cO_X(h)\cong\cO_X(uH)$, hence $\mu_h(\cE)=u^{n-1}\mu_H(\cE)$ for each sheaf $\cE$ on $X$. Thus, the stability properties of $\cE$ with respect to $\cO_X(h)$ and $\cO_X(H)$ are the same.

For the other assertions, see \cite[Lemma II.1.2.5]{O--S--S} (though the result is stated there for $X=\p n$ and  $\field=\bC$, it is easy to check that the proof holds for every cyclic manifold without restrictions on $\field$).
\end{proof}

If $\cE$ is a $h$--instanton bundle on $X$, then equality \eqref{Cyclic} holds for $c_1(\cE)$. In the following proposition we deal with the semistability properties of rank two $h$--instanton bundles on a cyclic $n$--fold (see \cite[Proposition 2]{J--MR} for a similar result).

\begin{proposition}
\label{pCyclicStable}
Let $X$ be a cyclic $n$--fold with $n\ge2$ endowed with an ample and globally generated line bundle $\cO_X(h)$.

If $\cE$ is a rank two $h$--instanton bundle on $X$ with defect $\defect$, then the following assertions hold.
\begin{enumerate}
\item $\cE$ is $\mu$--semistable unless possibly when: 
\begin{enumerate}
\item $X\cong\p n$, $\cO_X(h)\cong\cO_{\p n}(1)$, $\defect=1$;
\item $X\cong\p2$, $\cO_X(h)\cong\cO_{\p2}(u)$, $\defect=1$.
\end{enumerate}
\item If $\cE$ is $\mu$--semistable, then it is also $\mu$--stable unless possibly when:
\begin{enumerate}
\item $X\cong\p n$, $\cO_X(h)\cong\cO_{\p n}(1)$, $\defect=0$;
\item $X\cong\p3$,  $\cO_X(h)\cong\cO_{\p3}(2)$, $\defect=1$;
\item $X\subseteq\p{n+1}$ is a smooth quadric hypersurface, $\cO_X(h)\cong\cO_X\otimes \cO_{\p {n+1}}(1)$, $\defect=1$, $n\ge3$.
\end{enumerate}
\end{enumerate}
\end{proposition}
\begin{proof}
Since $\varrho_X=1$, then equality \eqref{Cyclic} yields $c_1(\cE)=(n+1-\defect)h+K_X=\varepsilon H$, where $\varepsilon=u(n+1-\defect)+v$.

By assumption we have $h^0\big(\cE(-h)\big)=0$. Since $t_H^{\cE}:=\left[-{\varepsilon}/{2}\right]$, it follows from Lemma \ref{lHoppe} that when $\varepsilon$ is even and $t_H^{\cE}-1\le -u$ (resp. $t_H^{\cE}\le -u$), then $\cE$ is $\mu$--semistable (resp. $\mu$--stable). When $\varepsilon$ is odd, the same lemma yields the $\mu$--stability and $\mu$--semistability of $\cE$ when $t_H^{\cE}\le -u$.  In what follows we will assume $n\ge3$: the computations in the case $n=2$ are similar.

If $\varepsilon$ is even, then $\cE$ is not $\mu$--semistable if $u(n-1-\defect)+v+3\le 0$. If $u\ge2$, then $v\le -2n-1+2\defect$. Since $v\ge-n-1$, it follows that $n\le 2\defect\le 2$, contradicting $n\ge3$. If $u=1$, then $v\le -n-2+\defect$, hence necessarily $\defect=1$ and $X\cong\p n$: thus $\varepsilon=-1$ which is not even, a contradiction.

If $\varepsilon$ is odd, then $\cE$ is not $\mu$--(semi)stable if $u(n-1-\defect)+v+2\le 0$.  If $u\ge3$, then $v\le -3n+1+3\defect$. Since $v\ge-n-1$, it follows that $2n\le 3+2\defect\le 5$, contradicting $n\ge3$. If $u=2$, the same argument leads to $\defect=1$ and $X\cong\p3$ necessarily: thus $\varepsilon=2$ which is not odd, a contradiction. Finally if $u=1$, arguing as above one deduces that the only admissible case is $\defect=1$ and $X\cong\p n$.

If $\varepsilon$ is even, then $\cE$ is not $\mu$--stable if $u(n-1-\defect)+v+1\le 0$. If $u\ge4$, then $v\le -4n+3+4\defect$ and one can exclude this case arguing as above. If $u=3$, then $\defect=1$ and $X\cong\p3$ necessarily, hence $\varepsilon=5$ which is not even, a contradiction. If $u=2$, then again one obtains $\defect=1$ and $X$ is either $\p3$ or a smooth quadric hypersurface in $\p{n+1}$: computing $\varepsilon$, one deduces that only the former case is possible. Finally if $u=1$, then $X$ is either $\p3$ or a smooth quadric hypersurface in $\p{n+1}$ or $v=-n+1$: the computation of $\varepsilon$ yields either $\defect=0$ and $X\cong \p n$ or $\defect=1$ and $X$ is a smooth quadric hypersurface in $\p{n+1}$
\end{proof}

\begin{remark}
\label{rCyclic}
Proposition \ref{pCyclicStable} is sharp. Indeed, the $\cO_{\p n}(u-1)\oplus\cO_{\p n}(u-2)$ is an unstable non--ordinary $\cO_{\p 2}(u)$--instanton bundle if either $u=1$ or $n=2$. 

Similarly, $\cO_{\p n}^{\oplus2}$ is a strictly $\mu$--semistable ordinary $\cO_{\p n}(1)$--instanton bundle and $\cO_{\p3}(1)^{\oplus2}$ is a  rank two strictly $\mu$--semistable non--ordinary $\cO_{\p3}(2)$--instanton bundle. Finally, if $X\subseteq\p{n+1}$ is a smooth quadric hypersurface, then  $\cO_X^{\oplus2}$ is a  rank two strictly $\mu$--semistable non--ordinary $\cO_X\otimes\cO_{\p{n+1}}(1)$--instanton bundle.
\end{remark}

More generally than projective spaces and smooth hyperquadric, other important examples of cyclic $n$--folds are provided by cyclic Fano $3$--folds. If $X$ is such an $n$--fold, then $\Pic(X)$ is generated by the fundamental line bundle $\cO_X(h)$. If $\cE$ is a $h$--instanton bundle on $X$, then equality \eqref{Cyclic} becomes
\begin{equation}
\label{ChernFano}
2c_1(\cE)={\rk(\cE)}(4-\defect-i_X)h
\end{equation}
where, as usual, $i_X$ is the index of $X$. In particular ordinary (resp. non--ordinary) $h$--instanton bundles must have even rank when $i_X=1,3$ (resp. $i_X=2,4$).

We say that a rank two vector bundle $\cE$ on $X$ is a {\sl classical instanton bundle} if $c_1(\cE)=-\varepsilon h$ with $\varepsilon\in \{\ 0,1\ \}$ and $h^0\big(\cE\big)=h^1\big(\cE(-q_X^\varepsilon h)\big)=0$, where
$$
q_X^\varepsilon:=\left[\frac{i_X+1-\varepsilon}2\right]
$$ 
(see \cite[Definition 1.1]{A--C--G}). Notice that $0\le q_X^\varepsilon\le 2$. Moreover, $q_X^\varepsilon=0$ if and only if $i_X=\varepsilon=1$, and $q_X^\varepsilon=2$ if and only if either $X\cong\p3$ or $X\subseteq\p4$ is a smooth quadric hypersurface and $\defect=1$.

Below, we confront the notion of classical instanton bundle with Definition \ref{dMalaspinion}. 

\begin{proposition}
\label{pFano}
Let $X$ be a cyclic Fano $3$--fold with very ample fundamental line bundle $\cO_X(h)$.

If $\cE$ is a rank two vector bundle with $c_1(\cE)=(4-\defect-i_X)h$ where $\defect\in\{\ 0,1\ \}$, then the following assertions hold.
\begin{enumerate}
\item If $\cE$ is an $h$--instanton bundle, then its defect is $\defect$ and:
\begin{enumerate}
\item if $(i_X,\defect)\not\in\{\ (4,0),(4,1),(3,1)\ \}$, then $\cE_{norm,h}$ is a classical instanton bundle;
\item if $(i_X,\defect)\in\{\ (4,0),(4,1),(3,1)\ \}$, then $\cE_{norm,h}$ is a classical instanton bundle if and only if $h^0\big(\cE\big)=0$.
\end{enumerate}
\item If $\cE_{norm,h}$ is a classical instanton bundle and:
\begin{enumerate}
\item $(i_X,\defect)\ne(1,0)$, then $\cE$ is a $h$--instanton bundle with defect $\defect$;
\item $(i_X,\defect)=(1,0)$, then $\cE$ is a $h$--instanton bundle with defect $\defect$ if and only if $h^0\big(\cE_{norm,h}(h)\big)=0$.
\end{enumerate}
\end{enumerate}
\end{proposition}
\begin{proof}
Let $\cE$ be a rank two vector bundle such that $c_1(\cE)=(4-\defect-i_X)h$ where $\defect\in\{\ 0,1\ \}$. Thus $\cE_{norm,h}\cong\cE(t_h^\cE h)$ where 
$$
t_h^\cE=\left[-\frac{\mu_h(\cE)}{h^3}\right]=\left[\frac{i_X+\defect}{2}\right]-2=q_X^{1-\defect}-2.
$$

Let $\cE$ be a rank two $h$--instanton bundle. Since $c_1(\cE)=(4-\defect-i_X)h$ its defect is $\defect$. Moreover, by definition we know that $h^0\big(\cE(-h)\big)=h^1\big(\cE(-2h)\big)=0$, hence
$$
h^0\big(\cE_{norm,h}((1-q_X^{1-\defect})h)\big)=h^1\big(\cE_{norm,h}(-q_X^{1-\defect} h)\big)=0.
$$
Notice that $1-q_X^{1-\defect}\ge-1$ and $1-q_X^{1-\defect}\ge0$ unless $q_X^{1-\defect}=2$. The latter equality holds if and only if either $X\cong\p3$ or $X\subseteq\p{4}$ is a smooth quadric hypersurface and $\defect=1$. In these cases $\cE_{norm,h}$ is a classical instanton bundle if and only if  $h^0\big(\cE_{norm,h}\big)=0$. In the remaining cases the vanishing $h^0\big(\cE_{norm,h}\big)=0$ is for free.

Conversely, if $\cE_{norm,h}$ is a classical instanton bundle, then $c_1(\cE_{norm,h})=-\varepsilon h$ where $\varepsilon:=i_X+\defect-2q_X^{1-\defect}$ and $h^0\big(\cE_{norm,h}\big)=h^1\big(\cE_{norm,h}(-q_X^{1-\defect} h)\big)=0$  by definition. Notice that
$$
q_X^\varepsilon=\left[\frac{i_X+1-\varepsilon}2\right]=\left[\frac{1-\defect+2q_X^{1-\defect}}2\right]=\left[\frac{1-\defect}2\right]+q_X^{1-\defect}=q_X^{1-\defect},
$$
because $\defect\in\{\ 0,1\ \}$, hence
$$
h^0\big(\cE((q_X^{1-\defect}-2)h)\big)=h^1\big(\cE(-2h)\big)=0,
$$
Notice that $q_X^{1-\defect}-2\ge-2$ and $1-q_X^{1-\defect}\ge-1$ unless $q_X^{1-\defect}=0$. The latter equality holds if and only $(i_X,\defect)=(1,0)$. In this case $\cE$ is a $h$--instanton bundle if and only if $h^0\big(\cE_{norm,h}(h)\big)=0$.
\end{proof}

In what follows we briefly deal with the case $(i_X,\defect)=(1,0)$. In this case $\cO_X(h)\cong\omega_X^{-1}$ and  $\cE_{norm,h}\cong\cE(-2h)$, hence $c_1(\cE_{norm,h})=-h$. Equality \eqref{RRgeneral} for $\cO_X(h)$ implies that $\cO_X(h)$ induces an embedding $X\subseteq\p{g_X+1}$ where $h^3=2g_X-2$, hence $g_X\ge3$ necessarily and $\cO_X(h)\cong\cO_X\otimes\cO_{\p{g_X+1}}(1)$.

Thanks to equality \eqref{RRgeneral} for $\cE(-h)$, the inequality $h^1\big(\cE(-h)\big)=-\chi(\cE(-h))\ge0$ is equivalent to $c_2(\cE)h\ge 5g_X-1$, hence $c_2(\cE_{norm,h})h\ge g_X+3$. In \cite[Theorem 3.1]{B--F1} the existence of classical instanton bundles $\cF$ satisfying $g_X+3>c_2(\cF)h$ is proved when $X$ is ordinary. Thus the condition $h^0\big(\cE_{norm,h}(h)\big)=0$ is restrictive. 

Nevertheless, $\omega_X^{-1}$--instanton bundles on ordinary prime Fano $3$--folds exist for each admissible quantum number as claimed in Theorem \ref{tPrime}. 

As we will check below, its proof is based on the results in \cite{B--F1} where $\field=\bC$ is assumed. For this reason, we make such a hypothesis in Theorem \ref{tPrime}.

\medbreak
\noindent{\it Proof of Theorem \ref{tPrime}.}
We will prove the statement by induction on $k\ge0$ checking the existence  of a vector bundle $\cE_k$ such that:
\begin{itemize}
\item $\cE_k$ is $\mu$--stable of rank $2$ with $c_1(\cE_k)=3h$, $c_2(\cE_k)h=5g_X-1+k$;
\item $h^0\big(\cE_k(-h)\big)=h^1\big(\cE_k(-2h)\big)=0$;
\item $h^i\big(\cE_k\otimes\cE_k^\vee\big)=0$ for $i\ge2$;
\item $\cE_k\otimes\cO_L\cong\cO_{\p1}(1)\oplus\cO_{\p1}(2)$ for each general line $L\subseteq X$.
\end{itemize}
If such an $\cE_k$ exists, then equalities \eqref{Serre} and  \eqref{RRgeneral} for $\cE_k$ imply
$$
h^1\big(\cE_k(-h)\big)=-\chi(\cE_k(-h))=\quantum,
$$
hence $\cE_k$ is a $h$--instanton bundle, thanks to Proposition \ref{pSpecial}. Since $\cE_k$ is $\mu$--stable, it follows $h^0\big(\cE_k\otimes\cE_k^\vee\big)=1$, hence yields $h^1\big(\cE_k\otimes\cE_k^\vee\big)=4+g_X+2k$ by equality \eqref{RRgeneral} for $\cE_k\otimes\cE_k^\vee$.

Let us shortly recall how to prove the base step. If $k=0$, thanks to \cite[Lemma 3.8]{B--F1} there exists on $X$ a rank two $\mu$--stable bundle $\cE_0$ such that $h^1\big(\cE_0(-2h)\big)=0$, $c_1(\cE_0)=3h$, $c_2(\cE_0)h=5g_X-1$, $h^2\big(\cE_0\otimes\cE_0^\vee\big)=0$ and $\cE_0\otimes\cO_L\cong \cO_{\p1}(1)\oplus\cO_{\p1}(2)$ for the general line $L\subseteq X$. Looking at \cite[Proof of Theorem 3.1 for $d=g+3$ (p. 132)]{B--F1}, one also deduces that $\cE_0$ can be assumed both aCM and satisfying $h^0\big(\cE_0(-h)\big)=0$.

By \cite[Theorem 3.1 (p.120)]{B--F1} we can also assume that $\cE_0$ corresponds to a smooth point in a component $M$ of dimension $4+g_X$ of its moduli space. The tangent space to $M$ at the point $\cE_0$ has  dimension $h^1\big(\cE_0\otimes\cE_0^\vee\big)=4+g_X$. Since $\cE_0$ is $\mu$--stable, it follows that $h^0\big(\cE_0\otimes\cE_0^\vee\big)=1$. Thus equality \eqref{RRgeneral} finally returns $h^3\big(\cE_0\otimes\cE_0^\vee\big)=0$, hence the proof of the base step of the induction is complete. 

In order to prove the inductive step, we assume the existence of the vector bundle  $\cE_k$ for some $k\ge0$ satisfying the list of properties above. Such properties are the same ones listed in the statement of \cite[Theorem 3.7]{B--F2}, but the vanishing $h^0\big(\cE_k(-h)\big)=0$. The proof of the inductive step almost coincides with the proof of the inductive step in \cite[proof of Theorem 3.7]{B--F2}, the only additional property we have to check being the vanishing $h^0\big(\cE_k(-h)\big)=0$. 

In order to show that such a vanishing holds as well we  recall the argument used in \cite{B--F2} for obtaining $\cE_{k+1}$.  Let $L\subseteq X$ be a  general line, hence $\cN_{L\vert X}\cong\cO_{\p1}\oplus\cO_{\p1}(-1)$. Consider the exact sequence
$$
0\longrightarrow\widehat{\cE}_{k+1}\longrightarrow{\cE_k}\longrightarrow\cO_L\otimes\cO_X(h)\longrightarrow0.
$$
The sheaf $\widehat{\cE}_{k+1}$ is not a vector bundle, but it  satisfies all the other required properties. Thus the same properties also hold for each general deformation $\cE_{k+1}$ of $\widehat{\cE}_{k+1}$ inside the moduli space of torsion--free coherent $\mu$--stable sheaves with Chern classes $c_1=h$, $c_2=c_2(\cE_k)h+1$, $c_3=c_3(\cE_k)=0$. In the proof of  \cite[Theorem 3.7]{B--F2} the authors show that $\cE_{k+1}$ is a vector bundle still satisfying the same properties. 

Moreover, if $h^0\big(\cE_k(-h)\big)=0$, then $h^0\big(\widehat{\cE}_{k+1}(-h)\big)=0$, hence $h^0\big({\cE}_{k+1}(-h)\big)=0$ by semicontinuity. In particular $\cE_{k+1}$ is a $h$--instanton sheaf. 

Notice that $h^3\big(\cE_{k+1}\otimes\cE_{k+1}^\vee\big)=0$ (e.g. see \cite[Lemma 2.1]{C--C--G--M}). Moreover, the $\mu$-stability of $\cE_{k+1}$ yields $h^0\big(\cE_{k+1}\otimes\cE_{k+1}^\vee\big)=1$. In the proof of the inductive step in \cite[proof of Theorem 3.7]{B--F2} the authors also check that $h^2\big(\cE_{k+1}\otimes\cE_{k+1}^\vee\big)=0$, thus 
$$
h^1\big(\cE_{k+1}\otimes\cE_{k+1}^\vee\big)=4+g_X+2k
$$
(see \cite[Equalities (3.15) and (3.20)]{B--F2}). In particular the proof of the inductive step is complete, hence the statement is proved by induction.
\qed
\medbreak

\begin{remark}
The bundles defined in Theorem \ref{tPrime} are simple. Thus they correspond to smooth points in the moduli space $\cS_X(2;3h,5g_X-1+k)$ of rank two simple sheaves $\cE$ with $c_1(\cE)=3h$ and $c_2(\cE)h=5g_X-1+k$. If we denote by $\cS\cI_X(3h,5g_X-1+k)$ the locus of points in $\cS_X(2;3h,5g_X-1+k)$ representing such bundles, we deduce that every such point is contained in a generically smooth component of $\cS\cI_X(3h,5g_X-1+k)$ of dimension $4+g_X+2k$.
\end{remark}

\section{Instanton bundles of low rank on sextic del Pezzo $3$--folds}
\label{sSextic}
If $X$ is a Fano $3$--fold with fundamental line bundle $\cO_X(h)$ and such that $\varrho_X\ge2$, then the notions of classical and $h$--instanton bundle can be considerably different. 

In order to give examples of the possible pathologies and to complete the results of the previous section when $i_X=2$, we deal with the del Pezzo $3$--folds of degree $6$, i.e. the flag $3$--fold and  the image of the Segre embedding $\p1\times\p1\times\p1\subseteq\p7$. 

Classical instanton bundles on the flag $3$--fold $X$ have been described in \cite{M--M--PL} (see also \cite{C--C--G--M}). Let $p_i\colon X\to\p2$, $i=1,2$, be the projections and set $\cO_X(h_i):=p_i^*\cO_{\p2}(1)$, $i=1,2$. The group $\Pic(X)$ is freely generated by such line bundles and
$$
A(X)\cong\frac{\bZ[h_1,h_2]}{(h_1^3,h_2^3,h_1^2-h_1h_2+h_2^2)}
$$
(for the details see \cite{M--M--PL}). The fundamental line bundle is $\cO_X(h)=\cO_X(h_1+h_2)$.

\begin{proposition}
\label{pFlag1}
Let $X$ be the flag $3$--fold and let $\cO_X(h)$ be its fundamental line bundle.

A line bundle $\mathcal L$ on $X$ is a $h$--instanton  bundle with defect $\defect$ if and only if there is an integer $a\ge1$ such that $\mathcal L\cong \mathcal L_a:=\cO_X(-ah_1+(a+2-\defect)h_2)$ up to permutations of the $h_i$'s. In particular $\quantum(\mathcal L_a)=\frac{(2-\defect)}2a(a+2-\defect)$.
\end{proposition}
\begin{proof}
Let $\cE:=\cO_X(a_1h_1+a_2h_2)$ be a $h$--instanton bundle on $X$: without loss of generality we assume $a_1\le a_2$. Then $a_1+a_2=2-\defect$ by Theorem \ref{tSlope}, hence $a_2=-a_1+2-\defect$. The vanishing $h^0\big(\cE(-h)\big)=0$ and \cite[Proposition 2.5]{C--F--M2} yields $a_1=-a$ for some $a\ge0$. The same proposition also yields that the other vanishings in Theorem \ref{tSlope} are fulfilled, hence every such line bundle is a $h$--instanton bundle.

The value $\quantum(\mathcal L_a)$ is obtained via  \cite[Proposition 2.5]{C--F--M2}.
\end{proof}

\begin{example}
\label{eFlagExtension}
If $a\ge b+2$ we have  $h^1\big(\cO_X((b-a)h_1-bh_2+ah_3)\big)\ne0$ by \cite[Proposition 2.5]{C--F--M2}. It follows the existence of a non--trivial extension
$$
0\longrightarrow\cO_X(-ah_1+(a+2-\defect)h_2)\longrightarrow\cE\longrightarrow\cO_X(-bh_1+(b+2-\defect)h_2)\longrightarrow0.
$$
We show below that $\cE$ is $\mu$--semistable and simple, hence indecomposable, which is non--orientable because $c_1(\cE)\ne2h$. To this purpose, we first notice that all the sheaves in the above sequence have the same slope $3(2-\defect)$. 

If $\mathcal L\subseteq\cE$ is a destabilizing sheaf of maximal slope $\mu_h(\mathcal L)>\mu_h(\cE)$. Set $\mathcal Q:=\cE/\mathcal L$ and let  $\cM$ be the kernel of the natural quotient morphism $\cE\to\mathcal Q/\mathcal T$, where $\mathcal T\subseteq\mathcal Q$ is the torsion subsheaf. Trivially $\cM$ is torsion--free and of rank $1$ and $\mu_h(\cM)\ge\mu_h(\mathcal L)>\mu_h(\cE)$. Since $\cE/\cM$ is torsion--free, it follows that $\cM$ is also normal thanks to \cite[Lemma II.1.1.16]{O--S--S}. Thus \cite[Lemma II.1.1.12]{O--S--S} implies that $\cM$ is reflexive, hence it is a line bundle by \cite[Lemma II.1.1.15]{O--S--S}. 

By composition we obtain a morphism $\mathcal M\to\cO_X(-bh_1+(b+2-\defect)h_2)$ which is zero by slope reasons. It follows that the inclusion $\mathcal M\subseteq\cE$  factors through an inclusion $\mathcal M\subseteq\mathcal L_a$, an absurd again by slope reasons. 

We deduce that $\cE$ is $\mu$--semistable. Moreover, $\cE$ is also simple, hence indecomposable: this assertion follows  from the indecomposability of sequence \eqref{seqSegreStable} and standard arguments (see \cite[Lemma 4.2]{C--H2} for a brief and explicit statement in an even more general setting).
\end{example}

We close our analysis of $h$--instanton bundles on the flag $3$--fold by proving Theorem \ref{tFlag2} stated in the introduction.

\medbreak
\noindent{\it Proof of Theorem \ref{tFlag2}.}
We have $h^0\big(\cE(-h)\big)=h^1\big(\cE(-2h)\big)=0$. Moreover, $c_1(\cE)=2h$ because $\cE$ is ordinary and orientable, hence  $\mu_h(\cE)=6$. 

If $\cE$ is not $\mu$--semistable, then the same argument used in Example \ref{eFlagExtension} yields the existence of a maximal destabilizing line bundle $\cM:=\cO_X(\alpha_1 h_1+\alpha_2 h_2)\subseteq\cE(-h)$ for some $\alpha_1,\alpha_2\in\bZ$ such that $3(\alpha_1+\alpha_2)=\mu_h(\cM)\ge1>0=\mu_h(\cE(-h))$ and $\alpha_1\le\alpha_2$. Moreover, its maximality implies the existence of an exact sequence
\begin{equation}
\label{seqFlag}
0\longrightarrow\cO_X(\alpha_1 h_1+\alpha_2 h_2)\longrightarrow\cE(-h)\longrightarrow\cI_{Z\vert X}(-\alpha_1 h_1-\alpha_2 h_2)\longrightarrow0
\end{equation}
where either $Z\subseteq X$ has pure codimension $2$ or $Z=\emptyset$. 

On the one hand $h^0\big(\cE(-h)\big)=0$, hence $\alpha_1\le-1$ and $3\alpha_2\ge-3\alpha_1+1$. Thus $\alpha_2\ge2$, hence
$$
h^0\big(\cI_{Z\vert X}(-(\alpha_1+1) h_1-(\alpha_2+1) h_2)\big)\le h^0\big(\cO_X(-(\alpha_1+1) h_1-(\alpha_2+1) h_2)\big)=0.
$$
The cohomology of sequence \eqref{seqFlag} tensored by $\cO_X(-h)$ finally yields
$$
h^1\big(\cO_X((\alpha_1-1)h_1+(\alpha_2-1) h_2)\big)\le h^1\big(\cE(-2h)\big).
$$

On the other hand $\alpha_1-1\le-2$ and $(\alpha_1-1)+(\alpha_2-1)+1\ge0$. Thus \cite[Proposition 2.5]{C--F--M2} implies that the left--hand member of the above inequality is positive, hence the same is true for the right--hand member, contradicting Definition \ref{dMalaspinion}. We deduce that $\cE$ must be $\mu$--semistable. 

Notice that if the $h$--instanton bundle $\cE$ is indecomposable, then all the hypotheses of \cite[Proposition 2.4]{A--C--G} are fulfilled by $\cE(-h)$, hence $\cE$ is also simple.
\qed
\medbreak

There are even more pathologies when $X=\p1\times\p1\times\p1$: classical instanton bundles on $X$ have been studied in \cite{A--M1}. In this case, there are three projections $p_i\colon X\to\p1$, $i=1,2,3$, and we set $\cO_X(h_i):=p_i^*\cO_{\p1}(1)$, $i=1,2,3$. Thus $\Pic(X)$ is freely generated by such line bundles,
$$
A(X)\cong\frac{\bZ[h_1,h_2,h_3]}{(h_1^2,h_2^2,h_3^2)}
$$
and the fundamental line bundle is $\cO_X(h_1+h_2+h_3)$.

\begin{proposition}
\label{pSegre1}
Let $X\cong \p1\times\p1\times\p1$ and let $\cO_X(h)$ be its fundamental line bundle.

A line bundle $\mathcal L$ on $X$ is a  $h$--instanton  bundle if and only if there is an integer $a\ge1$ such that $\mathcal L\cong \mathcal L_a:=\cO_X(-ah_1+h_2+(2+a)h_3)$ up to permutations of the $h_i$'s. In particular, it is ordinary and $\quantum(\mathcal L_a)=a(a+2)$.
\end{proposition}
\begin{proof}
Let $\mathcal L:=\cO_X(a_1h_1+a_2h_2+a_3h_3)$ be a $h$--instanton line bundle on $X$: without loss of generality we assume $a_1\le a_2\le a_3$. The condition on $c_1(\mathcal L)h^{2}$ in Definition \ref{dMalaspinion} implies $2(a_1+a_2+a_3)=3(2-\defect)$, hence $\defect=0$ necessarily and $a_3=3-a_1-a_2$. 

The vanishings $h^0\big(\mathcal L(-h)\big)=h^1\big(\mathcal L(-2h)\big)=0$ and the K\"unneth formulas imply $a_1\le0$ and $a_2\le1$. The vanishings $h^3\big(\mathcal L(-3h)\big)=h^2\big(\mathcal L(-2h)\big)=0$ and the K\"unneth formulas imply $a_3\ge2$ and $a_2\ge1$. Thus, $\mathcal L\cong\mathcal L_a$ where $a:=-a_1$.

The K\"unneth formulas return the value of the quantum number.
\end{proof}

There exist many non--orientable $h$--instanton bundles $\cE$ on such an $X$. 

\begin{example}
\label{eSegreExtension}
If $b\ge2$ and $b\ge a\ge0$ we have $h^1\big(\cO_X((b-a)h_1-bh_2+ah_3)\big)\ne0$, hence there are non--trivial extensions of the form
$$
0\longrightarrow\cO_X(-ah_1+h_2+(1+a)h_3)\longrightarrow\cE\longrightarrow\cO_X(-bh_1+(1+b)h_2+h_3)\longrightarrow0.
$$
One can check as in Example \ref{eFlagExtension} that such $\cE$'s are $\mu$--semistable and simple (hence indecomposable).
\end{example}

\begin{example}
\label{eSegreDeform}
In \cite{C--F--M1} the existence of two irreducible families of pairwise non--isomorphic rank two Ulrich bundles $\cE_0$ with $c_1(\cE_0)=h_1+2h_2+3h_3$ is proved: moreover, the general one in each family is $\mu$--stable. 

Such bundles are ordinary $h$--instanton bundles thanks to Corollary \ref{cUlrich}. For each positive integer $\quantum$ one can then construct by deformation ordinary, $\mu$--stable, $h$--instanton bundles $\cE$ with $c_1(\cE)=h_1+2h_2+3h_3$, starting from the aforementioned $\cE_0$ and deforming the kernel of a general morphism $\cE_0(-h)\to\cO_L$ where $L$ is a general line with class $h_1h_3\in A^2(X)$ using the same argument of the proof of Theorem \ref{tPrime}. Notice that such bundles, being $\mu$--stable are simple.
\end{example}

We show below that there are rank two, orientable, ordinary, simple, unstable $h$--instanton bundles $\cE$ on $X$ with arbitrarily large quantum number. 

\begin{example}
\label{eSegreStable}
let $L\subseteq X$ be the complete intersection of general divisors in $\vert h_2\vert$ and $\vert h_3\vert$. We have $L\cong\p1$, because $Lh^2=1$, and we can find $s\ge0$ pairwise disjoint curves of this type on $X$: let $Z$ be their union. 

We claim that $\det(\cN_{Z\vert X})\cong \cO_Z\otimes\cO_X(2h_2-4h_3)$. In order to prove such an isomorphism it suffices to check that the two line bundles are trivial when restricted to each connected component, because they are all isomorphic to $\p1$. Theorem \ref{tSerre} implies the existence of a rank two vector bundle $\cE$ fitting into the exact sequence
\begin{equation}
\label{seqSegreStable}
0\longrightarrow\cO_X(h_1+3h_3)\longrightarrow\cE\longrightarrow\cI_{Z\vert X}(h_1+2h_2-h_3)\longrightarrow0
\end{equation}

We can assume that sequence \eqref{seqSegreStable} does not split. This is for free if $s>0$, while if $s=0$ it follows from the equality $h^1\big(\cO_X(-2h_2+4h_3)\big)=5$. 

We have $h^0\big(\cE(-h)\big)=h^1\big(\cE(-2h)\big)=0$ and equality \eqref{Cyclic} is fulfilled because $c_1(\cE)=2h$, hence $\cE$ is a rank two orientable,  ordinary $h$--instanton bundle by Proposition \ref{pSpecial}. Moreover, the cohomologies of sequences \eqref{seqSegreStable} tensored by $\cO_X(-h)$ and \eqref{seqStandard} tensored by $\cO_X(h_2-2h_3)$ imply that the quantum number is $\quantum:=s-2$.

Since $\mu_h(\cO_X(h_1+3h_3))=8>6=\mu_h(\cE)$, it follows that it is not $\mu$--semistable. Thus, $\cE$ is obtained neither as in Examples \ref{eSegreExtension} and \ref{eSegreDeform}, nor possibly via the method described in the proof of Theorem \ref{tPrime} applied to $X$.

We claim that $\cE$ is simple. If $s=0$ the claim is standard, because sequence \eqref{seqSegreStable} is assumed non--split (see\cite[Lemma 4.2]{C--H2}). If $s>0$, it suffices to check $h^0\big(\cE\otimes\cE^\vee\big)\le1$. Sequence \eqref{seqSegreStable} tensored by $\cE^\vee\cong\cE(-2h)$ yields
\begin{equation}
\label{Simple}
h^0\big(\cE\otimes\cE^\vee\big)\le h^0\big(\cE(-h_1-2h_2+h_3)\big)+h^0\big(\cI_{Z\vert X}\otimes\cE(-h_1-3h_3)\big).
\end{equation}
Sequence \eqref{seqSegreStable} tensored by $\cO_X(-h_1-2h_2+h_3)$ returns the exact sequence
$$
h^0\big(\cE(-h_1-2h_2+h_3)\big)\le h^0\big(\cO_X(-h_1-2h_2+h_3)\big)+h^0\big(\cI_{Z\vert X}\big)=0.
$$
Sequences \eqref{seqStandard} tensored by $\cE(-h_1-3h_3)$ and \eqref{seqSegreStable} by $\cO_X(-h_1-3h_3)$ return
$$
h^0\big(\cI_{Z\vert X}\otimes\cE(-h_1-3h_3)\big)\le h^0\big(\cE(-h_1-3h_3)\big)\le h^0\big(\cO_X\big)+h^0\big(\cI_{Z\vert X}(2h_2-4h_3)\big).
$$
Trivially $h^0\big(\cI_{Z\vert X}(2h_2-4h_3)\big)=0$, hence $h^0\big(\cI_{Z\vert X}\otimes\cE(-h_1-3h_3)\big)=1$. We deduce that inequality \eqref{Simple} leads to $h^0\big(\cE\otimes\cE^\vee\big)\le1$ and the claim follows.
\end{example}

\section{Existence of rank two orientable ordinary\\ instanton bundles on scrolls over smooth curves}
\label{sScrollCurve}
We already pointed out in Section \ref{sMonad} the importance of ordinary instanton bundles on scrolls over smooth curves. In this section we will prove the existence of rank two orientable, ordinary instanton bundles on such kind of $n$--folds. 

For the notation we refer the reader to Section \ref{sMonad}. Let $X\subseteq \p N$ be a scroll of dimension $n$ on a smooth curve $B$ and let $g:=p_a(B)$ be its genus. It is well--known that if $\theta\in \Pic^{g-1}(B)$ is a non--effective theta--characteristic, then $\cO_X(h+\theta f)$ and $\cO_X((\frak g+\theta)f)$ are both Ulrich line bundles, hence $h$--instanton bundles with zero quantum number. In what follows we will construct $h$--instanton bundles of rank $2$ with arbitrarily large quantum number.

Each hyperplane $L$ in a fibre of $\pi$ is cut on that fibre by a divisor in $\vert h\vert$. In particular the class of $L$ inside $A(X)$ is $hf$. Thus there is an exact sequence
$$
0\longrightarrow\cO_X(-h-f)\longrightarrow\cO_X(-h)\oplus\cO_X(-f)\longrightarrow\cI_{L\vert X}\longrightarrow0,
$$
it follows that $\cN_{L\vert X}\cong\cO_{\p{n-2}}\oplus\cO_{\p{n-2}}(1)$. In particular
\begin{equation}
\label{EZ-g-e}
h^i\big(\cN_{L\vert X}\big)=\left\lbrace\begin{array}{ll} 
n\quad&\text{if $i=0$,}\\
0\quad&\text{if $i\ne0$.}
\end{array}\right.
\end{equation}

\begin{construction}
\label{conScrollCurve}
Let $\cG$ be a vector bundle of rank ${n}\ge3$ on a smooth curve $B$ and set $X:=\bP(\cG)$. Assume that $\cO_X(h):=\cO_{\bP(\cG)}(1)$ is an ample and globally generated line bundle.

For each $k\ge0$ take general points $b_i\in B$, hyperplanes $L_i\subseteq\pi^{-1}(b_i)\cong\p{n-1}$ for $1\le i\le k$ and set $Z:=\bigcup_{i=1}^kL_i$. 
The definition of $Z$ implies $\det(\cN_{Z\vert X})\cong\cO_Z\otimes\cO_X(h-\frak g f)$ by the adjunction formula. 
Since Lemma \ref{lDerived} implies $h^i\big( \cO_X(-h+\frak g f)\big)=0$, it follows that Theorem \ref{tSerre} yields the existence of an exact sequence
\begin{equation}
\label{seqScrollCurve}
0\longrightarrow\cO_X((\frak g+\theta)f)\longrightarrow\cE\longrightarrow\cI_{Z\vert X}(h+\theta f)\longrightarrow0
\end{equation}
where $\theta\in \Pic^{g-1}(B)$ is a non--effective theta--characteristic. 
\end{construction}

Notice that
$$
c_1(\cE)=h+(\frak g+K_B)f=(n+1)h+K_X,\qquad c_2(\cE)=c_2^{B,\cG,k}:=(\frak g+\theta+k)hf
$$
for the bundles obtained via the above construction. We are ready to prove Theorem \ref{tScrollCurve} stated in the Introduction.

\medbreak
\noindent{\it Proof of Theorem \ref{tScrollCurve}.}
Sequences \eqref{seqScrollCurve} and \eqref{seqStandard} yield
$$
-\chi(\cE(-h))=-\chi(\cO_X(-h+(\frak g+\theta)f))-\chi(\cO_X(\theta f))+\chi(\cO_Z\otimes\cO_X(\theta f)).
$$
Lemma \ref{lDerived} and equality \eqref{RRcurve} for $\cO_B(\theta)$ imply
$$
\chi(\cO_X(-h+(\frak g+\theta)f))=\chi(\cO_X(\theta f))=0.
$$
Since $\cO_Z\otimes\cO_X(\theta f)\cong\cO_Z$, it finally follows that $-\chi(\cE(-h))=k$.

The cohomology of sequence \eqref{seqScrollCurve} tensored by $\cO_X(th)$ yields
$$
h^i\big(\cE(th)\big)\le h^i\big(\cO_X(th+(\frak g+\theta)f)\big)+h^i\big(\cI_{Z\vert X}((t+1)h+\theta f)\big).
$$
First, we show that the summands on the right are both zero if $2\le i\le n-2$ and $-n\le t\le -1$. 

From Lemma \ref{lDerived}, we have $h^i\big(\cO_X(th+(\frak g+\theta)f)\big)=0$ if $1-n\le t\le -1$. If $t=-n$, then equality \eqref{Serre} yields
$$
h^i\big(\cO_X(th+(\frak g+\theta)f)\big)=h^i\big(\cO_X(-nh+(\frak g+\theta)f)\big)=h^{n-i}\big(\cO_X(\theta f)\big).
$$
The latter dimension is zero again by Lemma \ref{lDerived} and $2\le i\le n-2$. 

Now consider the second summand on the right. The same argument as above also gives  $h^i\big(\cO_X((t+1)h+\theta f)\big)=0$ for $i\ge2$ and $-n-1\le t\le-1$. Moreover, if $L$ is any component of $Z$, then
$$
h^{i-1}\big(\cO_Z\otimes\cO_X((t+1)h+\theta f)\big)=k h^{i-1}\big(\cO_{L}(t+1)\big)=0
$$
for $2\le i\le n-2$ and $-n\le t\le -1$. Thus the cohomology of sequence \eqref{seqStandard} tensored by $\cO_X((t+1)h+\theta f)$ yields $h^i\big(\cI_{Z\vert X}((t+1)h+\theta f)\big)=0$ in the same range.

Similarly one checks that $h^1\big(\cE(-2h)\big)=0$. The same argument, sequence \eqref{seqStandard} tensored by $\cO_X(\theta f)$ and the choice of $\theta$ yield
$$
h^0\big(\cE(-h)\big)\le h^0\big(\cI_{Z\vert X}(\theta f)\big)\le h^0\big(\cO_X(\theta f)\big)=0.
$$
Since $c_1(\cE)=(n+1)h+K_X$, it follows that $\cE$ is an ordinary, orientable $h$--instanton bundle by Proposition \ref{pSpecial}.

Since $\mu_h(\cO_X((\frak g+\theta)f))=\mu_h(\cI_{Z\vert X}(h+\theta f))$, it follows that the assertion on the $\mu$--semistability can be proved with the same argument used in Example \ref{eFlagExtension}.
Thus, the proof of the statement is then complete.
\qed
\medbreak

We now deal with the decomposability of the bundles in Construction \ref{conScrollCurve}.

\begin{proposition}
\label{pSplit}
Let $\cG$ be a vector bundle of rank $n\ge3$ on a smooth curve $B$ and set $X:=\bP(\cG)$. Assume that $\cO_X(h):=\cO_{\bP(\cG)}(1)$ is an ample and globally generated line bundle.

If $\cE$ is the vector bundle defined in Construction \ref{conScrollCurve}, then it is decomposable if and only if $k=0$. In this case $\cE\cong\cO_X((\frak g+\theta)f)\oplus\cO_X(h+\theta f)$ and sequence \eqref{seqScrollCurve} splits.
\end{proposition}
\begin{proof}
Let $\cE\cong\mathcal L\oplus\cM$ where $\mathcal L,\cM\in\Pic(X)$. Thanks to Theorem \ref{tScrollCurve} $\cE$ is $\mu$--semistable, then 
$\mu_h(\mathcal L)=\mu_h(\cM)=\mu_h(\cE)=\deg(\frak g)+g-1$.

If neither $\mathcal L$ nor $\cM$ is isomorphic to $\cO_X((\frak g+\theta)f)$, then the map $\cO_X((\frak g+\theta)f)\to\cE$ should be zero. Thus we can assume $\mathcal L\cong \cO_X((\frak g+\theta)f)$. It follows the existence of a surjective morphism $\varphi\colon\cM\to\cI_{Z\vert X}(h+\theta f)$. Composing such morphism with the inclusion $\psi\colon \cI_{Z\vert X}(h+\theta f)\to\cO_X(h+\theta f)$ we deduce $\cM\cong\cO_X(h+\theta f)$. Thus $\psi\varphi$ is an isomorphism, hence $\varphi$ is also injective. Thus $\cI_{Z\vert X}(h+\theta f)=\cO_X(h+\theta f)$, i.e. $Z=\emptyset$ or, in other words, $k=0$.

Conversely, if $k=0$, then the extensions of $\cO_X(h+\theta f)$ with $\cO_X((\frak g+\theta)f)$ are parameterized by the sections of $H^1\big(\cO_X(-h+\frak g f)\big)$ up to scalars. Lemma \ref{lDerived} shows that such a space is zero, hence sequence \eqref{seqScrollCurve} always splits when $k=0$.
\end{proof}

In what follows we deal with the parameter space of the bundles obtained via Construction \ref{conScrollCurve}. Thus we need to compute $h^i\big(\cE\otimes\cE^\vee\big)$ when $\cE$ is indecomposable.

\begin{proposition}
\label{pExt}
Let $\cG$ be a vector bundle of rank ${n}\ge3$ on a smooth curve $B$ and set $X:=\bP(\cG)$. Assume that $\cO_X(h):=\cO_{\bP(\cG)}(1)$  is an ample and globally generated line bundle.

If $\cE$ is the bundle defined in Construction \ref{conScrollCurve} and $k\ge1$, then $\cE$ is simple, 
$$
h^1\big(\cE\otimes\cE^\vee\big)=(n-1)\deg(\frak g)+(n+1)(g-1)+2nk,
$$
and $h^i\big(\cE\otimes\cE^\vee\big)=0$ for $i\ge2$.
\end{proposition}
\begin{proof}
Assume that $k\ge1$. Thus the definition of $Z$, the cohomology of sequence \eqref{seqScrollCurve} tensored by $\cO_X(-h-\theta f)$ and Lemma \ref{lDerived} imply
\begin{equation}
\label{E-h-f}
h^i\big(\cE(-h-\theta f)\big)=h^i\big(\cI_{Z\vert X}\big)=\left\lbrace\begin{array}{ll} 
k-1\quad&\text{if $i=1$,}\\
0\quad&\text{if $i\ne1$.}
\end{array}\right.
\end{equation}
On the one hand $(h-\frak g f)fh^{n-2}=1$, hence $\cO_X(h-\frak g f)\not\cong\cO_X$. On the other hand $(h-\frak g f)h^{n-1}=0$: it follows that $h^0\big(\cO_X(h-\frak g f)\big)=0$, thanks to the Nakai--Moishezon criterion (see \cite[Theorem A.5.1]{Ha2}). Moreover, Lemma \ref{lDerived} yields $h^i\big(\cO_X(h-\frak g f)\big)=h^i\big(\cG(-\frak g)\big)$ for $i\ge1$, hence it vanishes for $i\ge2$. Equality \eqref{RRcurve} on the curve $B$ for $\cG(-\frak g)$ finally returns $h^1\big(\cO_X(h-\frak g f)\big)$. Thus 
\begin{equation}
\label{Xh-G}
h^i\big(\cO_X(h-\frak g f)\big)=\left\lbrace\begin{array}{ll} 
(n-1)\deg(\frak g)+n(g-1)\quad&\text{if $i=1$,}\\
0\quad&\text{if $i\ne1$.}
\end{array}\right.
\end{equation}
Since $Z:=\bigcup_{i=1}^kL_i$ where $L_i\cong\p{n-2}$ and $\cO_{L_i}\otimes\cO_X(h-\frak g f)\cong\cO_{\p{n-2}}(1)$ it follows that 
\begin{equation}
\label{Zh-G}
h^i\big(\cO_Z\otimes\cO_X(h-\frak g f)\big)=\left\lbrace\begin{array}{ll} 
(n-1)k\quad&\text{if $i=0$,}\\
0\quad&\text{if $i\ne0$.}
\end{array}\right.
\end{equation}
The cohomology of sequence \eqref{seqStandard} tensored by $\cO_X(h-\frak g f)$ and equalities \eqref{Xh-G} and \eqref{Zh-G} yield
\begin{equation}
\label{IZh-G}
h^i\big(\cI_{Z\vert X}(h-\frak g f)\big)=\left\lbrace\begin{array}{ll} 
\begin{aligned}
n&(g-1)\\
&+(n-1)(\deg(\frak g)+k)
\end{aligned}
\quad&\text{if $i=1$,}\\
\,0\quad&\text{if $i\ne1$.}
\end{array}\right.
\end{equation}
The cohomology of sequence \eqref{seqScrollCurve} tensored by $\cO_X(-(\frak g+\theta) f)$ and equalities \eqref{IZh-G} yield
\begin{equation}
\label{E-g-G}
h^i\big(\cE(-(\frak g+\theta) f)\big)=\left\lbrace\begin{array}{ll} 
\,1\quad&\text{if $i=0$,}\\
\begin{aligned}
n&(g-1)\\
&+(n-1)(\deg(\frak g)+k)+g
\end{aligned}\quad&\text{if $i=1$,}\\
\,0\quad&\text{if $i\ne0,1$.}
\end{array}\right.
\end{equation}
Notice that $\cO_Z\otimes\cE(-(\frak g+\theta)f)\cong\cN_{Z\vert X}$ thanks to equality \eqref{Normal}. The cohomology of sequence \eqref{seqStandard} tensored by $\cE(-(\frak g+\theta) f)$ and equalities \eqref{E-g-G} and \eqref{EZ-g-e} yield
\begin{equation}
\label{IZE-g-e}
h^i\big(\cI_{Z\vert X}\otimes\cE(-(\frak g+\theta) f)\big)=\left\lbrace\begin{array}{ll} 
\,x\quad&\text{if $i=0$,}\\
{\begin{aligned}
x&+(n-1)\deg(\frak g)\\
&+(n+1)(g-1)\\
&+(2n-1)k
\end{aligned}}\quad&\text{if $i=1$,}\\
\,0\quad&\text{if $i\ne0,1$,}
\end{array}\right.
\end{equation}
where $x\le 1$.
The cohomology of sequence \eqref{seqScrollCurve} tensored by $\cE^\vee$, Theorem \ref{tScrollCurve} and equalities \eqref{IZE-g-e} and \eqref{E-h-f} finally yield the statement.
\end{proof}

If $k=0$, then Construction \ref{conScrollCurve} leads to a decomposable bundle. Nevertheless also in this case there exist indecomposable $h$--instanton bundles.

\begin{proposition}
\label{pExt0}
Let $\cG$ be a vector bundle of rank ${n}\ge3$ on a smooth curve $B$ and set $X:=\bP(\cG)$. Assume that $\cO_X(h):=\cO_{\bP(\cG)}(1)$  is an ample and globally generated line bundle.

There exist bundles fitting into a non--split exact sequence
\begin{equation}
\label{seqScrollCurve0}
0\longrightarrow\cO_X(h+\theta f)\longrightarrow\cE\longrightarrow\cO_X((\frak g+\theta)f)\longrightarrow0
\end{equation}
Each such a bundle is a rank two orientable, ordinary, $\mu$--semistable, simple $h$--instanton bundle. Moreover, 
$$
h^1\big(\cE\otimes\cE^\vee\big)=(n-1)\deg(\frak g)+(n+1)(g-1)+g.
$$
and $h^i\big(\cE\otimes\cE^\vee\big)=0$ for $i\ge2$.
\end{proposition}
\begin{proof}
Thanks to Lemma \ref{lDerived} and equality \eqref{RRcurve} we have
$$
\chi(\cO_X(h-\frak g f))=\chi(\cG(-\frak g))=(1-n)\deg(\frak g)+n(1-g).
$$
Since $h^n=\deg(\frak g)$, it follows that $h^1\big(\cO_X(h-\frak g f)\big)>0$, hence sequence \eqref{seqScrollCurve0} can be assumed non--split.

Taking into account that $\mu_h(\cO_X(h+\theta f))=\mu_h(\cO_X((\frak g+\theta)f))$ and \cite[Lemma 4.2]{C--H2}, we deduce that $\cE$ is simple. Moreover, $\cE$ is an extension of Ulrich line bundles, hence it is an ordinary $\mu$--semistable, $h$--instanton bundle. Essentially the same argument used in the proof of Proposition \ref{pExt} also returns the values of $h^i\big(\cE\otimes\cE^\vee\big)$ for $i\ge1$.
\end{proof}

\begin{remark}
Construction \ref{conScrollCurve} and Proposition \ref{pExt0} yield the existence of rational maps from $\Lambda^k$ and $\vert \cO_X(h-\frak g f)\vert$ to the moduli space $\cS_X$ of simple bundles on $X$ if $k\ge1$ and $k=0$ respectively. It follows that all the indecomposable bundles obtained via the above constructions represent smooth points in one and the same component $\cS^0_X$. Thus such an $\cS^0_X$ is generically smooth and its dimension is
$$
(n-1)\deg(\frak g)+(n+1)(g-1)+2nk+\epsilon g
$$
where $\epsilon=1$ if $k=0$ and $0$ otherwise.
\end{remark}

\bigskip
\noindent
Vincenzo Antonelli,\\
Dipartimento di Scienze Matematiche, Politecnico di Torino,\\
c.so Duca degli Abruzzi 24,\\
10129 Torino, Italy\\
e-mail: {\tt vincenzo.antonelli@polito.it}

\bigskip
\noindent
Gianfranco Casnati,\\
Dipartimento di Scienze Matematiche, Politecnico di Torino,\\
c.so Duca degli Abruzzi 24,\\
10129 Torino, Italy\\
e-mail: {\tt gianfranco.casnati@polito.it}


\begin{thebibliography}{44}


\bibitem{A--O1} 
V. Ancona, G. Ottaviani: {\em Some applications of Beilinson's theorem to projective spaces and quadrics}. Forum Math. \textbf{3} (1991), 157--176.

\bibitem{A--C--G}
V. Antonelli, G. Casnati, O. Genc: \emph{Even and odd instanton bundles on Fano threefolds}. Available at  	arXiv:2105.00632 [math.AG], to appear in Asian J. Math.

\bibitem{A--M1}V. Antonelli, F. Malaspina: {\em Instanton bundles on the Segre threefold with Picard number three}. Math. Nachr. \textbf{293} (2020), 1026--1043. 

\bibitem{A--M2}
V. Antonelli, F. Malaspina: {\em $H$--instanton bundles on three-dimensional polarized projective varieties}. J. Algebra \textbf{598} (2022), 570--607.

\bibitem{A--H--M--PL}
M. Aprodu, S. Huh, F. Malaspina, J. Pons-Llopis: {\em Ulrich bundles on smooth projective varieties of minimal degree}. Proc. Amer. Math. Soc. \textbf{147} (2019), 5117--5129.

\bibitem{Ap--Kim}
M. Aprodu, Y. Kim: {\em Ulrich line bundles on Enriques surfaces with a polarization of degree four}. Ann. Univ. Ferrara Sez. VII Sci. Mat. \textbf{63} (2017), 9--23.

\bibitem{Ar}
  E. Arrondo: \emph{A home--made Hartshorne--Serre correspondence}. Comm. Alg.  \textbf{ 20} \rm (2007),  423--443.

\bibitem{A--S} 
 E.  Arrondo,  I.  Sols: {\em Classification of smooth congruences of low degree}. J. Reine Angew. Math. \textbf{393} (1989), 199--219.
 
\bibitem{A--W}
M.F. Atiyah, R.S. Ward: \emph{Instantons and algebraic geometry}. Comm. Math. Phys. \textbf{55} (1977), 117--124.

\bibitem{Bea}
A. Beauville: \emph{Determinantal hypersurfaces}.  Michigan Math. J. \textbf{48} \rm(2000), 39--64.

\bibitem{Bea6}
A. Beauville: \emph{An introduction to Ulrich bundles}. Eur. J. Math. \textbf{4} (2018), 26--36.

\bibitem{B--F1}
M.C. Brambilla, D. Faenzi: {\em Moduli spaces of rank--$2$ ACM bundles on prime Fano threefolds}. Michigan Math. J. \textbf{60} (2011), 113--148.

\bibitem{B--F2}
M.C. Brambilla, D. Faenzi: {\em Vector bundles on Fano threefolds of genus $7$ and Brill--Noether loci}. Internat. J. Math. \textbf{25} (2014), 1450023, 59 pp..

\bibitem{C--H2}
M. Casanellas, R. Hartshorne, F. Geiss, F.O. Schreyer: \emph{Stable Ulrich bundles}. Int. J. of Math. \textbf{23} \rm(2012), 1250083. 

\bibitem{Cs4}
G. Casnati: \emph{Special Ulrich bundles on non--special surfaces with $p_g=q=0$}. Int. J. Math. \textbf{28} (2017), 1750061. \emph{Erratum}. Int. J. of Math. \textbf{29} (2018), 1892001.

\bibitem{C--C--G--M}
G. Casnati, E. Coskun, O. Genc, F. Malaspina: \emph{Instanton bundles on the blow up of $\p3$ at a point}. Michigan Math. J. \textbf{70} (2021), 807--836.

\bibitem{C--F--M1}
G. Casnati, D. Faenzi, F. Malaspina: \emph{Rank two aCM bundles on the del Pezzo threefold with Picard number $3$}. J. Algebra \textbf{429} (2015), 413--446.

\bibitem{C--F--M2}
G. Casnati, D. Faenzi, F. Malaspina: \emph{Rank two aCM bundles on the del Pezzo fourfold of degree $6$ and its general hyperplane section}. J. Pure Appl. Algebra \textbf{22} (2018), 585--609.

\bibitem{C--N}
G. Casnati, R. Notari: \emph{Examples of rank two aCM bundles on smooth quartic surfaces in $\p3$}. Rend. Circ. Mat. Palermo (2) \textbf{66} (2017), 19--41.

\bibitem{C--K--M}
  E. Coskun, R.S. Kulkarni, Y. Mustopa: \emph{Pfaffian quartic surfaces and representations of Clifford algebras}. Doc. Math.  \textbf{17} \rm (2012), 1003--1028.

\bibitem{C--MR1}
L. Costa, R.M. Mir\'o--Roig: {\em Monads and regularity of vector bundles on projective varieties}. Mich. Math. J. \textbf{55} (2007),  417--436.

\bibitem{C--MR3}
L. Costa, R.M. Mir\'o--Roig: {\em Monads and instanton bundles on smooth hyperquadrics}. Math. Nachr. \textbf{282} (2009),  169--179.

\bibitem{C--MR7}
L. Costa, R.M. Mir\'o--Roig: {\em Ulrich bundles on Veronese surfaces}. In \lq Singularities, algebraic geometry, commutative algebra, and related topics\rq\ Festschrift for Antonio Campillo on the occasion of his 65th birthday (G-M. Greuel, L. Narv\'aez Macarro, S. Xamb\'o-Descamps eds.), Springer, Cham (2020), 375--381.

\bibitem{C--MR8}
L. Costa, R.M. Mir\'o--Roig: {\em Instanton bundles vs Ulrich bundles on projective spaces}. Beitr. Algebra. Geom. \textbf{62} (2021), 429--439.

\bibitem{C--MR--PL}
L. Costa, R.M. Mir\'o--Roig, J. Pons--Llopis: {\em An account of instanton bundles on hyperquadrics}. In \lq From classical to modern algebraic geometry\rq\ (G. Casnati, A. Conte, L. Gatto, L. Giacardi, M. Marchisio, A. Verra eds.), Trends Hist. Sci., Birkh\"auser/Springer (2016), 409--428.

\bibitem{Ea--No}
J.A. Eagon, D.G. Northcott: \emph{Ideals defined by matrices and a certain complex associated with them}. Proc. Roy. Soc. London Ser. A \textbf{269} (1962), 188--204. 

\bibitem{Ein--So}
L. Ein, I. Sols: \emph{Stable vector bundles on quadric hypersurfaces}. Nagoya Math. J. \textbf{96} (1984), 11--22. 

\bibitem{Ei}
D. Eisenbud: \emph{The geometry of syzygies. A second course in commutative algebra and algebraic geometry}. G.T.M. 229. Springer (2005).

\bibitem{Ei--Go}
D. Eisenbud, S. Goto: \emph{Linear free resolutions and minimal multiplicity}. J. Algebra \textbf{88} (1984), 89--133.

\bibitem{Ei--Ha1}
D. Eisenbud, J. Harris: \emph{On varieties of minimal degree (a centennial account)}. In Algebraic geometry, Bowdoin 1985), 3--13, Proceedings of symposia in pure mathematics, \textbf{46}, A.M.S., 1987.

\bibitem{Ei--Ha2}
D. Eisenbud, J. Harris: \emph{3264 and all that -- a second course in algebraic geometry}.  Cambridge U.P. (2016).

\bibitem{E--S--W}
D. Eisenbud, F.O. Schreyer, J. Weyman: \emph{Resultants and Chow forms via exterior syzigies}.  J. Amer. Math. Soc. \textbf{16} \rm(2003), 537--579.

\bibitem{El--Gr}
Ph. Ellia, L. Gruson: \emph{On the Buchsbaum index of rank two vector bundles on $\p3$}. Rend. Istit. Mat. Univ. Trieste \textbf{47} (2015), 65--79.

\bibitem{Fa1}
D. Faenzi: \emph{A remark on Pfaffian surfaces and aCM bundles}. In \lq Vector bundles and low codimensional subvarieties: state of the art and recent developments,\rq\ (G. Casnati, F. Catanese, R. Notari eds.), Quad. Mat., 21, Dept. Math., Seconda Univ. Napoli, Caserta (2007), 209--217.

\bibitem{Fa2}
D. Faenzi: \emph{Even and odd instanton bundles on Fano threefolds of Picard number one}. Manuscripta Math. \textbf{144} (2014), 199--239.

\bibitem{Flo}
G. Fl\o ystad: {\em Monads on projective spaces}. Comm. Algebra \textbf{28} (2000), 5503--5516. 

\bibitem{Ha2}
R. Hartshorne: {\em Algebraic geometry}. G.T.M. 52, Springer \rm (1977).

\bibitem{Ha3}
R. Hartshorne: {\em Coherent functors}. Adv. Math. \textbf{140} (1998), 44--94.

\bibitem{Hu--Le}
D. Huybrechts, M. Lehn: \emph {The geometry of moduli spaces of sheaves. Second edition}. Cambridge Mathematical Library, Cambridge U.P. \rm (2010).

\bibitem{Ja}
M.B. Jardim: {\em Instanton sheaves on complex projective spaces}. Collect. Math. \textbf{57} (2006), 69--91.

\bibitem{J--MR}
M.B. Jardim, R.M. Mir\'o--Roig: {\em On the semistability of instanton sheaves over certain projective varieties}. Comm. Alg. \textbf{36} (2008), 288--298.

\bibitem{J--VM}
M.B. Jardim, R. Vidal Martins: {\em Linear and Steiner bundles on projective varieties}. Comm. Alg. \textbf{38} (2010), 2249--2270.

\bibitem{Jou}
J.P. Jouanolou: \emph{Th\'eor\ga emes de Bertini et applications}. Progress in Mathematics 42, Birkh\"auser (1983).

\bibitem{Ka--Ko}
Y. Kachi, J. Koll\'ar: \emph{Characterizations of $\p N$ in arbitrary characteristic}. Asian J. Math. \textbf{4} (2000), 115--21. 

\bibitem{K--M--S}
R.S.  Kulkarni, Y. Mustopa, I. Shipman: {\em Vector bundles whose restriction to a linear section is Ulrich}. Math. Z. \textbf{287} (2017), 1307--1326.

\bibitem{Kuz}
A. Kuznetsov: {\em Instanton bundles on Fano threefolds}. Cent. Eur. J. Math. \textbf{10} (2012), 1198--1231.

\bibitem{Mal1}
F. Malaspina: {\em Monads and Vector Bundles on Quadrics}.  Adv. Geom. \textbf{9} (2009), 137--152.

\bibitem{M--M--PL}
F. Malaspina, S. Marchesi, J. Pons-Llopis: {\em Instanton bundles on the flag variety $F(0,1,2)$}.  Ann. Sc. Norm. Super. Pisa Cl. Sci. (5) \textbf{20} (2020), 1469--1505.

\bibitem{MK--P--R}
N. Mohan Kumar, C. Peterson, A.P. Rao: {\em Monads on projective spaces}. Manuscripta Math. \textbf{112} (2003), 183--189.

\bibitem{Ok--Sp}
  C. Okonek, H. Spindler: {\em Mathematical instanton bundles on $\p{2n+1}$}. J. Reine Angew. Math. \textbf{364} (1985) 35--50.

\bibitem{O--S--S}
  C. Okonek, M. Schneider, H. Spindler: {\em Vector bundles on complex projective spaces}. Progress in Mathematics 3, Birkh\"auser \rm(1980).

\bibitem{Ott2}
G. Ottaviani: {\em Spinor bundles on quadrics}. Trans. Am. Math. Soc. \textbf{307} (1988), 301--316.

\bibitem{Ott3}
G. Ottaviani: {\em On Cayley bundles on the five--dimensional quadric}.  Boll. Un. Mat. Ital. A \textbf{4} (1990), 87--100.

\bibitem{Ot--Sz}
G. Ottaviani, M. Szurek: {\em On moduli of stable $2$-bundles with small Chern classes on $Q_3$. With an appendix by Nicolae Manolache}. Ann. Mat. Pura Appl. \textbf{167} (1994), 191--241.

\bibitem{Rah1}
O. Rahavandrainy: \emph{R\'esolution des fibr\'es instantons g\'en\'eraux}. C. R. Acad. Sci. Paris S\'er. I Math. \textbf{325} (1997), 189--192.

\bibitem{Rah2}
O. Rahavandrainy: \emph{R\'esolution des fibr\'es g\'en\'eraux stables de rang $2$ sur $\p 3$ de classes de Chern $c_1=-1$, $c_2=2p\ge6$. I.} Ann. Fac. Sci. Toulouse Math. (6) \textbf{19} (2010), 231--267.

\bibitem{So--VV}
A. Sommese, A. Van de Ven: \emph{On the adjunction mapping}. Math. Ann. \textbf{278} \rm(1987), 593--603.

\end{thebibliography}
\end{document}